\documentclass[12pt]{amsart}

\usepackage{epsf,amssymb,overpic,amscd,psfrag,stmaryrd,erewhon,bbold,mathrsfs,tikz-cd}
\usepackage[all]{xy}

\usepackage{epsfig}
\usepackage{epstopdf}

\newtheorem{thm}{Theorem}[section]
\newtheorem{prop}[thm]{Proposition}
\newtheorem{lemma}[thm]{Lemma}
\newtheorem{cor}[thm]{Corollary}
\newtheorem{conj}[thm]{Conjecture}

\newtheorem{claim}[thm]{Claim}

\numberwithin{equation}{subsection}
\numberwithin{thm}{subsection}

\theoremstyle{definition}
\newtheorem{defn}[thm]{Definition}

\newtheorem{notation}[thm]{Notation}

\theoremstyle{remark}
\newtheorem{q}[thm]{Question}
\newtheorem{rmk}[thm]{Remark}
\newtheorem{example}[thm]{Example}

\DeclareMathAlphabet{\mathpzc}{OT1}{pzc}{m}{it}
\newcommand{\C}{\mathbb{C}}

\newcommand{\R}{\mathbb{R}}
\newcommand{\Z}{\mathbb{Z}}

\newcommand{\bdry}{\partial}
\newcommand{\s}{\vskip.1in}
\newcommand{\n}{\noindent}

\newcommand{\F}{\mathbb{F}}
\newcommand{\E}{\mathcal{E}}

\newcommand{\be}{\begin{enumerate}}
	\newcommand{\ee}{\end{enumerate}}
\newcommand{\op}{\operatorname}
\newcommand{\bs}{\boldsymbol}

\newcommand{\wh}{\widehat}

\newcommand{\rsnk}{R(S,n,{\bf a};k)}
\newcommand{\rsno}{R(S,n,{\bf a};1)}
\newcommand{\mb}{\mathbf{1}}

\newcommand{\Hom}{\op{Hom}}
\newcommand{\End}{\op{End}}
\newcommand{\rk}{R_k}
\newcommand{\rkkh}{R_k^{\op{nil}}}
\newcommand{\nck}{\op{NC}_k}
\newcommand{\nhk}{\op{NH}_k}
\newcommand{\smk}{\op{Sym}_k}
\newcommand{\polk}{\op{Pol}_k}

\newcommand{\rnk}{R(n;k)}
\newcommand{\rnkkh}{R(n;k)^{\op{nil}}}
\newcommand{\rnkp}{R(n;k+1)^{\op{nil}}}
\newcommand{\rnkm}{R(n;k-1)^{\op{nil}}}
\newcommand{\rnke}{R(n+1;k+1)^{\op{nil}}}
\newcommand{\rnkf}{R(n+1;k)^{\op{nil}}}
\newcommand{\cnk}{\mathcal{C}_{n,k}}
\newcommand{\cnkp}{\mathcal{C}_{n,k+1}}
\newcommand{\cnkm}{\mathcal{C}_{n,k-1}}
\newcommand{\zq}{\mathbb{Z}[q, q^{-1}]}
\newcommand{\ech}{\check{\mathcal{E}}}
\newcommand{\fch}{\check{\mathcal{F}}}
\newcommand{\lk}{\Lambda_k}
\newcommand{\lkkh}{\Lambda_k^{\op{nil}}}

\newcommand{\slk}{H_k \tilde{\otimes} \Lambda_k}

\newcommand{\lm}{\Lambda}
\newcommand{\ai}{A_{\infty}}
\newcommand{\mfn}{\mathbf{n}}
\newcommand{\pnk}{\mathcal{P}_k(\mathbf{n})}
\newcommand{\pnke}{\mathcal{P}_{k+1}(\mathbf{n+1})}
\newcommand{\pnkf}{\mathcal{P}_{k}(\mathbf{n+1})}
\newcommand{\pnkp}{\mathcal{P}_{k+1}(\mathbf{n})}
\newcommand{\pnkm}{\mathcal{P}_{k-1}(\mathbf{n})}

\newcommand{\po}{\mathcal{PO}}
\newcommand{\inv}{\op{Inv}}
\newcommand{\ti}{\tilde}

\newcommand{\ra}{\rightarrow}
\newcommand{\xra}{\xrightarrow}

\newcommand{\es}{\varnothing}
\newcommand{\Ker}{\op{Ker}}
\newcommand{\im}{\op{Im}}
\newcommand{\ot}{\otimes}
\newcommand{\cal}{\mathcal}
\newcommand{\upl}{\mathcal{U}^{+}}
\newcommand{\dgupl}{\mathcal{DG}(\mathcal{U}^{+})}
\newcommand{\lle}{L_{\mathcal{E}}}
\newcommand{\lek}{L_{\mathcal{E}}^{\otimes k}}
\newcommand{\id}{\op{id}}

\newcommand{\circled}[1]{\raisebox{.5pt}{\textcircled{\raisebox{-.9pt} {#1}}}}

\begin{document}
	
	\title[Khovanov homology for transverse links in fibered $3$-manifolds]
	{Towards a symplectic Khovanov homology for links in fibered $3$-manifolds}
	
	\author{Vincent Colin}
	\address{Nantes Universit\'e, 44322 Nantes, France}
	\email{vincent.colin@univ-nantes.fr}
	
	\author{Ko Honda}
	\address{University of California, Los Angeles, Los Angeles, CA 90095}
	\email{honda@math.ucla.edu} \urladdr{http://www.math.ucla.edu/\char126 honda}
	
	\author{Yin Tian}
	\address{School of Mathematical Sciences, Beijing Normal University; 
Laboratory of Mathematics and Complex Systems, Ministry of Education, Beijing 100875, China}
	\email{yintian@bnu.edu.cn}

	\date{\today}

	
	
	\thanks{VC supported by ERC Geodycon and ANR Quantact and COSY.  KH supported by NSF Grants DMS-1406564, DMS-1549147, DMS-2003483, and DMS-2505876. YT supported by NSFC 11601256, 11971256 and 12471064.}

	\begin{abstract}
		The goal of this paper is twofold: (i) define a symplectic Khovanov type homology for a transverse link in a fibered closed $3$-manifold $M$ (with an auxiliary choice of a homotopy class of loops that intersect each fiber once) and (ii) give conjectural combinatorial dga descriptions of surface categories that appear in (i).  These dgas are higher-dimensional analogs of the strands algebras in bordered Heegaard Floer homology, due to Lipshitz-Ozsv\'ath-Thurston \cite{LOT}.
	\end{abstract}

\maketitle

\setcounter{tocdepth}{1}
\tableofcontents

\section{Introduction} \label{section: intro}

Khovanov homology \cite{Kh1} associates a bigraded $\Z$-module $Kh(\ell)$ to a link $\ell$ in $S^3$ through somewhat mysterious combinatorial formulas expressed on a planar diagram of $\ell$. In 2006, Seidel and Smith \cite{SS} proposed a more geometric reformulation of Khovanov homology, called {\em symplectic Khovanov homology} $Kh_{\op{symp}}(\ell)$, as a Lagrangian Floer homology.

Let us recall the reinterpretation of $Kh_{\op{symp}}(\ell)$ due to Manolescu~\cite{Ma}: If $\ell$ is a closure of an $n$-strand braid $\mathbb{b}$, then let $W_{2n-1}$ be the $4$-dimensional Milnor fiber of the $A_{2n-1}$ singularity. It admits a symplectic Lefschetz fibration $\pi : W_{2n-1} \to D^2$ over the closed $2$-disk $D^2$ with $2n$ singular points (and also $2n$ critical values) and whose regular fiber is an annulus $A=S^1\times[-1,1]\subset T^*S^1$. The braid $\mathbb{b}$ induces a symplectomorphism $\Phi : W_{2n-1}\stackrel\sim\to W_{2n-1}$ by permuting the fibers and hence a symplectomorphism $\overline{\Phi}$ of an open set $(Y_n , \Omega)$ of the Hilbert scheme $\op{Hilb}^n(W_{2n-1})$ of $n$ points in $W_{2n-1}$. There are disjoint Lagrangian $2$-spheres $\alpha_1,\dots,\alpha_n\subset W_{2n-1}$ lying over $n$ disjoint arcs $\gamma_1,\dots,\gamma_n$ in $D^2$ which connect the critical values of $\pi$ in pairs.  Writing $\mathcal{L}_{\bs\gamma}:= \alpha_1 \times \dots \times \alpha_n \subset (Y_n , \Omega )$, where ${\bs \gamma} =\{\gamma_1 ,\dots ,\gamma_n \}$, the symplectic Khovanov homology $Kh_{\op{symp}}(\ell)$ of $\ell$ is defined as the Lagrangian Floer homology $HF(\mathcal{L}_{\bs \gamma} ,\overline{\Phi} (\mathcal{L}_{\bs\gamma} ))$. It was shown to be independent of choices by Seidel-Smith~\cite{SS}, isomorphic to $Kh(\ell)$ by Abouzaid-Smith~\cite{AS}, and given a cylindrical reformulation by Mak-Smith~\cite{MS}. 

The goal of this paper is twofold: (1) define a symplectic Khovanov type homology for a transverse link in a fibered $3$-manifold $M$ (together with an auxiliary choice of a homotopy class of loops that intersect each fiber once) and (2) give conjectural combinatorial descriptions of surface categories that appear in (1).   

\begin{rmk}
    The reason we say symplectic ``Khovanov type'' homology is that our invariant generalizes the variant $Kh^\sharp$ of Khovanov homology given in \cite[Section 9]{CHT}, not $Kh_{\op{symp}}$; see \cite[Section 9.6]{CHT} for the spectral sequence for $Kh^\sharp$ whose $E^1$-term is $Kh_{\op{symp}}$.
\end{rmk}

This paper should be viewed as a sequel to \cite{CHT} and we will freely refer to the construction of the higher-dimensional Heegaard Floer (HDHF) homology groups from \cite{CHT}.  The HDHF groups will be defined over a field $\F$.

Let $(W^{2\nu},\beta,\phi)$ be a $2\nu$-dimensional Weinstein domain, where $\beta$ is a Liouville form and $\phi:W\to\R$ is a compatible Morse function, and let $h$ be an element of $\op{Symp}(W,\bdry W,d\beta)$, the group of symplectomorphisms of $(W,d\beta)$ that restrict to the identity on $\bdry W$. 

The starting point of this paper is Section~\ref{section: wrapped version}, where we define the wrapped HDHF $\ai$-bi\-module $B^\star(W,\beta,\phi;h)$ over a wrapped HDHF $\ai$-algebra $R^\star(W,\beta,\phi)$ and prove its Morita invariance under handleslides (Theorem~\ref{thm: invariance under handleslides}). Here $\star$ refers to a fixed wrapping scheme, e.g., $\star=f$ for full wrapping and $\star=p$ for partial wrapping. 

The rest of the paper is organized into Parts 1 and 2 corresponding to (1) and (2) above and summarized in Sections~\ref{subsection: Khovanov} and \ref{subsection: combinatorial} below.

\subsection{Symplectic Khovanov type homology of a transverse link in a fibered $3$-manifold}\label{subsection: Khovanov}

Let $S$ be a {\em bordered surface}, i.e., a compact connected oriented surface with nonempty boundary, let $\op{Diff}^+(S,\bdry S)$ be the set of orien\-tation-preserving diffeomorphisms of $S$ that restrict to the identity on $\bdry S$, and let ${\bf x}=\{x_1,\dots,x_n\}\subset \op{int}S$ be an $n$-element subset called the set of {\em marked points}.  We also choose a finite subset $\tau\subset \bdry S$, called a {\em stop}. Then $\star$ will be partial wrapping with respect to $\tau$ if $\tau\not=\varnothing$ and full wrapping if $\tau=\varnothing$. Let ${\bf a}= \{a_1,\dots, a_{n+s}\}$ be a {\em parametrization of $(S,{\bf x},\tau)$,} described a few paragraphs below (also see Definition~\ref{defn: parametrization}).

The goal of Part 1 is to give definitions of 
\be
\item[(A)] surface categories $\overline {\mathcal{R}}^\star (S,n,{\bf a})$ motivated by symplectic Khovanov homology and 
\item[(B)] a symplectic Khovanov type invariant of a transverse link in a fibered $3$-manifold $M$ together with an auxiliary choice of a homotopy class of loops that intersect each fiber once.
\ee  
By (B) we mean an invariant of a transverse link $\ell\subset M$ (i.e., a link that is transverse to all the fibers of $M$), up to an isotopy through transverse links. 

Consider the mapping torus
$$N_{(S,h)}=(S\times[0,1])/ (x,1)\sim (h(x),0) \quad \forall x\in S$$
of $h\in \op{Diff}^+(S,\bdry S)$. Let $\ell\subset \op{int}(N_{(S,h)})$ be a transverse link and suppose it intersects the fiber $S\times\{1\}$ at $n$ points along ${\bf x}=\{x_1,\dots,x_n\}$.  Then we cut $\ell$ along $S\times \{1\}$ to obtain a braid $\mathbb{b}$ in $S\times  [0,1]$.  We may assume that $h$ takes ${\bf x}$ to itself, i.e., is an element of $\op{Diff}^+(S,\bdry S,{\bf x})$, and that the endpoints of $\mathbb{b}$ are ${\bf x}\times \{0,1\}$.

We then consider a symplectic Lefschetz fibration $\pi :W_{S,n} \to S$ with base $S$ instead of $D^2$, regular fiber $A=S^1\times[-1,1]$, and $n$ critical values $x_1,\dots, x_n$. The family $\{a_1,\dots, a_n\}$ of pairwise disjoint half-arcs in $S$ joining the critical values of $\pi$ to $\partial S\setminus \tau$ is augmented by a basis $\{a_{n+1} ,\dots ,a_{n+s}\}$ of $s$ properly embedded pairwise disjoint arcs in $S$ with endpoints on $\bdry S\setminus \tau$ such that:
\be
\item the elements of the {\em parametrization} ${\bf a}:= \{a_1,\dots, a_{n+s}\}$ of $(S,{\bf x}, \tau)$ are pairwise disjoint; and
\item if $\tau=\varnothing$ then $S\setminus\cup_{i=n+1}^{n+s}a_i$ is a single polygon; otherwise each component of $S\setminus\cup_{i=n+1}^{n+s}a_i$ is a polygon containing a single point of $\tau$.
\ee 
(For example, when $\bdry S$ is connected and $\tau=\varnothing$ or $\tau$ is a $1$-element set, then $s=2g(S)$.)  We consider the collection $\{L_1,\dots, L_{n+s}\}$ of Lagrangian disks and annuli of $W_{S,n}$, where $L_i$, $i\leq n$, are thimbles over $a_i$  and $L_i$, $i\geq n+1$, are annuli over $a_i$ which are obtained by parallel transporting the $0$-section of the fiber using the symplectic connection.\footnote{The Lagrangians we actually use are the completions $\widehat{L}_i$ of $L_i$ that lie over the cylindrical completions $\widehat{a}_i$ of $a_i$ in $S\cup ([0,\infty)\times \bdry S)$, but for the moment we do distinguish between $\widehat{L}_i$ and $L_i$.}

For each nonnegative integer $k$, consider a partition ${\bf m}= (m_1,\dots, m_{n+s})$ of $k$ subject to the condition that $m_i=0$ or $1$ when $i\leq n$.  In other words, we allow parallel copies of arcs but at most one copy of half-arcs. Compare this to the symplectic Khovanov setting of \cite{SS,Ma}, where $S$ is a disk and there is only one copy of each half-arc.  The main reason for only allowing one copy of each half-arc is that two copies would intersect at their interior endpoint, which we want to avoid in our collections of Lagrangian submanifolds.

Let $a_{\bf m}$ be the union of $k$ disjoint arcs/half-arcs consisting of $m_i$ disjoint copies of $a_i$ and let $L_{\bf m}$ be the family of Lagrangian disks and cylinders sitting above $a_{\bf m}$.

We then define the $\ai$-category $\mathcal{R}^\star(S,n,{\bf a})$, whose objects are $L_{\bf m}$ and whose morphisms are $CF^\star(L_{\bf m}, L_{\bf m'})$, the $\star$-wrapped HDHF chain complex with coefficients in $\F[\mathcal{A}]\llbracket \hbar, \hbar^{-1}]$ (this means power series in $\hbar$ and polynomial in $\hbar^{-1}$; $\mathcal{A}$ takes care of relative homology classes of degree $2$, provided a complete set of capping surfaces is given, see \cite[Sections 3.5.3, 3.8 and 9.2]{CHT})\color{black}, where the stop $\tau$ determines how far the arcs $a_i$ (resp.\ Lagrangians $L_i$) can be wrapped around $\bdry S$ (resp.\ $\bdry W_{S,n}$).  Viewed as an $\ai$-algebra,
\begin{equation} \label{eq surface alg}
R^\star(S,n,{\bf a}):= \oplus_{k\geq 0, |{\bf m}|=|{\bf m'}|=k} CF^\star (L_{\bf m}, L_{\bf m'}).
\end{equation}

Following Seidel \cite{Se3}, a triangulated envelope of an $A_\infty$-category $\mathcal{A}$  is a pair $(\mathcal{B},\mathcal{F})$ where $\mathcal{B}$ is a triangulated $A_\infty$-category and  $\mathcal{F} : \mathcal{A}\to \mathcal{B}$ is a cohomologically full and faithful functor such that the objects in the image of $\mathcal{F}$ are generators of $\mathcal{B}$.
In our case, we realize a triangulated envelope of $\overline {\mathcal{R}}^\star (S,n,{\bf a})$ of $\mathcal{R}^\star (S,n,{\bf a})$ as the $A_\infty$-category of its twisted complexes.

\begin{thm}\label{thm: invariance}
The triangulated envelope $\overline {\mathcal{R}}^\star (S,n,{\bf a})$ of $\mathcal{R}^\star (S,n,{\bf a})$ does not depend on the choice of ${\bf a}$, when $\star$ is a partial wrapping with respect to $\tau$ or a full wrapping.
\end{thm}

If $\ell$ is the closure of a braid $\mathbb{b}$  by $h$, then it induces an $R^\star(S,n,{\bf a})$-bimodule
$$B^\star(S,n,{\bf a};\mathbb{b},h):=\oplus_{k\geq 0,|{\bf m}|=|{\bf m'}|=k} CF^\star( h_{\mathbb{b}}\circ h(L_{\bf m}),L_{\bf m'}),$$
where $h_{\mathbb{b}}$ is isotopic to the identity in $\op{Diff}^+(S,\bdry S)$ and traces out $\mathbb{b}$. We then consider the Hochschild homology $HH^\star (S,n,{\bf a};\mathbb{b},h)$ of $B^\star(S,n,{\bf a};\mathbb{b},h)$.

In view of the following theorem, we can view $HH^\star (S,n,{\bf a};\mathbb{b},h)$ as a symplectic Khovanov type transverse link invariant of $\ell$ with respect to the fibration $N_{(S,h)}\to S^1$.
%or $HH_{\op{symp}}^f (S,h,\ell)$, called the {\it (partially, resp.\ fully) wrapped symplectic Khovanov homology} of $(S,h,\ell)$, is an invariant  of the isotopy class of $\ell$ as a braid in $M_{(S,h)}$. In the partially wrapped case, it corresponds to a sutured type Khovanov homology.

\begin{thm}\label{thm: wrapped}
The Hochschild homology groups
$$HH^p(S,n,{\bf a};\mathbb{b},h) \quad \mbox{and} \quad HH^f(S,n,{\bf a};\mathbb{b},h)$$
depend only on $(S,h)$ and on the transverse isotopy class of $\ell$ in $N_{(S,h)}$.
\end{thm}

The suspension of the stops $\tau\subset \bdry S$ gives a collection of curves $\Gamma\subset \partial N_{(S,h)}$; if  $\tau$ divides $\bdry S$ into alternating positive and negative regions, then $\Gamma$ is a set of {\em sutures}. % that endow $N_{(S,h)}$ with the structure of a {\it sutured} manifold. 
The associated partially wrapped Hochschild homology $HH^p(S,n,{\bf a};\mathbb{b},h)$ can therefore be interpreted as some version of sutured Khovanov homology of $\ell$ in the {\em sutured manifold} $(N_{(S,h)},\Gamma)$.

\begin{rmk}
In the case of $S=D^2$, $|\tau|=1$, our Hochschild homology is expected to be related to annular Khovanov homology (see \cite{APS}). On the algebraic side, the connection between Hochschild homology and annular Khovanov homology was first described in \cite{AuGrWe} for a special case, and fully established in \cite{BPW} via the trace functor. 
\end{rmk}

Next we consider the case when the ambient manifold is a closed fibered $3$-manifold.  
For simplicity assume $\partial S$ is connected and denote by $\overline{S} =S\cup D$ the capping off of $S$ by a disk $D$ and by $\overline{h}$ the extension of $h$ to $\overline{S}$ by the identity on $D$. The mapping torus of $(\overline{S},\overline{h})$ is the closed fibered $3$-manifold $N_{(\overline{S},\overline{h})}$.  Also let the knot $K_0$ be the mapping torus of the center $0$ of the disk $D\subset \overline{S}$.

In view of the following theorem, we refer to $HH^f(S,n,{\bf a};\mathbb{b},h)$ as a symplectic Khovanov-type invariant of the transverse link $\ell$ with respect to the fibration $\pi: N_{(\overline{S},\overline{h})}\to S^1$ and the transverse isotopy class $[K_0]$ of $K_0$ (i.e., the connected component of loops that intersect each fiber transversely and exactly once, and contains $K_0$):

\begin{thm}\label{thm: braid in fibration}
    The Hochschild homology $HH^f(S,n, {\bf a}; \mathbb{b},h)$ is an invariant of the transverse link $\ell$ with respect to the pair consisting of a fibration $\pi: N_{(\overline{S},\overline{h})}\to S^1$ and the transverse isotopy class $[K_0]$. When $\ell$ is empty, this is an invariant of the pair $(\pi: N_{(\overline{S},\overline{h})}\to S^1,[K_0])$.
\end{thm}

Let us fix a reference transverse isotopy class $[K_0]$, where $K_0$ is the center of the solid torus $D\times[0,1]/\sim$ as above. Given any other transverse isotopy class $[K]$, we cut $K$ along $\overline{S}\times\{1\}$ and view it as a path $K^\dagger:[0,1]\to\overline{S}\times [0,1]$ from $\overline{S}\times\{0\}$ to $\overline{S}\times\{1\}$. If $\zeta$ is a path in $\overline{S}\times \{0\}$ from $0$ to $K^\dagger(0)$, then $h_*^{-1}(\zeta)$ is a path in $\overline{S}\times \{1\}$ from $0$ to $K^\dagger(1)$.  Writing $K^\ddagger$ for the projection of $K^\dagger$ to $\overline{S}$, we obtain a loop $\zeta K^\ddagger h_*^{-1}(\zeta^{-1})$ based at $0$. Since the loop $\zeta K^\ddagger h_*^{-1}(\zeta^{-1})$ depends on the choice of connecting path $\zeta$, the set of transverse isotopy classes $[K]$ is in bijection with 
$$\pi^{h_*}_1(\overline{S},0):=\pi_1(\overline{S},0)/ \langle a b a^{-1} h_*^{-1}(b^{-1})~|~ a,b\in \pi_1(\overline{S},0)\rangle.$$

Hence Theorem~\ref{thm: braid in fibration} can be rephrased as follows:

\begin{thm} \label{cor: invariant}
    The Hochschild homology $HH^f(S,n, {\bf a}; \mathbb{b},h)$ is an invariant of the transverse link $\ell$ with respect to the pair $(\pi: N_{(\overline{S},\overline{h})}\to S^1,[a])$, where $[a]\in\pi^{h_*}_1(\overline{S},0)$. 
\end{thm}

\begin{conj}
    The transverse link invariant is a genuine link invariant of a fibration together with $[K]$.
\end{conj}

To show this we need to introduce cup- and cap-bimodules and show bimodule equivalences corresponding to standard moves.  This will be left for future work.

\begin{rmk} 
Theorem~\ref{cor: invariant} extends to the case where $\partial S$ is not connected and we cap all the boundary components of $S$ by disks. In this situation, $K$ is a link corresponding to the mapping torus of the centers of the disks.
\end{rmk}

\begin{rmk}
These constructions seem to fit some of the expectations of Witten~\cite[p.11, 2nd paragraph]{Wi3}, namely there should be a Khovanov-type $3$-manifold invariant for a $3$-mani\-fold $M$ together with a representation $\rho: \pi_1(\R \times M)\to SL(2,\C)$. In the case of a fibration $\pi:M\to S^1$ with monodromy $h:S\stackrel\sim\to S$, the group $\pi_1(\R\times M)\simeq \pi_1(M)$ is an HNN-extension of $\pi_1 (S)$ by $h_*$. Two different choices of transverse knots $K_0$ and $K$ that intersect each fiber exactly once give rise to an element $[a]\in \pi^{h_*}_1(\overline{S},0)$ as above. 
% give an element $[\ell \ell_0^{-1}]$ in $\pi_1(S) \subset PSL_2(\R)$. Having a fixed reference knot $\ell_0$, 
We thus have a collection of representations associated with every $[K]$ given by the composition
$$\pi_1 (\R\times M)\stackrel{\pi_*} \to \pi_1 (S^1)\simeq \langle[ K]\rangle \to \langle a \rangle \subset PSL_2(\R)=SL_2(\R)/\{\pm \op{id}\},$$
where $a$ is a representative of $[a]$.
It can be lifted to $\rho :\pi_1 (\R\times M) \to SL_2(\R) \subset SL_2(\C)$ by choosing a lift of $a$ in $SL_2(\R)$ and extending the map to the subgroup it generates.
%Moreover, there is a transitive action of $\pi_1 (S)$ on the set of isotopy classes of curves that intersect each fiber of the fibration once transversely. The action on such a transverse curve $\ell_0$ of the subset of $\pi_1(S)$ given by the orbit of an element $g$ under the action of $h_*$ does give the same isotopy class of transverse curves.
%Also, the homology we are producing is invariant under conjugation by $h$.
\end{rmk}

\subsection{Combinatorial description} \label{subsection: combinatorial}

The Hochschild homology invariants of links constructed in Part 1 have the defect of being difficult to compute. In Part 2 of this paper we take the first steps towards computing the above invariants. In particular, we present conjectural combinatorial descriptions of the surface categories of that appear in Part 1.  We present some justification for the combinatorial descriptions, but at this moment the equivalence with the geometric description from Part 1 should be viewed as being {\em entirely conjectural.}  It relies in particular on the fact that one could get genus bounds on holomorphic curves and set $\hbar=1$.

More specifically, we give (conjectural) combinatorial descriptions of the partially wrapped $\ai$-algebras $R^p(S,n,{\bf a})$ in (\ref{eq surface alg}) in terms of dgas $R(S,n,{\bf a})$. Our dgas are higher-dimensional analogs of the {\em strands algebras} in bordered Heegaard Floer homology, due to Lipshitz-Ozsv\'ath-Thurston \cite{LOT}.

\subsubsection{The local model} 

Khovanov used the local features of the strands algebra to construct a simplified version of the strands algebra called the {\em LOT dga} \cite{Kh2}.  We first define a family of dgas $\rk$ for $k \ge 0$, which are higher-dimensional analogs of these LOT dgas; see Definition \ref{def Rk}. The corresponding surface is a disk with no singular points and with a stop consisting of two points on the boundary, i.e., $S = D^2, n = 0, |\tau| = 2$, and $W_{S,n}=D^2\times A=D^2\times S^1\times[-1,1]$. Let $a\subset D^2$ be an arc such that each component of $D^2\setminus a$ nontrivially intersects $\tau$; see Figure \ref{fig algn7}. Let $L\subset W_{S,n}$ be the Lagrangian $a\times S^1\times\{0\}$. Then  $\rk$ is the conjectural algebraic description of $\End(L_k)$, where $L_k$ is the collection of Lagrangian cylinders sitting above $k$ parallel copies of $a$.

\begin{figure}[ht]
\begin{overpic}
[scale=0.5]{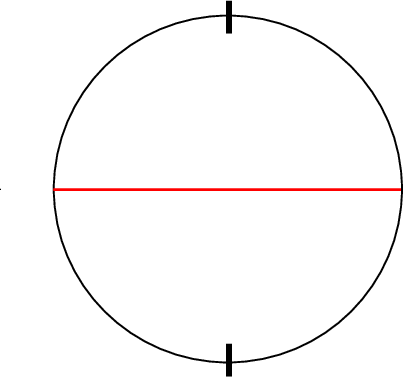}
\end{overpic}
\caption{The red arc is $a$, the projection of the Lagrangian to the base $D^2$, and the two marks on the boundary form the stop.}
\label{fig algn7}
\end{figure}

We will discuss the algebraic properties of the dga $\rk$ in Section \ref{sec local}. In particular, we prove that its cohomology algebra $H(\rk)$ is isomorphic to an exterior algebra and conjecture that $\rk$ is formal. 

There is an additional filtration on $\rk$, and the associated $q$-graded dga is denoted by $\rkkh$. 
It turns out that $\rkkh$ is naturally connected to the nilHecke algebra $\nhk$ and categorified quantum $\mathfrak{sl}_2$.
Recall that Lauda \cite{La} gave a diagrammatic categorification of $U_q\mathfrak{sl}_2$ and in particular there is an additive monoidal category $\upl$ which categorifies the positive half of $U_q\mathfrak{sl}_2$. 
It has objects $\E^{\otimes k}$ for $k \ge 0$ such that $\End(\E^{\otimes k})=\nhk$. Let $\dgupl$ be the dg category of complexes over $\upl$; it has a natural monoidal structure induced from $\upl$. 
Define a two-term complex:
\begin{equation} \label{def le1}
\lle:=(\E \xra{x_1} \E),
\end{equation}   
where the two copies of $\E$ are in cohomological degrees $0$ and $1$, and the differential is given by $x_1 \in \End(\E)$. 
Let $\End(\lek)$ denote the endomorphism algebra of $\lek$ in $\dgupl$ which is a $q$-graded dga. 
There is a natural action of $\rkkh$ on $\lek$ via an inclusion of $q$-graded dgas 
$$\rho_k: \rkkh \hookrightarrow \End(\lek).$$ 
Let $\F$ denote the ground field. Here is one of our main results.

\begin{thm}\label{thm rkkh dual}
Assuming $\op{char}(\F) \neq 2$, the inclusion $\rho_k: \rkkh \hookrightarrow \End(\lek)$ is a quasi-iso\-morphism of $q$-graded dgas. 
Moreover, $\rkkh$ is formal, i.e., it is quasi-isomorphic to $H(\rkkh)$ which is an exterior algebra on $k$ generators.  
\end{thm}

The object $\lek$ can be viewed as a ``Koszul dual'' of $\E^{\otimes k}$ in the categorified quantum $\mathfrak{sl}_2$. 
Together with the conjectural isomorphism $\rkkh \cong \End(L_k)$, Theorem \ref{thm rkkh dual} indicates that the Lagrangian $L_k$ is a geometric representative of the Koszul dual of $\E^{\otimes k}$. 

%{\cg Yin: The Lagrangian $L$ is $a \times (S^1 \times \{0\})$, and its Koszul dual should be $\check{L}=a \times (\{pt\} \times [-1,1])$ with proper perturbation condition. Is it possible to relate $\End(\check{L}^{\otimes k})$ to $\nhk$? 
%Meanwhile, for the $A_n$-singularity, Lagrangian thimbles as the analog of $L$ seem better. Are there Lagrangians as the analog of $\check{L}$ in the $A_n$-case?}

\subsubsection{$A_n$-singularity case}

In Section \ref{sec An} we consider the case $S = D^2$, $n>0$,  $|\tau| = 1$, which is the case most directly related to symplectic Khovanov homology.

We define dgas $\rnk$ and their associated $q$-graded versions $\rnkkh$, for $0 \le k \le n$. As $\rkkh$ is related to the nilHecke algebra $\nhk$ and the categorified quantum $\mathfrak{sl}_2$, $\rnkkh$ is expected to be related to the categorification of the tensor product representation $V^{\ot n}$ of quantum $\mathfrak{sl}_2$, where $V$ is the fundamental representation of $\mathfrak{sl}_2$. 

Let $\cnk$ be the bounded derived category of finitely generated $q$-graded right dg $\rnkkh$ modules. Then we have:

\begin{thm}\label{thm K0 cnk}
	There is an isomorphism of free $\mathbb{Z}[q,q^{-1}]$-modules: $\oplus_{0 \le k \le n} K_0(\cnk) \cong V^{\ot n}$, such that the classes of projective modules are mapped to the tensor basis of $V^{\ot n}$.
\end{thm}

We then construct two bimodules $\ech, \fch$ over $\oplus_{0 \le k \le n} \rnkkh$. They induce functors (still denoted $\ech, \fch$) between the $\cnk$. The functors $\ech, \fch$ descend to linear maps $$K_0(\ech), K_0(\fch): \oplus_{0 \le k \le n} K_0(\cnk) \ra \oplus_{0 \le k \le n} K_0(\cnk).$$

\begin{thm}\label{thm K0 ef}
	Under the isomorphism $\oplus_{0 \le k \le n} K_0(\cnk) \cong V^{\ot n}$, 
	$$K_0(\ech)=(q-q^{-1})E, \quad K_0(\fch)=(q-q^{-1})F,$$ where $E, F$ denote the action of quantum $\mathfrak{sl}_2$ on $V^{\ot n}$.
\end{thm}

The complex $\cal{E} \xra{x_1} \cal{E}$ in Equation \eqref{def le1} is the algebraic presentation of the Lagrangian $L$ in the local model. The categorical action of $\ech$ on geometric side is given by ``adding $L$'', i.e., by adding a Lagrangian of the form arc $a\subset D^2$ times the zero section of $T^*S^1$, where $a$ is an arc that is close to and straddles the stop $\tau$ as given in Figure~\ref{fig algn6}.  
This is compatible with the fact that the action of $\ech$ commutes with the action induced by any braid since the braid acts in the interior of the disk.  
Dually, the categorical action of $\fch$ can be understood as ``contracting $L$''.

\begin{figure}[ht]
\begin{overpic}
[scale=0.5]{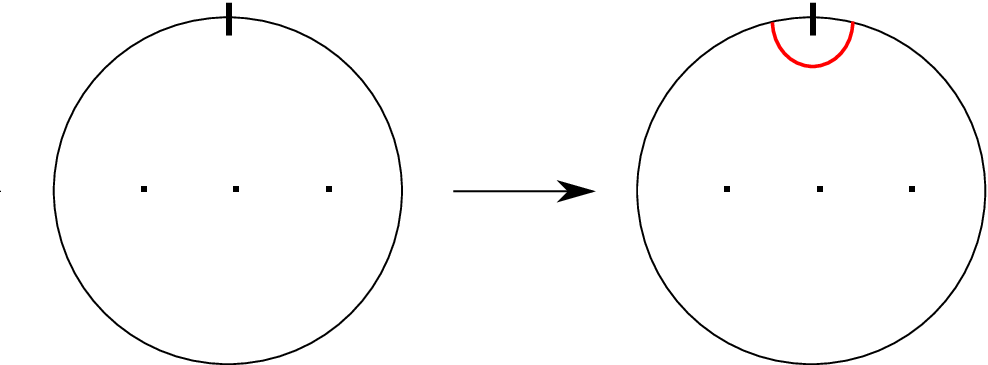}
\put(48,20){$\ech$}
\end{overpic}
\caption{The categorical action of $\ech$ given by adding a Lagrangian near the stop. The red arc is the projection of the Lagrangian to the base $D^2$.}
\label{fig algn6}
\end{figure}

Let $\cal{O}^{k, n-k}$ denote the parabolic subcategory of a regular block of $\cal{O}(\mathfrak{gl}_n)$, corresponding to the parabolic subalgebra of $\mathfrak{gl}_n$ that contains $\mathfrak{gl}_k \oplus \mathfrak{gl}_{n-k}$. Bernstein-Frenkel-Khovanov~\cite{BFK} showed that $\cal{O}^n=\oplus_{0 \le k \le n}\cal{O}^{k, n-k}$ categorifies $V^{\ot n}$ as a representation of $\mathfrak{sl}_2$ and that the Verma modules in $\cal{O}^n$ descend to the tensor basis of $V^{\ot n}$. Stroppel \cite{St} extended this result to the graded version. 

%We conjecture that the derived category of $\rnkkh$ is equivalent to the derived category $\cal{O}^{k,n-k}$.
%Moreover, the Lagrangian basis which are the idempotents of $\rnkkh$ correspond to {\em standard / Verma modules} in $\cal{O}^{k,n-k}$.
%We further conjecture that $\rnkkh$ is quasi-isomorphic to the $\ai$-Ext algebra of Verma modules in $\cal{O}^{k,n-k}$; see Conjecture \ref{conj rnk ext}. 
%This $\ai$-Ext algebra of Verma modules was studied by Klamt and Stroppel \cite{KlSt, Kl} and they conjectured that the Ext algebra is not formal in general.
%As evidence supporting the conjecture, we prove that the dga $\rnkkh$ is not formal in general.

\begin{conj} \label{conj rnk ext} $\mbox{}$
	\be
	\item The category $\cnk$ is equivalent to the derived category $\cal{O}^{k, n-k}$. Under the equivalence, the projective modules are mapped to the Verma modules.   
	Moreover, the dga $\rnkkh$ is quasi-isomorphic to the $\ai$-Ext algebra of the Verma modules.
	
	\item The functors $\ech$ and $\fch$ are equivalent to the action of complexes $\cal{E}\{1\} \xra{x_1} \cal{E}\{-1\}$ and $\cal{F}\{1\} \xra{x_1} \cal{F}\{-1\}$, respectively, where $\cal{E}$ and $\cal{F}$ refer to the categorified quantum $\mathfrak{sl}_2$ action on the categories  $\cal{O}^{k, n-k}$.
	\ee
\end{conj}

\begin{rmk}
Mak-Smith \cite{MS} considered the Fukaya category of an affine open subset of the Hilbert schemes of the Milnor fiber of the $A_{n-1}$ singularity, and showed that it is equivalent to the category $\cal{O}^n$. They mentioned that the {\em Lefschetz thimbles} correspond to the {\em standard} (Verma) modules which lift the tensor basis. The dga $\rnkkh$ is a conjectural algebraic description of the endomorphism algebra of the Lefschetz thimbles.  
\end{rmk}

The conjecture holds for $k=1$. The dga $R(n,1)^{\op{nil}}$ is isomorphic to a dga $B$ considered by Auroux-Grigsby-Wehrli \cite{AuGrWe}.  They showed that the derived category of $B$ is equivalent to the derived category of the zigzag algebra considered by Khovanov-Seidel \cite{KhSe} and it is known that both categories are equivalent to $\cal{O}^{1, n-1}$. 
Moreover, $R(n,1)^{\op{nil}}$ is isomorphic to the Ext algebra of the Verma modules in $\cal{O}^{1, n-1}$. 

There is further evidence supporting the conjecture:  For $2 \le k \le n-2$, we prove that $\rnkkh$ is not formal. On the $\cal{O}^{k, n-k}$ side, Klamt-Stroppel conjectured that the $\ai$-Ext algebra of Verma modules is not formal for $2 \le k \le n-2$; see more details in \cite{Kl, KlSt}.

\begin{rmk}
The dga $\rnkkh$ and the bimodules $\ech, \fch$ are diagrammatically described by a generalization of the strands algebra. Other generalizations are used to categorify the tensor product representations of $\mathfrak{sl}(1|1)$ by the third author \cite{Ti} and in the work of Manion and Rouquier \cite{MR} on the categorification of $\mathfrak{gl}(1|1)$.
\end{rmk}

\subsubsection{The surface case}

The Lagrangian basis in the $(D^2,n,|\tau|=1)$ case descends to the tensor basis of representations of $U_q\mathfrak{sl}_2$. The tensor basis admits a direct generalization to surfaces via the Lagrangian basis in this paper. In Section \ref{sec surface}, we define the dga $R(S,n,{\bf a})$ as a modification of the strands algebra of the surface $S$ by locally replacing the LOT algebra by $\rk$.
Rouquier~\cite{Rou} has a project to algebraically construct the whole TQFT associated to $\mathfrak{sl}_2$ and it would be interesting to compare our approach with his project.

\begin{rmk}
     There are other bases such as the {\em canonical} and {\em dual canonical bases}, with better and simpler algebraic properties. It seems hard to generalize these bases from the case of disks to general surfaces from the TQFT point of view.
\end{rmk}

\begin{rmk}
Aganagic et.\ al.\ \cite{ADLSZ} recently realized a variant of quiver Hecke algebras from Floer homology in Coulomb branches. This variant is called {\em weighted Khovanov-Lauda-Rouquier algebras}, which was introduced by Webster \cite{Web}. A special case of this is a generalization of the nilHecke algebra from the plane to a cylinder. The relevant variety is $(\mathbb{C}^*)^2$. The same manifold appears in our case as the Lefschetz fibration with base $S$ an annulus. Our resulting dga can be viewed as a ``Koszul dual'' to their cylindrical generalization of the nilHecke algebra.
\end{rmk}

\s\n
{\em Acknowledgments.} VC thanks Baptiste Chantraine and Paolo Ghiggini for their interest and support. KH is grateful to Yi Ni and the Caltech Mathematics Department for their hospitality during his sabbatical in 2018. KH also thanks Ciprian Manolescu, John Baldwin, and Otto van Koert for their help. YT thanks Yuan Gao, Rapha\"{e}l Rouquier and Peng Shan for helpful discussions.

\section{Wrapped higher-dimensional Heegaard Floer homology} \label{section: wrapped version}

The goal of this section is to associate to a pair $((W^{2\nu},\beta,\phi),h)$ consisting of a Weinstein domain $(W^{2\nu},\beta,\phi)$ and $h\in \op{Symp}(W,\bdry W,d\beta)$, a wrapped higher-dimensional Heegaard Floer (HDHF) $\ai$-bimodule $B^f(W,\beta,\phi;h)$ over a wrapped HDHF $\ai$-algebra $R^f(W,\beta,\phi)$. 
We will treat the fully wrapped case; the partially wrapped case is analogous and is left to the reader. We also show that $R^f(W,\beta,\phi)$ and $B^f(W,\beta,\phi;h)$ are invariant under Weinstein homotopies and handleslides.

\subsection{Definition of the fully wrapped $\ai$-algebra and $\ai$-bimodules}

We will freely refer to the construction of the HDHF groups from \cite{CHT} and the terminology introduced there.

Let $\mathcal{B}=\{a_1,\dots,a_\kappa\}$ be the basis of Lagrangian disks (i.e., the collection of unstable submanifolds of the Liouville vector field of $\beta$ emanating from the index $\nu$ critical points) for $(W,\beta,\phi)$.  We work in the completion $\widehat{W}=W\cup ([0,\infty)_s\times \bdry W)$, where the Lagrangians $a_i$ are completed to $\widehat{a}_i$ by attaching Lagrangian cylinders $[0,\infty)\times \partial a_i\subset [0,\infty)\times \bdry W$.  We will often use subscripts such as $s$ above to indicate coordinates.

A symplectomorphism $h$ of $(W,\bdry W,d\beta)$ can be extended to two symplectomorphisms of $(\widehat{W},d\widehat{\beta})$, namely $\widetilde{h}=h\cup\op{id}_{[0,\infty)\times \bdry W}$ and $\widehat{h}$ which is defined as follows: Consider the function
$$F_2 : [0,\infty )_s\times \partial W\to \R, \quad (s,x)\mapsto(e^s-1)^2$$
which is quadratic with respect to $e^s$. 
Then $\widehat h$ is the extension of $h$ by a time-one Hamiltonian flow of $-F_2$ on $[0,\infty )\times \partial W$.  We are using the sign convention $i_{X_{F_2}} d(e^s\beta|_{\bdry W})= dF_2$ so that the generating Hamiltonian vector field $X_{-F_2}$ on $[0,\infty)\times \bdry W$ is $2(e^s-1)R|_{\beta\vert_{\partial W}}$, where $R|_{\beta\vert_{\partial W}}$ is the Reeb vector field of $\beta\vert_{\partial W}$.

{\em Without loss of generality we always assume that we have applied a generic perturbation to $\widehat{h}$ on a compact region so that all the relevant Lagrangian intersections below become transverse; this is particularly relevant when $a_{\bf m}$ and $a_{\bf m'}$ given below use copies of the same Lagrangian.}

For each nonnegative integer $k$, take a partition ${\bf m}= (m_1,\dots, m_\kappa)$ of $k$. We write $|{\bf m}|=\sum_{i=1}^\kappa m_i=k$.  Let $a_{\bf m}$ (resp.\ $\widehat{a}_{\bf m}$) be the union of $k$ disjoint asymptotically cylindrical Lagrangian planes consisting of $m_i$ disjoint copies of $a_i$ (resp.\ $\widehat{a}_i$) for $i=1,\dots, \kappa$.

Following \cite[Section 5.6 and Notation 5.6.1]{CHT}, let $\Hom(\widehat{a}_{\bf m}, \widehat{a}_{\bf m'})$ be the HDHF cochain complex $\widehat{CF}(\widehat{\op{id}}(\widehat{a}_{\bf m}), \widehat{a}_{\bf m'})$ whose generators are $k$-tuples ${\bf y}=\{y_1,\dots,y_k\}$ of intersection points between $\widehat{\op{id}}(\widehat{a}_{\bf m})$ and $\widehat{a}_{\bf m'}$, where the $y_i$ lie on distinct components of $\widehat{\op{id}}(\widehat{a}_{\bf m})$ and $\widehat{a}_{\bf m'}$. Observe that we are fully wrapping the first term in the $\Hom$. The differential $\mu_1$ will be defined in Section~\ref{subsection: A-infinity}.

We then define the $A_\infty$-algebras
\begin{gather*}
R^f(W,\beta,\phi;k):= \oplus_{|{\bf m}|=|{\bf m'}|=k} \Hom(\widehat{a}_{\bf m},\widehat{a}_{\bf m'}),\\
R^f(W,\beta,\phi) := \oplus_{k\geq 0} R^f(W,\beta,\phi;k),
\end{gather*}
where the $A_\infty$-operations will be defined in Section~\ref{subsection: A-infinity}.

Similarly we define 
\begin{gather} \label{eq bimodule for h}
B^f(W,\beta,\phi;h;k):=  \oplus_{|{\bf m}|=|{\bf m'}|=k} \Hom(\widetilde h(\wh a_{\bf m}),\wh a_{\bf m'})=\oplus_{|{\bf m}|=|{\bf m'}|=k} \widehat{CF}(\widehat{h}(\widehat{a}_{\bf m}),\widehat{a}_{\bf m'}),\\
B^f(W,\beta,\phi;h):= \oplus_{k\geq 0} B^f(W,\beta,\phi;h; k).
\end{gather}
The $R^f(W,\beta,\phi;k)$ and $R^f(W,\beta,\phi)$ $\ai$-bimodule operations will also be given in Section~\ref{subsection: A-infinity}.

%We also view $R^f(W,\beta,\phi)$ as an $\ai$-category $\mathcal{R}^f(W,\beta,\phi)$ and $B^f(W,\beta,\phi;h)$ as an $\ai$-bi\-module $\mathcal{B}^f(W,\beta,\phi;h)$ over $\mathcal{R}^f(W,\beta,\phi)$.   Let $\oR$ and $\oBh$ be the triangulated envelopes of $\mathcal{R}^f(W,\beta,\phi)$ and $\mathcal{B}^f(W,\beta,\phi;h)$  (i.e., the smallest triangulated $A_\infty$-category containing a given $A_\infty$-category). 

%$$(CF^f (W,\beta,\phi; h), \partial):=\textstyle\oplus_{a,a' \in \mathcal{A}} CF(\widehat{a},\widehat{h} (\widehat{a}' ),\partial ).$$

Since the number of intersection points between $\widehat{h}(\wh a_{\bf m})$ and $\wh a_{\bf m'}$ is potentially infinite, the definition relies on the following lemma:

\begin{lemma}\label{lemma: finiteness}
If ${\bf y}=\{y_1,\dots,y_k\}$ is a generator of $\Hom(\widetilde{h} (\wh a_{\bf m}),\wh a_{\bf m'})$ with $|{\bf m}|=|{\bf m'}|=k$, then the summand of $d {\bf y}$ corresponding to a fixed genus is finite.
\end{lemma}

\begin{proof}
Let $u: \dot F\to D_1\times \widehat{W}=\R \times [0,1]\times \widehat{W}$ be a $J^\Diamond$-holomorphic map, where $J^\Diamond$ is compatible almost complex structure on $\R \times [0,1]\times \widehat{W}$ in the sense of \cite[Section 3.2]{CHT}, from ${\bf y}$ to a $k$-tuple of chords ${\bf y'}$ and let $u'=\pi_{\widehat{W}}\circ u$ be the composition, where $\pi_{\widehat{W}}$ is the projection to $\widehat{W}$. By Stokes' theorem, $\textstyle\int_{u' (\partial F)} \widehat{\beta} = \textstyle\int_{ u' (F)} d\widehat{\beta}$, where the right-hand side is positive by the holomorphicity of $u$ and the compatibility condition for $J^\Diamond$. Moreover,
\begin{equation} \label{eqn: coffee}
\textstyle\int_{u' (\partial F)} \widehat{\beta} =\textstyle\int_{u' (\partial F)\cap \widehat{h}(\wh a_{\bf m'})}\widehat{\beta} - \textstyle\int_{u' (\partial F)\cap \wh a_{\bf m}} \widehat{\beta}.
\end{equation}

The portion of $u' (\partial F)\cap \wh a_{\bf m'}$ in $[0,\infty)\times \bdry W$ is contained in the half-cylinder $[0,\infty)\times \bdry \wh a_{\bf m'}$ on which $\widehat\beta$ vanishes.  On the other hand, the portion of $u' (\partial F)\cap \widehat{h}(\wh a_{\bf m})$ in $[0,\infty)\times \bdry W$ has the same $\wh \beta$-length as a concatenation of arcs in the half-cylinder $[0,\infty)\times \bdry \wh a_{\bf m}$ and some Reeb chords of $R$.  The contributions of the arcs in the half-cylinders to $\textstyle \int \widehat\beta$ are zero and the contributions of the Reeb chords to $\textstyle \int \widehat\beta$ are $\pm$ the actions of the Reeb chords, depending on whether the Reeb chord corresponds to a point in ${\bf y}$ or ${\bf y'}$.
%The term $\int_{u' (\partial F)\cap \wh a_{\bf m}} \widehat{\beta}$ takes only finitely many values, since the contribution of the part of the boundary that is in $[0,\infty )\times \partial W$ vanishes. A coordinate of ${\bf y'}$ sitting in $[0,\infty )\times \partial W$ at an altitude $s$ contributes to $\int_{u' (\partial F)\cap \widehat{h}(\wh a_{\bf m'})}\widehat{\beta}$ to an amount of order $2e^s$, which is the action of the corresponding Reeb chord for $R$.

Hence if ${\bf y}$ is fixed, then the maximum $s$-value of the coordinates of ${\bf y'}$ is bounded and there are only finitely many possible ${\bf y'}$. Finally, if ${\bf y}$, ${\bf y'}$, and the topological type of curve are fixed, the curve count is finite.
\end{proof}

\subsection{The $\ai$-operations}\label{subsection: A-infinity}

We now equip $R^f(W,\beta,\phi;k)$ with an $\ai$-structure following \cite[Section 4]{CHT} and \cite[Section 2]{HTY}. There are some subtleties explained in \cite{Au} and worked out in detail in \cite{Ab}.  In this subsection we briefly review how to adapt the constructions to the cylindrical setting, leaving more details to \cite[Section 2.2]{HTY}, where we replace $T^*M$ there by $\widehat{W}$.

As shown in Figure~\ref{fig: A-infty}, let $D$ be the unit disk in $\C$ and let $D_m =D\setminus\{p_0,\dots,p_{m}\}$, where $p_0,\dots,p_m$ are arranged in counterclockwise order around $\partial D$.  The points $p_1,\dots,p_m$ are incoming and $p_0$ is outgoing. Let $\bdry_i D_m$, $i=0,\dots,m$, be the counterclockwise arc of $\bdry D_m$ from $p_i$ to $p_{i+1}$, where $p_{m+1}=p_0$.
\begin{figure}[ht]
	\begin{overpic}[scale=1]{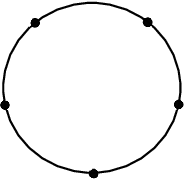}
		\put(47,-7){\tiny $p_0$}
		\put(102,40){\tiny $p_1$}
		\put(85,85){\tiny $p_2$}
		\put(-6,85){\tiny $p_{m-1}$}
		\put(-14,40){\tiny $p_m$}
		\put(82,8){\tiny $\bdry_0 D_m$}
		\put(97.5,65){\tiny $\bdry_1 D_m$}
		\put(-36,65){\tiny $\bdry_{m-1} D_m$}
		\put(-6,8){\tiny $\bdry_m D_m$}
	\end{overpic}
	\caption{The $A_\infty$-base $D_m$ of the symplectic fibration $D_m\times \widehat W$.}
	\label{fig: A-infty}
\end{figure}
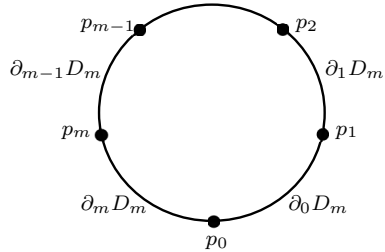
Let $\mathcal{A}_m$ be the moduli space of $D_m$ modulo automorphisms.  We choose representatives $D_m$ of equivalence classes of $\mathcal{A}_m$ in a smooth manner (e.g., by setting $p_0=-i$ and $p_1=i$) and abuse notation by writing $D_m\in \mathcal{A}_m$.
We call $D_m$ the ``$A_\infty$-base direction''.  Let $e_i\subset D_m$ be the strip-like end $[0,\infty)_{s_i}\times [0,1]_{t_i}$ corresponding to a neighborhood of $p_i$, which smoothly varies with $D_m\in\mathcal{A}_m$.

% the moment {\em assume that $D_m$ is ``sufficiently far'' from $\bdry \mathcal{A}_m$.} 
In orer to define the $A_\infty$-operation $\mu_m$ we need to construct a {\em consistent universal choice of datum} on all of $\mathcal{A}_m$ following \cite{Se3} and \cite[Definition 4.2]{Ab} and described in \cite[Section 2.2]{HTY}. We sketch some ingredients of the construction.

Consider the symplectic fibrations
\begin{equation}\label{eqn: pi D m}
\pi_{D_m}: (D_m\times \widehat{W},\widehat\Omega_m)\to (D_m,\omega),
\end{equation}
where $\omega$ is an area form on $D$ which restricts to $ds_i\wedge dt_i$ on each $e_i$. Also let $\pi_{\widehat W}$ be the projection $D_m\times \widehat W\to \widehat W$. We take 
$$\widehat \Omega_m=\omega +dH+ d\widehat \beta,$$
where $H: D_m\times \widehat{W}\to \R$ is a Hamiltonian which is supported over the strip-like end $e_0$. 

\s\n
{\em Almost complex structures.} Let $\mathcal{J}_{\widehat{W},\widehat \beta}$ be the set of $d\widehat \beta$-compatible almost complex structures $J_{\widehat{W}}$ on $\widehat{W}$ that are asymptotic to an almost complex structure on $[0,\infty)_s\times \bdry W$ that takes $\partial_s$ to the Reeb vector field of $\widehat \beta|_{\bdry W}$, takes $\ker \widehat \beta|_{\bdry W}$ to itself, and is compatible with $d\widehat \beta|_{\bdry W}$. 

There is a smooth assignment 
$D_m\mapsto J_{D_m}$, where $D_m\in \mathcal{A}_m$, such that:
\begin{enumerate}
    \item[(J1)] on each fiber $\pi_{D_m}^{-1}(p)=\{p\}\times \widehat{W}$, $J_{D_m}$ restricts to an element of $\mathcal{J}_{\widehat{W},\widehat \beta}$;
    \item[(J2)] $J_{D_m}$ projects holomorphically onto $D_m$;
    \item[(J3)] over each strip-like end $[0,\infty)_{s_i}\times[0,1]_{t_i}$, $J_{D_m}$ is invariant in the $s_i$-direction and takes $\partial_{s_i}$ to $\partial_{t_i}$; when $m=1$, $J_{D_1}$ is invariant under $\mathbb{R}$-translation of the base and takes $\partial_{s_i}$ to $\partial_{t_i}$.
\end{enumerate}
One can inductively construct such an assignment for all $m\geq1$ in a manner which is (A) consistent with the boundary strata 
and (B) for which all the moduli spaces $\mathcal{M}({\bf y}_m,\dots, {\bf y}_1, {\bf y}_0)$, defined below, are transversely cut out. 
A collection 
$$\{ J_{D_m}~|~D_m\in \mathcal{A}_m,~m\in \mathbb{Z}_{>0}\}$$ 
satisfying (A) will be called a {\em consistent collection} of almost complex structures; if it satisfies (B) in addition, it is a {\em sufficiently generic} consistent collection.

\s\n {\em Lagrangian boundary conditions.} We define Lagrangians $\widetilde L_i$, $i=0,\dots,m$, over $\bdry_i D_m$; its components will be denoted $\widetilde L_{ij}$, $j=1,\dots, k$, where $k=|{\bf m}^i|$ for all $i$.
The Lagrangian $\widetilde L_i$ restricts to $\widehat{\op{id}}^{m-i} (\wh a_{{\bf m}^i})$ over each point of $\bdry_i D_m\setminus e_0$.  Note that $\widehat{\op{id}}^{m-i} (\wh a_{{\bf m}^i})$ roughly denotes the wrapping of $\wh a_{{\bf m}^i}$ at ``speed'' $m-i$.

In order to define $\widetilde L_0$ and $\widetilde L_m$ over $e_0$ we need some preparation: Let $\psi^\rho$ be the time $\log \rho$ flow of the Liouville vector field of $(\widehat{W},\widehat\beta)$ which is equal to $\bdry_s$ on $[0,\infty)\times \bdry W$. We have a natural isomorphism
\begin{align} \label{eqn: iso wrapped}
CF( \psi^\rho &(\widehat{\op{id}}(\wh a_{{\bf m}^0})), \psi^\rho(\wh a_{{\bf m}^m});\psi^\rho_*(J_{D_1}))\simeq \\
& CF(\widehat{\op{id}}(\wh a_{{\bf m}^0}), \wh a_{{\bf m}^m}; J_{D_1})= \Hom(\wh a_{{\bf m}^0},\wh a_{{\bf m}^m}).\nonumber
\end{align}
Observe that $\psi^\rho(\wh a_{{\bf m}^m})$ is exact Lagrangian isotopic to $\wh a_{{\bf m}^m}$ by a compactly supported isotopy and $\psi^\rho (\widehat{\op{id}}(\wh a_{{\bf m}^0}))$ limits to $\widehat{\op{id}}^{m}(\wh a_{{\bf m}^0})$ as $s\to \infty$ when $\rho=-\log m$. [Idea: View $\widehat{\op{id}}(\wh a_{\bf m})$ as a graph $x= 2(e^s-1)$, where $x$ is the coordinate in the Reeb direction, and view $\psi^\rho(\widehat{\op{id}}(\wh a_{\bf m}))$ as the translate $x= 2(e^{s-\rho}-1)$.]  

We then take $\widetilde L_0$ over $\{t_0=1\}\subset e_0$ to be the trace of the Lagrangian isotopy from $\widehat{\op{id}}^{m}(\wh a_{{\bf m}^0})$ to  $\psi^\rho (\widehat{\op{id}}(\wh a_{{\bf m}^0}))$ as $s_0\to +\infty$ and $\widetilde L_m$ over $\{t_0=0\}\subset e_0$ to be the trace of the Lagrangian isotopy from $\wh a_{{\bf m}^m}$ to $\psi^\rho(\wh a_{{\bf m}^m})$ as $s_0\to +\infty$. As explained in \cite[Lemma 6.1.1]{CHT}, there exists a Hamiltonian $H: D_m\times \widehat{W}\to \R$ which makes the traces Lagrangians with respect to $\widehat\Omega_m$. 
%It is also easy to extend from $j_{\R\times[0,1]}\times J_{\widehat W}$ to $\psi^\rho_*( j_{\R\times[0,1]}\times J_{\widehat W})$ over $e_0$ as $s_0\to +\infty$.

%\begin{rmk}  
%    To avoid cumbersome terminology, in what follows, when we say {\em sufficiently generic}, we mean that all the moduli spaces under consideration are transversely cut out.
%\end{rmk}

\s\n {\em Modui spaces.} Given generators ${\bf y}_i$, $i=1,\dots,m$, of $\Hom (\wh a_{{\bf m}^{i-1}},\wh a_{{\bf m}^i})$ and a generator ${\bf y}_0$ of $\Hom(\wh a_{{\bf m}^0},\wh a_{{\bf m}^m})$,
let $\mathcal{M}(\mathbf{y}_m,\dots,\mathbf{y}_1;\mathbf{y}_0)$ be the moduli space of maps
\begin{equation*}
    u\colon (\dot F,j)\to(D_m\times \widehat{W},J_{D_m}),
\end{equation*}
modulo domain automorphisms, where $(F,j)$ is a compact Riemann surface with boundary, $\mathbf{p}_0,\dots,\mathbf{p}_m$ are disjoint $\kappa$-tuples of boundary marked points of $F$, $\dot F=F\setminus\cup_i \mathbf{p}_i$, and $D_m\in \mathcal{A}_m$, so that $u$ satisfies
\begin{align}
    \label{floer-condition}
    \left\{
        \begin{array}{ll}
            \text{$du\circ j=J_{D_m}\circ du$;}\\
            \text{each component of $\partial \dot F$ is mapped to a unique $\widetilde{L}_{ij}$;}\\
            \text{$\pi_{\widehat{W}}\circ u$ tends to $\widehat{\op{id}}^{m-i} (\mathbf{y}_i)$ as $s_i\to+\infty$ for $i=1,\dots,m$;}\\
             \text{$\pi_{\widehat{W}}\circ u$ tends to $\psi^\rho(\mathbf{y}_0)$ as $s_0\to+\infty$;}\\
            \text{$\pi_{D_m}\circ u$ is a $\kappa$-fold branched cover of $D_m$.}
        \end{array}
    \right.
\end{align}
With the identification of the strip-like end $e_i$ with $[0,\infty)_{s_i}\times[0,1]_{t_i}$, the 3rd and 4th conditions mean that $u$ maps the neighborhoods of the punctures of $\mathbf{p}_i$ asymptotically to the Reeb chords $[0,1]_{t_i}\times \mathbf{y}_i$, modified by the relevant diffeomorphisms, as $s_i\to +\infty$.

%The maps $\pi_{D_m}$ are subject to some strata compatibility conditions as $D_m$ approaches $\bdry \mathcal{A}_m$. The construction of such a {\em consistent universal choice of datum} follows

\s\n {\em $A_\infty$-operations.} We now define the $A_\infty$-operations
\begin{gather*}
\mu_m :\Hom (\wh a_{{\bf m}^{m-1}},\wh a_{{\bf m}^m})\otimes\dots\otimes \Hom(\wh a_{{\bf m}^0},\wh a_{{\bf m}^1}) \to \Hom(\wh a_{{\bf m}^0},\wh a_{{\bf m}^m})\\
\mathbf{y}_m\otimes\dots\otimes\mathbf{y}_1\mapsto\sum_{\mathbf{y}_0,\chi\leq\kappa}\#\mathcal{M}^{\op{ind}=0,\chi}(\mathbf{y}_m,\dots,\mathbf{y}_1;\mathbf{y}_0)\cdot\hbar^{\kappa-\chi}\cdot\mathbf{y}_0,
\end{gather*}
where the superscript $\chi$ denotes the Euler characteristic of $F$ and the symbol $\#$ denotes the signed count of the corresponding moduli space. The superscript $\op{ind}$ refers to the virtual dimenson of the moduli space; when $m=1$, $\mathcal{M}^{\op{ind}=0,\chi}({\bf y}_1;{\bf y}_0)$ is shorthand for $\mathcal{M}^{\op{ind}=1,\chi}({\bf y}_1;{\bf y}_0)/ \R$ where $\R$ is the target $\R$-translation of the $A_\infty$-base.

\begin{prop}\label{prop: algebra}
$R^f(W,\beta,\phi;k)=B^f(W,\beta,\phi;\op{id};k)$ is a cohomologically unital $\ai$-algebra with respect to the $A_\infty$-operations $(\mu_m)_{m\in \Z^+}$.
\end{prop}

\begin{proof} 
The $\ai$-relations arise from the description of the codimension $1$ boundary of $\mathcal{A}_m$ and the elimination of undesirable bubbles in \cite[Section 3.7 and 3.9]{CHT}.
% One needs to take care of the twisted Lagrangian boundary conditions by a consistent universal choice of datum; see \cite{Se3} and also \cite[Definition 4.2]{Ab} for details.
The cohomological unitality is explained in \cite[Sections II.8c and II.9j]{Se3}.
The identity element associated with the object $a_{{\bf m}}$ is obtained by counting curves in the fibration $D_0\times \widehat{W}$ over a monogon $D_0\in \mathcal{A}_0$, with Lagrangian boundary condition given by tracing a Lagrangian isotopy between $\wh a_{{\bf m}}$ and a pushoff $\wh{\op{id}} (\wh a_{{\bf m}}')$. To construct $a_{{\bf m}}'$, we choose a Morse function $F: a_{{\bf m}}\to \R$ such that there is exactly one maximum on each component of $a_{{\bf m}}$ and the maximum is in the interior.
We use this Morse function to perturb $a_{{\bf m}}$ to $a_{{\bf m}}'$, obtained as the graph of $dF$ in a Weinstein neighborhood $T^*a_{{\bf m}}$ of $a_{{\bf m}}$ in $W$.
\end{proof}

\s\n
{\em $A_\infty$-bimodule operations.}
Similarly we define the $\ai$-bimodule operations
\begin{align*}
&\mu_{a,b}: (\Hom (\wh a_{{\bf m}^{a-1}},\wh a_{{\bf m}^a})\otimes\dots\otimes \Hom(\wh a_{{\bf m}^0},\wh a_{{\bf m}^1}))\otimes \Hom(\widetilde{h}(\wh a_{{\bf m'}^{b}}),\wh a_{{\bf m}^0}) \\
& \otimes (\Hom (\widetilde{h}(\wh a_{{\bf m'}^{b-1}}),\widetilde{h}(\wh a_{{\bf m'}^b}))\otimes\dots\otimes \Hom(\widetilde{h}(\wh a_{{\bf m'}^0}),\widetilde{h}(\wh a_{{\bf m'}^1}))) \to \Hom(\widetilde{h}(\wh a_{{\bf m'}^{0}}),\wh a_{{\bf m}^a}),
\end{align*}
by counting elements of 
$$\#\mathcal{M}^{\op{ind}=0,\chi}(\mathbf{y}_a,\dots,\mathbf{y}_1, \mathbf{x},\mathbf{y}'_b,\dots,\mathbf{y}'_1;\mathbf{x}').$$
This gives $B^f(W,\beta,\phi;h;k)$ the structure of an $\ai$-bimodule over $R^f(W,\beta,\phi;k)$.

\begin{prop}\label{prop: bimodule}
$B^f(W,\beta,\phi;h;k)$ is an $\ai$-bimodule over $R^f(W,\beta,\phi;k)$ with respect to the operations $(\mu_{a,b})_{a,b\in \Z^{\geq 0}}$.
\end{prop}

\subsection{Handleslides}

The goal of this subsection is to prove:

\begin{thm}\label{thm: invariance under handleslides}
%The  envelope $\ai$-category $\overline{\mathcal{R}}^f(W,\beta,\phi)$ and the envelope $\ai$-$\overline{\mathcal{R}}^f(W,\beta,\phi)$-bi\-module $\overline{\mathcal{B}}^f(W,\beta,\phi;h)$ are invariant up to quasi-equivalence under handleslides arising from Weinstein homotopies $(\beta_t,\phi_t)$, $t\in[0,1]$. 
The $\ai$-algebras $R^f(W,\beta_t,\phi_t)$, $t=0,1$, are Morita equivalent under handleslides arising from the Weinstein homotopy $(\beta_t,\phi_t)$, $t\in[0,1]$.
Moreover, the $\ai$-$R^f(W,\beta_t,\phi_t)$-bi\-modules $B^f(W,\beta_t,\phi_t;h)$, $t=0,1$, are quasi-equivalent with respect to the Morita equivalence of $R^f(W,\beta_t,\phi_t)$, $t=0,1$.
%\marginpar{\cg I move the notion of idempotent completions in front of Corollary 2.3.9. Now the handleslide theorem is stated in the form of algebra and bimodules.}
\end{thm}

To this end we present two models of a handleslide that connect the handleslide operation to the better-studied operation of the resolution of a Lagrangian intersection. This leads to the description of the relevant space of $J$-holomorphic curves in Corollary \ref{cor: resolution}. The proof of Theorem \ref{thm: invariance under handleslides} then follows that of Seidel's long exact sequence \cite{Se1} which relates the geometric operations to the algebraic cone.

The following is an immediate corollary:

\begin{cor}\label{cor: Hochschild}
The Hochschild homology of the $\ai$-$R^f(W,\beta,\phi)$-bi\-module $B^f(W,\beta,\phi;h)$ is invariant under handleslides, i.e., Weinstein homotopies $(\beta_t,\phi_t)$, $t\in[0,1]$.
\end{cor}

It is a difficult problem to compute the Hochschild homology groups.  The goal of Part 2 is to give a conjectural combinatorial description of the $A_\infty$-algebras in certain cases.

\subsubsection{First handleslide model}\label{subsub: model1}

We consider an {\em elementary handleslide} Weinstein homotopy $(\beta_t,Y_t,\phi_t)$, $t\in[0,1]$, on $W$ which corresponds to a $\nu$-dimensional handleslide; here we are explicitly writing the Liouville vector field $Y_t$ for $\beta_t$.  More precisely, 
\be
\item[(i)] $\phi_t$ is Morse for all $t$;
\item[(ii)] for $t\not=\tfrac{1}{2}$, every trajectory of $Y_t$ between critical points is from a lower index critical point to a higher index critical point;
\item[(iii)] for $t=\tfrac{1}{2}$, every trajectory of $Y_{1/2}$ between critical points is from a lower index critical point to a higher index critical point with the exception of a single trajectory between two index $\nu$ critical points, say from $p$ to $q$;
\item[(iv)] $(\beta_t,Y_t,\phi_t)$ is generic in a family.
\ee

We describe a model for a handleslide on a neighborhood of the union of unstable manifolds of $p$ and $q$.  This model will be used in the proof of Proposition~\ref{prop: half-arc over a half-arc}.

\s\n
{\em Case $\op{dim} W=2$.} Consider $\R^2_{x_1,y_1}$.  Let $p=(0,0)$ and $q=(1,0)$. We take 
\begin{equation}
\beta_{1/2}=\left\{  \begin{array}{cl}
2x_1dy_1 +y_1dx_1, & \mbox{ on } x_1\leq \tfrac{1}{4}, \\
(1-x_1)dy_1-2y_1dx_1, & \mbox{ on } x_1\geq \tfrac{3}{4}.
\end{array}\right.
\end{equation}
For $x_1\leq \tfrac{1}{4}$, $Y_{1/2}=2x_1 \partial_{x_1} -y_1\partial_{y_1}$. This is the standard model of a saddle whose unstable trajectories are along the $x_1$-axis and whose stable ones are along the $y_1$-axis. For $x_1\geq \tfrac{3}{4}$, $Y_{1/2}=(1-x_1)\partial_{x_1} +2y_1\partial_{y_1}$. The stable manifold of $q$ is along the $x_1$-axis and the unstable manifold $\widetilde{a}_1$ of $q$ is in the $y_1$-direction.
We easily extend $\beta_{1/2}$ to the region $\tfrac{1}{4}\leq x_1\leq \tfrac{3}{4}$ so that $d\beta_{1/2}=dx_1\wedge dy_1$ and $\{ 0< x_1< 1, y_1=0\}$ is a trajectory of $Y_{1/2}$.

Now for all $0<|\epsilon|\leq \epsilon_0$ small, let $\beta_{1/2+\epsilon}$ be a small perturbation of $\beta_{1/2}$ which is compactly supported on a neighborhood of $\{ \tfrac{1}{3} <x_1<\tfrac{2}{3}, y_1=0\}$ so that the unstable manifold $a_1(\tfrac{1}{2}+\epsilon)$ of $Y_{1/2+\epsilon}$ from $p$ passes through $(\tfrac{3}{4},\epsilon )$. %We denote $a_1(\epsilon )$ the unstable manifold of $p$ with respect to $\beta_\epsilon$.
Note that if $\epsilon$ is small, it hits the boundary of $W$ close to the unstable manifold $\widetilde{a}_1$ of $q$ (which does not depend on $t$). In particular, there is a short Reeb chord between the Legendrian submanifolds $\partial a_1(\tfrac{1}{2}+\epsilon)$ and $\partial \widetilde{a}_1$ of $(\partial W ,\beta_{1/2+\epsilon})$, which goes from $\partial \widetilde{a}_1$ to $\partial a_1(\tfrac{1}{2}+\epsilon)$ when $\epsilon >0$ and vice versa when $\epsilon<0$; the latter one will be called $c$.

We then extend $\beta_t$ to all $t\in[0,1]$ so that it is locally constant on $[0,1]\setminus [\tfrac{1}{2}-\epsilon_0,\tfrac{1}{2}+\epsilon_0]$.

\s\n
{\em Case $\op{dim} W=2\nu$.} The general case is a product $\R^2_{x_1,y_1} \times \R^{2\nu-2}_{x_2,y_2,\dots,x_\nu,y_\nu}$. Let $\overline p=(0,\dots,0)$ and $\overline q=(1,0,\dots,0)$.
For $|\epsilon|\leq \epsilon_0$ we take
\begin{gather*}
\overline{\beta}_{1/2+\epsilon} =\beta_{1/2+\epsilon} +\Sigma_{i=2}^\nu (2x_idy_i+y_idx_i),\\
\overline{Y}_{1/2+\epsilon} = Y_{1/2+\epsilon} +\Sigma_{i=2}^\nu (2x_i\partial_{x_i}-y_i\partial_{y_i}).
\end{gather*}
The unstable submanifolds of $\overline p$ and $\overline q$ are:
\begin{gather*}
\overline{a}_1(\tfrac{1}{2}+\epsilon) =a_1(\tfrac{1}{2}+\epsilon) \times \{ y_2=\dots =y_\nu=0\},\\
\overline{\widetilde{a}}_1 =\{ x_1=1, y_2 =\dots =y_\nu =0\}.
\end{gather*}
%If we have chosen coordinates so that $W\cap N(\widetilde a_2 ) =\{ \Sigma_{i=1}^\nu y_i^2 \leq 1, \Sigma_{i=1}^\nu x_i^2 \leq \delta \}$, then

One can verify that there is one short (nondegenerate) Reeb chord $c_{1/2+\epsilon}$ in $(\partial W,\overline \beta_{1/2+\epsilon})$ between the Legendrians $\partial \overline{a}_1(\tfrac{1}{2}+\epsilon)$ and $\partial \overline{\widetilde{a}}_1$, which goes from $\partial \overline{\widetilde{a}}_1$ to $\partial \overline a_1(\tfrac{1}{2}+\epsilon)$ when $\epsilon >0$ and vice versa when $\epsilon<0$.

\s
{\em From now on we omit the bars and write $c$, $d$, $a_1$, $b_1$ for $c_{1/2-\epsilon}$, $c_{1/2+\epsilon}$, $a_1(\tfrac{1}{2}- \epsilon)$, $a_1(\tfrac{1}{2}+\epsilon)$ with $\epsilon>0$.}

\subsubsection{Second handleslide model}\label{subsub: model2}

We now give a second description of a handleslide.  

Let $(\widehat{W},\widehat\beta)$ be the completion of $(W,\beta)$ obtained by gluing $([0,\infty)_s \times \partial W, d(e^s \beta\vert_{\partial W}))$; also let $W_{[0,s_0]}:=W\cup ([0,s_0]\times \bdry W)$. We extend $a_1$, $\widetilde{a}_1$ and $b_1$ to $\widehat{a}_1$, $\widehat{\widetilde{a}}_1$ and $\widehat{b}_1$ on $\widehat W$ by attaching Lagrangian cylinders $[0,\infty) \times \partial a_1$, $[0,\infty) \times \partial \widetilde{a}_1$ and $[0,\infty) \times \partial b_1$. We then apply a partial wrapping to $\widehat a_1$ to obtain $\widehat{a}_1^{pw}$ using the time-$1$ flow of a Hamiltonian $\delta F(s)$, where $F(s)=\int_0^s e^t \arctan (t)dt$ and the constant $\delta>0$ is chosen %(with order of magnitude of the action of the small chord $c$) 
so that $\widehat{a}_1^{pw}$ and $\widehat{\widetilde{a}}_1$ intersect transversely at a point $p_c\subset\{s=1\}$ corresponding to $c$, and at no other point. The chord $c$ is transformed into a chord $c'\subset \{s=2\}$ from $\widehat{\widetilde{a}}_1 \cap \{s=2\}$ to $\widehat{a}_1^{pw} \cap \{s=2\}$.

%Now we take the Lagrangian connected sum of $\widehat{a}_1^{pw}\cup \widehat{\widetilde{a}}_1$ at the intersection point $p_c$ to obtain a Lagrangian cylinder $l_1$. There are two possible connected sums, and in this paper we will always use one particular one that we also call the ``resolution'': 

At this point we recall the Lagrangian connected sum (which we also call ``resolution''), given by the following local model:
Let $(\R^{2\nu}_{x_1,y_1,\dots,x_\nu,y_\nu},\omega=\Sigma_{i=1}^\nu dx_i \wedge dy_i)$ be a symplectic manifold and let $L_1 =\{ y_1=\dots=y_\nu=0\}$ and $L_2=\{ x_1 =\dots =x_\nu=0\}$ be Lagrangian submanifolds that intersect at $(0,\dots, 0)$. Then we define $L_1 \sharp L_2$ to be given by the graph of the differential of $(x_1,\dots,x_\nu)\mapsto \log \Vert x\Vert$.  Note that $L_1\sharp L_2\not= L_2\sharp L_1$ according to the definition.

Let $l_1$ be the Lagrangian cylinder obtained by taking $ \widehat{\widetilde{a}}_1\sharp \widehat{a}_1^{pw}$ at the intersection point $p_c$.

\begin{lemma}
    After possibly applying a small Lagrangian isotopy to $l_1$, the restriction $l_1\cap W_{[0,1]}$ is a Lagrangian disk with Legendrian boundary that is isotopic to $\widehat{b}_1\cap W_{[0,1]}$ through Lagrangian disks with Legendrian boundary in $\partial W_{[0,1]}$.
\end{lemma}

\begin{proof}
    First observe that in the above model $(\R^{2\nu},\omega)$ the restriction $(L_1 \sharp L_2 )\cap \{x_1=y_1=0\}$ coincides with the resolution $(L_1 \cap \{x_1=y_1=0\}) \sharp (L_2 \cap \{x_1=y_1=0\})$ in $\{x_1=y_1=0\}$. 
    
    Consider the primitive $\beta = (1+x_1)dy_1 +\Sigma_{i=2}^\nu x_i dy_i$ of $\omega$. The hypersurfaces
    $$(\mathcal{H}_\delta:=\{x_1+y_1=\delta\},\beta|_{\mathcal{H}_\delta} )$$ 
    are contact since they are transverse to the Liouville vector field $X_\beta=(1+x_1)\bdry_{x_1} + \sum_{i=2}^\nu x_i \bdry_{x_i}$ of $\beta$.  Observe that $X_\beta$ points in the direction of increasing $\delta$ and the local model describes the switching of the ``Reeb heights'' of $L_1\cap \mathcal{H}_\delta$ and $L_2\cap \mathcal{H}_\delta$ as $\delta$ goes from negative to positive.  We claim that $L_1 \sharp L_2$ can be modified by a small isotopy so that the intersection $(L_1 \sharp L_2 )\cap\mathcal{H}_0$ is Legendrian. Indeed, one can view $(L_1 \cup L_2 )\cap \mathcal{H}_0$ as an immersed Legendrian in $(\mathcal{H}_0,\beta )$ with a double point at $p_c$. A neighborhood of it is obtained by plumbing the $1$-jet spaces of $L_1 \cap \mathcal{H}_0$ and $L_2 \cap \mathcal{H}_0$, where the spaces $L_1 \cap \mathcal{H}_0$ and $L_2 \cap\mathcal{H}_0$ correspond to the $0$-sections. We then resolve $L_1\cup L_2$ near $p_c$ so that $(L_1\sharp L_2)\cap \mathcal{H}_0$ is a resolution of $(L_1 \cup L_2 )\cap \mathcal{H}_0$ near $p_c$, viewed as a Lagrangian in the Lagrangian projection of the plumbed jet spaces.
    Recall that the total resolution is obtained as the graph of the differential of $F\colon (x_1,\dots,x_\nu)\mapsto \log \Vert x\Vert$, and the resolution of $(L_1 \cup L_2 )\cap \mathcal{H}_0$ is its restriction to $x_1=y_1=0$. Finally we deform $L_1 \sharp L_2$ so that its intersection with $\mathcal{H}_0$ is the Legendrian lift of the resolution of $(L_1 \cup L_2 )\cap \mathcal{H}_0$ near $p_c$ in the $1$-jet space neighborhood. To adapt the intersection with $\mathcal{H}_0$, we  just need to modify $F$ to give $\frac{\partial F}{\partial x_1} \vert_{\{x_1=0\}}$ the desired value.
    By \cite[Theorem 6.3, Figure 8(B)]{CMP}, this Legendrian lift has a front projection that is a cone-sum pictured in Figure \ref{picture: surgery} at level $s=1$.
    Note that the Legendrian lift is close to its Lagrangian projection.

   % By \cite[Theorem 6.3, Figure 8(A)]{CMP}, the 

    %{\color{blue} Another more geometric approach can be taken from \cite{CMP}. Indeed, a neighborhood of the Lagrangian resolution (i.e., of the Lagrangian handle) has a Legendrian lift in $(\R^{2\nu+1}_{(z,x,y)},dz+\beta$ that has a cusp-sum front in the $(z,y)$-space, \cite[Theorem 6.3, Figure 8(A)]{CMP}.
   % The $x$-coordinates are recovered by the slopes in the $y$-directions.
    %Now, we can slice this front along hyperplanes $y_1=\delta$. 
     %There are two singular values $\pm \epsilon$ when $y_1=\delta$ is tangent to the $S^{n-1}$-sphere of cusps. These correspond to moments where the intersection of $\{y_1=\delta\}$ with the Lagrangian resolution is singular. 
   % Passing the two singular values corresponds to the two different possible resolutions of intersections. 
    %In this model, the vector field $\partial_{y_1}$ is tangent to the front along $\{y_1=\delta=0\}$. That is, the intersection of the resolution with $\{y_1=0\}$ is Legendrian for $\beta$ in the $(x_2,\dots,x_n,y_1,\dots,y_n)$-coordinates and its front projection in the $(y_2,\dots,y_n)$ is a cusp sum picture at $s=1$ in Figure \ref{picture: surgery}.
    
   % Not quite correct, to be continued.}

    We picture, in the front projection, the composition of two Legendrian cobordisms from bottom to top: the Lagrangian $L_{[0,1]}$ corresponding to a connected cone-sum operation followed by a pinch cobordism $L_{[1,2]}$ along the disk $D$, whose union gives the resolution of $p_c$; see Figure \ref{picture: surgery}.  
%\marginpar{\cg Vincent, what is the left-hand arrow? {\color{blue} If you take a local Lagrangian projection of the Legendrian at $s=0$, you get a cross whose double point is $c$. There are two ways of resolving this crossing ($\pm 1$ resolutions in 2D). When looking at the Legendrian lifts of these resolutions, one of them leads to a local picture (in the front projection) at $s=-1$ of two cusps facing each others, and the level s=0 is obtained by Legendrian surgery along a middle disk. Here we consider the other resolution of the Lagrangian projection at $c$, whose Legendrian lift leads, in the front projection, to the picture at $s=1$, with a saddle type Lagrangian cobordism in between (I think). I think I had some sort of reference, that I cannot find back though.}}
\end{proof}

\begin{figure}[ht]
\begin{overpic}
[scale=0.8]{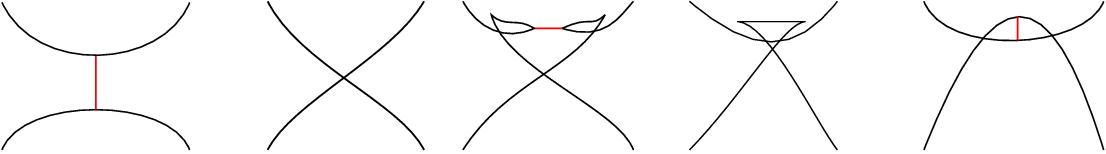}
\put(18,7){$\longleftarrow$}\put (59,7){$\longleftarrow$}\put(79,7){$\simeq$} \put(39,7){$\simeq$}
\put(91,8){$c'$}\put(10,6){$c$}\put(29,9){$\partial b_1$}\put(100,10){$\partial \widetilde{a}_1$}\put(100,2){$\partial a_1^{pw}$}
\put(-5,10){$\partial \widetilde{a}_1$}\put(-5,2){$\partial a_1$}\put(48,13){$D$}
\put(6,-2){$s=0$}\put(28,-2){$s=1$}\put(90,-2){$s=2$}
\end{overpic}
\vskip 0.5cm
\caption{Decomposing the Lagrangian resolution of $p_c$ into two cobordisms. The arrows indicate a surgery operation and $\simeq$ indicates the trace of a Legendrian isotopy.}
\label{picture: surgery}
\end{figure}
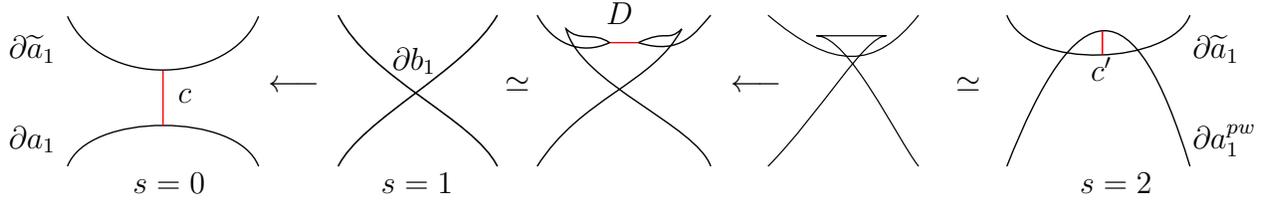

Let $L_1,\dots,L_k$ be Lagrangian submanifolds of $(W,d\beta)$, possibly with Legendrian boundary. As usual we denote by $\widehat L_i\subset \widehat W$ the extension of $L_i$ obtained by attaching the cylindrical end $[0,\infty)\times \bdry L_i$. Let $p_1 \in  \widetilde{a}_1 \cap L_1$, $p_2 \in L_1\cap L_2,\dots, p_k \in L_{k-1}\cap L_k$ and $p_{k+1}\in L_k \cap a_1$ be intersection points. 

Let $\mathcal{M}_{\widehat W} (p_{k+1},\dots p_1,c)$ be the moduli space of holomorphic maps $u:\dot F\to (\widehat W,\widehat J_W)$ (modulo domain automorphisms) for a $\widehat\beta$-adapted end-cylindrical almost complex structure $\widehat J_W$ such that:
\be
\item[(i)] $F$ is a compact Riemann surface of genus $g$ and one boundary component, $\dot F$ is $F$ with boundary punctures removed, and we are letting $\dot F$ vary;
\item[(ii)] $u$ is asymptotic to the intersection points $p_{k+1},\dots,p_1$ and the chord $c$ and $u$ maps $\bdry \dot F$ to the Lagrangians $\widehat{\widetilde{a}}_1$, $\widehat{L}_1,\dots,\widehat{L}_k$, $\widehat{a}_1$ in counterclockwise order.
\ee
Similarly, let $\mathcal{M}_{W_{[0,2]}} (p_{k+1},\dots p_{1},p_c)$ be the moduli space of holomorphic maps $v: \dot F\to (W_{[0,2]}, \widehat J_W|_{W_{[0,2]}})$ such that (i) above holds and:
\be
\item[(ii')] $v$ is asymptotic to the intersection points $p_{k+1},\dots,p_{1},p_c$ and $v$ maps $\bdry \dot F$ to the restrictions of the Lagrangians $\widehat{\widetilde{a}}_1$, $\widehat{L}_1,\dots,\widehat{L}_k$, $\widehat{a}^{pw}_1$ to $W_{[0,2]}$ in counterclockwise order.
\ee

\begin{lemma}\label{lemma: p_c to c}
If $c$ is sufficiently short, then for a suitable $\widehat{J}_W$ we have a diffeomorphism 
$$\mathcal{M}_{\widehat W} (p_{k+1},\dots p_1,c)\simeq \mathcal{M}_{W_{[0,2]}} (p_{k+1},\dots p_{1}, p_c).$$ 
\end{lemma}

\begin{proof} 
The proof is the same as that of \cite[Lemma 5.7.2]{CHT}, obtained by stretching $(\widehat W,\widehat{J}_W)$ along $\{0\} \times \partial W$. The diffeomorphism arises in the close-to-breaking situation and the corresponding $\widehat{J}_W$. By action and Fredholm index considerations, in the limit the curves from $p_c$ break into strips from $p_c$ to $c$ followed by curves from $c$.
\end{proof}

Let $\mathcal{M}_{W_{[0,2]}}^{\widehat{l}_1}(p_{k+1}',\dots,p_{1}')$, $\mathcal{M}_{W_{[0,2]}}^{\widehat{b}_1}(p_{k+1}',\dots,p_{1}')$ (resp.\ $\mathcal{M}_{\widehat W}^{\widehat{l}_1} (p_{k+1}',\dots,p_{1}')$, $\mathcal{M}_{\widehat W}^{\widehat{b}_1} (p_{k+1}',\dots,p_{1}')$) be the moduli space of holomorphic maps $\dot F\to ({W_{[0,2]}},\widehat{J}_W|_{W_{[0,2]}})$ (resp.\ $\dot F\to (\widehat W, \widehat{J}_W)$) such that (i) above holds and
\be
\item[(ii'')] $u$ maps $\bdry \dot F$ to the Lagrangians $\widehat{L}_1,\dots,\widehat{L}_k$, and $\widehat{l}_1$ or $\widehat{b}_1$ in counterclockwise order and is asymptotic to the intersection points $p_{k+1}',\dots,p_{1}'$, where $p_i'=p_i$ for $i\neq 1,k+1$ and $p_1'$ and $p_{k+1}'$ are obtained from $p_1$ and $p_{k+1}$ by a small isotopy when replacing $\widehat{a}_1\cup\widehat{\widetilde a}_1$ by $\widehat{l}_1$ or $\widehat{b}_1$. 
\ee
%The spaces $\mathcal{M}_W^\star(p_{k+1}',\dots,p_{1}')$ and $\mathcal{M}_{\widehat W}^\star (p_{k+1}',\dots,p_{1}')$ are diffeomorphic.

\begin{lemma}[Fukaya-Oh-Ohta-Ono~\cite{FO3}]\label{lemma: FO3}
For a suitable $\widehat{J}_W$, we have a diffeomorphism 
$$\mathcal{M}_{W_{[0,2]}} (p_{k+1},\dots p_{1}, p_c)\simeq \mathcal{M}_{W_{[0,2]}}^{\widehat{l}_1}(p_{k+1}',\dots,p_{1}').$$ 
\end{lemma}

\begin{proof} 
This corresponds to \cite[Theorem 55.7, Statement 55.8.3]{FO3}. 
\end{proof}

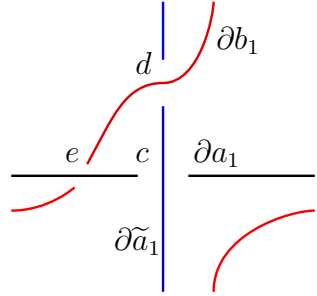
\begin{figure}[ht]
\begin{overpic}[scale=0.5]{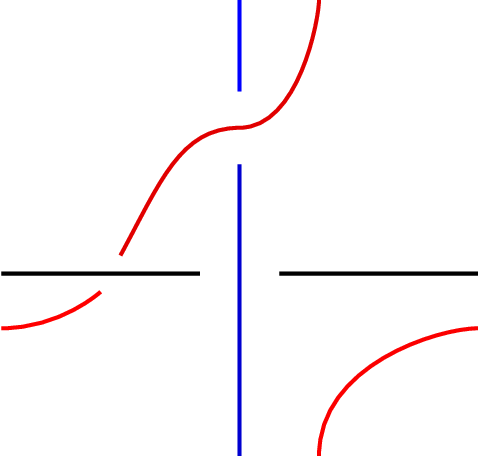}
\put(41,72){$d$} \put(41,43){$c$} \put(18,43){$e$}
\put(60,43){$\partial a_1$} \put(67.5,80){$\partial b_1$} \put(34,14){$\partial \widetilde{a}_1$}
\end{overpic}
\caption{The Legendrians $\bdry a_1$, $\bdry b_1$, and $\bdry \widetilde a_1$ in $\bdry W$ and the chords $c$, $d$, and $e$.}
\label{fig triangle}
\end{figure}

Next, stretching $W_{[0,2]}$ along $s=1$ we obtain:

\begin{lemma}\label{lemma: l to b} 
For a suitable $\widehat{J}_W$, we have a diffeomorphism 
$$\mathcal{M}_{W_{[0,2]}}^{\widehat{l}_1}(p_{k+1}',\dots,p_{1}')\simeq \mathcal{M}^{\widehat{b}_1}_{\widehat{W}} (p_{k+1}',\dots p_{1}').$$ 
\end{lemma}

The combination of Lemmas \ref{lemma: p_c to c}, \ref{lemma: FO3}, and \ref{lemma: l to b} gives:

\begin{cor}\label{cor: resolution} 
If $c$ is sufficiently short, then for suitable $\widehat{J}_W $ and $\widehat{J}_W'$, we have a diffeomorphism 
$$\mathcal{M}_{\widehat W} (p_{k+1},\dots p_{1}, c)\simeq \mathcal{M}_{\widehat{W}}^{\widehat{b}_1}(p_{k+1}',\dots p_{1}').$$
\end{cor}

Corollary~\ref{cor: resolution} is the key to showing that the cone of $c$ quasi-represents $b_1$ (Theorem \ref{thm: handleslide}). 
%We will see in the next subsection that it also works in the HDHF context.  In particular, when $k=1$ and $L_1$ is a deformation of $b_1$, the space of holomorphic triangles between $a_1,\widetilde{a}_1$ and $b_1$ and asymptotic to $c$, the chord $d$ from $\tilde a_1$ to $b_1$, and the chord $e$ from $a_1$ to $b_1$ is diffeomorphic to the space of strips between $b_1$ and its deformation: there is a count of $1$ triangle passing through a given point in a Lagrangian, e.g. $a_1$. 

\subsubsection{The handleslide theorem}

Let $\mathcal{R}^f(W)$ denote a larger $A_{\infty}$-category whose objects are $\widehat{a}_{\bf m}$ from all possible Lagrangian bases of $W$, i.e., $Ob(\mathcal{R}^f(W))$ is the union of objects of $\mathcal{R}^f(W,\beta,\phi)$ as we range over all $\phi$. Let $\overline{\mathcal{R}}^f(W)$ denote the idempotent completion of the triangulated envelope of $\mathcal{R}^f(W)$.

Let $(W,\beta_t,\phi_t)_{t\in [0,1]}$ be an elementary handle\-slide Weinstein homotopy such that the basis $\mathcal{B}_1$ of Lagrangian disks for $(\beta_1,\phi_1)$ is obtained by one handleslide from the basis $\mathcal{B}_0$ of $(\beta_0,\phi_0)$. In particular, let $a_1,\widetilde{a}_1\in \mathcal{B}_0$ and $b_1,\widetilde{a}_1\in \mathcal{B}_1$, where $b_1$ is obtained by handlesliding $a_1$ over $\widetilde{a}_1$ along the shortest Reeb chord $c$ (taken to be arbitrarily short) from $\bdry a_1$ to $\bdry \widetilde{a}_1$. {\em In what follows we freely identify Reeb chords with their corresponding intersection points in the wrapped setting.}

The main step in proving the handleslide theorem is to show that any collection of disjoint Lagrangian disks for $(W,\beta_1,\phi_1)$ is quasi-isomorphic to a twisted complex of those for $(W,\beta_0,\phi_0)$ in $\overline{\mathcal{R}}^f(W)$. 
To this end, consider a collection of disjoint Lagrangian disks ${\bf b}=\{b_1,a_2'' ,a_3'',\dots ,a_k''\}$, where $a_i''$ are from the basis $\mathcal{B}_0 \cup \mathcal{B}_1$. Let
$${\bf a}=\{a_1,a_2 ,a_3,\dots ,a_k\} \quad\mbox{ and } \quad \widetilde{\bf a}=\{\widetilde{a}_1,a_2' ,a_3',\dots,a_k'\},$$
where $a_i$, $a_i'$, $a_i''$ are Hamiltonian isotopic so that:
\be
\item any two of $a_i$, $a_i'$, $a_i''$ intersect once in the interior; and
\item $\bdry a_i'$ (resp.\ $\bdry a_i''$) is obtained from a time-$\varepsilon$ (resp.\ $2\varepsilon$) Reeb flow of $\bdry a_i$ for small $\varepsilon>0$ and applying a small perturbation.
\ee

%Let ${\bf a}=\{a_1 ,\dots,a_k\}$ and ${\bf d}=\{d_1 ,\dots,d_k\}$ be two collections of disjoint Lagrangian disks for $(W,\beta_0,\phi_0)$, where $a_1$ is as above and $a_i, d_i\in \mathcal{B}_0$. Let
%let $\widetilde{a}_1$ be a Lagrangian disk which is disjoint from ${\bf a}$, and let $c$ be the shortest Reeb chord from $\partial a_1$ to $\partial \widetilde{a}_1$ (we assume the shortest Reeb chord is arbitrarily short).  
%Let $b_1$ be the Lagrangian disk obtained by handle\-sliding $a_1$ over $\widetilde{a}_1$ along $c$, taken to be disjoint from $a_2,\dots, a_{k}$. 
%$${\bf b}=\{b_1,a_2' ,a_3',\dots ,a_k'\} \quad\mbox{ and } \quad \widetilde{\bf a}=\{\widetilde{a}_1,a_2'' ,a_3'',\dots,a_k''\},$$ where $a_i'$ and $a_i''$ are obtained from $a_i$ by a Hamiltonian isotopy so that:
%\be
%\item any two of $a_i$, $a_i'$, $a_i''$ intersect once in the interior; and
%\item $\bdry a_i'$ (resp.\ $\bdry a_i''$) is obtained from a time-$\varepsilon$ (resp.\ $2\varepsilon$) Reeb flow  of $\bdry a_i$ for small $\varepsilon>0$ and perturbing.
%\ee

For $i\geq 2$ and $\nu\geq 2$, there are two Reeb chords from $\partial a_i$ to $\partial a_i'$ and we denote the {\em longer chord of lower cohomological degree}
by $\theta_{a\widetilde{a},i}^{\op{bot}}$. (In the $\nu=1$ case we take the sum of the two Reeb chords.)  %(In the wrapped setting there is an $(n-1)$-dimensional family of holomorphic bigons from $\theta_{ab,i}^{\op{bot}}$ to $\theta_{ab,i}^{\op{top}}$ with boundary on $a_i \cup a_i'$.) 
Similarly, we define chords $\theta_{\widetilde{a}b,i}^{\op{bot}}$ from $\partial a_i'$ to $\partial a_i''$ and $\theta_{ab,i}^{\op{bot}}$ from $\partial a_i$ to $\partial a_i''$ with the same properties. We denote the interior intersection points $a_i\cap a_i'$, $a_i'\cap a_i''$, and $a_i\cap a_i''$ by $\theta_{a\widetilde{a},i}^{\op{top}}$, $\theta_{\widetilde{a}b,i}^{\op{top}}$, and $\theta_{ab,i}^{\op{top}}$.

Recalling the notation $c=c_{1/2-\epsilon}$ and $d=c_{1/2+\epsilon}$ from Section~\ref{subsub: model1}, we define 
%\marginpar{\cg maybe $\Theta_e$ is from $b$ to $a$ and $\Theta_d$ is from $\tilde a$ to $b$.  For the maps $\Phi$ and $\Psi$ below they need to be changed.  I'm a little confused about which Legendrians are above which, and whether that matters.}
\begin{align*}
\Theta_d &= \{d, \theta_{b\widetilde{a},2}^{\op{bot}} ,\theta_{b\widetilde{a},3}^{\op{bot}} ,\dots ,\theta_{b\widetilde{a},k}^{\op{bot}}\}\in \Hom (\widetilde {\bf a}, {\bf b}),\\
\Theta_{c} & =\{c,\theta_{a\widetilde{a},2}^{\op{bot}} ,\theta_{a\widetilde{a},3}^{\op{bot}} ,\dots,\theta_{a\widetilde{a},k}^{\op{bot}}\}\in \Hom ({\bf a},\widetilde {\bf a}),\\
\Theta_e& =\{e, \theta_{ab,2}^{\op{top}} ,\theta_{ab,3}^{\op{top}} ,\dots ,\theta_{ab,k}^{\op{top}}\}\in \Hom({\bf b},{\bf a}),
\end{align*}
where $e$ is the unique chord from $\bdry a_1$ to $\bdry b_1$ as in Figure~\ref{fig triangle}.  The tuples $\Theta_c$, $\Theta_e$, and $\Theta_d$ are cycles since $c,e, d$ can be made arbitrarily short. %The class $\Theta_{c}$ induces a chain map
%$$m^l_c:\Hom (\widetilde{h}({\bf d}),{\bf a}) \to \Hom(\widetilde{h}({\bf d}), \widetilde{\bf a}), \quad {\bf y}\mapsto \mu_2 (\Theta_c,{\bf y}).$$

\begin{thm}\label{thm: handleslide}
The object $\bf b$ is quasi-isomorphic to $\op{Cone}(\Theta_c)$ in $\overline{\mathcal{R}}^f(W)$.
%\marginpar{\cg changed statement of theorem; we still need figure out what to say about bimodules; also, is the statement of the theorem the correct way of saying what I want?}
\end{thm}

\begin{proof}[Sketch of proof.]

The proof is similar to that of Seidel's exact triangle \cite{Se1} and we only give a sketch. 

We define two $A_\infty$-module maps $\Psi=(\Psi_{b})_{b}$ and $\Phi=(\Phi_{b})_{b}$, given by counts of holomorphic curves in some symplectic fibrations.
\begin{align*}
    \Psi_{b} = (\Psi_{b,1},\Psi_{b,2}): \Hom(\widetilde{h}(\wh a_{{\bf m}^b}),{\bf b}) &\otimes \Hom (\widetilde{h}(\wh a_{{\bf m}^{b-1}}),\widetilde{h}(\wh a_{{\bf m}^b}))\otimes\dots\otimes \Hom(\widetilde{h}(\wh a_{{\bf m}^0}),\widetilde{h}(\wh a_{{\bf m}^1})) \\
    & \to \Hom (\widetilde{h}(\wh a_{{\bf m}^0}), {\bf a})\oplus \Hom(\widetilde{h}(\wh a_{{\bf m}^0}), \widetilde{\bf a}),\\
    ({\bf y},{\bf z}_b,\dots,{\bf z}_1) \mapsto (\mu_{b+2} (\Theta_e, {\bf y},&{\bf z}_b,\dots,{\bf z}_1), \mu_{b+3} (\Theta_c, \Theta_e, {\bf y},{\bf z}_b,\dots,{\bf z}_1)),
\end{align*}
\begin{align*}
\Phi_{b} =\Phi_{b,1}\oplus &\Phi_{b,2} :  (\Hom (\widetilde{h}(\wh a_{{\bf m}^b}), {\bf a})\oplus \Hom(\widetilde{h}(\wh a_{{\bf m}^b}), \widetilde{\bf a})) \\
&  \otimes \Hom (\widetilde{h}(\wh a_{{\bf m}^{b-1}}),\widetilde{h}(\wh a_{{\bf m}^b}))\otimes\dots\otimes \Hom(\widetilde{h}(\wh a_{{\bf m}^0}),\widetilde{h}(\wh a_{{\bf m}^1})) \to \Hom(\widetilde{h}(\wh a_{{\bf m}^0}),{\bf b}),\\
(({\bf y}',{\bf y}''),&{\bf z}_b,\dots,{\bf z}_1) \mapsto\mu_{b+3} (\Theta_d, \Theta_c, {\bf y}',{\bf z}_b,\dots,{\bf z}_1) + \mu_{b+2} (\Theta_d, {\bf y}'',{\bf z}_b,\dots,{\bf z}_1).
\end{align*}
In these formulas, the families $\wh a_{{\bf m}^i}$ are taken over all possible bases of Lagrangians in $\mathcal{R}^f(W)$.

The fact that $\Psi$ and $\Phi$ are $A_\infty$-module maps is a classical argument based on the analysis of the boundaries of one-dimensional moduli spaces
\begin{gather*}
\mathcal{M}^{\op{ind}=1}(\Theta_e, {\bf y},{\bf z}_b,\dots,{\bf z}_1;\mathbf{y}'),\quad  \mathcal{M}^{\op{ind}=1}(\Theta_c, \Theta_e, {\bf y},{\bf z}_b,\dots,{\bf z}_1;\mathbf{y}''), \\
\mathcal{M}^{\op{ind}=1}(\Theta_d, \Theta_c, {\bf y}',{\bf z}_b,\dots,{\bf z}_1;\mathbf{y}),\quad  \mathcal{M}^{\op{ind}=1}(\Theta_d, {\bf y}'',{\bf z}_b,\dots,{\bf z}_1; {\bf y}),
\end{gather*}
and observing that $\Theta_d$, $\Theta_c$ and $\Theta_e$ are closed.

The proof that $\Phi$ and $\Psi$ are quasi-isomorphisms is also a classical argument. The composition of the maps is quasi-isomorphic to the count of curves in the concatenation of cobordisms.  We consider $\Phi_0\circ \Psi_0$ for simplicity; the case for higher $b$ is analogous.  For this, we analyze the boundary of the moduli space $\mathcal{M}^{\op{ind}=1,\chi}(\Theta_d, \Theta_c,\Theta_e, {\bf y};{\bf y}')$. Since $\Theta_d$,  $\Theta_c$ and $\Theta_e$ are closed, $\bdry \mathcal{M}^{\op{ind}=1,\chi}(\Theta_d, \Theta_c,\Theta_e, {\bf y};{\bf y}')$ consists of the following:
\begin{gather*}
    \mathcal{M}^{\op{ind}=0,\chi_1}(\Theta_d, {\bf y}'';{\bf y}')\times \mathcal{M}^{\op{ind}=0,\chi_2}(\Theta_c,\Theta_e, {\bf y};{\bf y}''),\\
    \mathcal{M}^{\op{ind}=0,\chi_1}(\Theta_d, \Theta_c, {\bf y}'';{\bf y}')\times \mathcal{M}^{\op{ind}=0,\chi_2}(\Theta_e, {\bf y};{\bf y}''),\\
    \mathcal{M}^{\op{ind}=0,\chi_1}({\bf y}'';{\bf y}')\times \mathcal{M}^{\op{ind}=0,\chi_2}(\Theta_d, \Theta_c,\Theta_e, {\bf y};{\bf y}''),\\
    \mathcal{M}^{\op{ind}=0,\chi_1}(\Theta_d, \Theta_c,\Theta_e, {\bf y}'';{\bf y}')\times \mathcal{M}^{\op{ind}=0,\chi_2}({\bf y};{\bf y}''),\\
    \mathcal{M}^{\op{ind}=0,\chi_1}({\bf y}'', {\bf y};{\bf y}')\times \mathcal{M}^{\op{ind}=0,\chi_2}(\Theta_d, \Theta_c,\Theta_e;{\bf y}''),
\end{gather*}
where $\chi=\chi_1+\chi_2$.
The first two correspond to compositions $\Phi_{0,1}\circ \Psi_{0,2}$ and $\Phi_{0,2} \circ \Psi_{0,1}$, respectively, and the next two give rise to chain homotopy terms. We analyze the last term. By Corollary~\ref{cor: resolution}, the term ${\bf y}''$ in $\mathcal{M}^{\op{ind}=0,\chi_2}(\Theta_d, \Theta_c,\Theta_e;{\bf y}'')$ must be the cohomological unit and the algebraic count must be one.  Then ${\bf y}={\bf y}'$ in $\mathcal{M}^{\op{ind}=0,\chi_1}({\bf y}'', {\bf y};{\bf y}')$.  Hence the last term gives the identity map. Since similar considerations hold for the other compositions, both compositions $\Phi \circ \Psi$ and $\Psi \circ \Phi$, and hence $\Phi$ and $\Psi$, are quasi-isomorphisms. 
\end{proof}

%{\cg On the other hand, we have a map\marginpar{\cg revise this in view of above change}
%$$m^r_c : \Hom (\widetilde{h}(\widetilde{\bf a}),{\bf d})\to \Hom (\widetilde{h}({\bf a}), {\bf d}),$$
%given by right multiplication by $\Theta_c$. In this case, $ \Hom (\widetilde{h}({\bf b}),{\bf d})$ is quasi-isomorphic to $\op{Cone}(m^r_c)$ by an analog of Theorem~\ref{thm: handleslide}.

%Then by Theorem~\ref{thm: handleslide} the map $\mathcal{R}^f(W,\beta_1,\phi_1)\to \overline{\mathcal{R}}^f(W,\beta_0,\phi_0)$ induced by $\Psi$, where we replace the Lagrangian $b$ by the cone of $m^r_c$, is an $\ai$-morphism that is homologically full and faithful, and the objects in its image generate $\overline{\mathcal{R}}^f(W,\beta_0,\phi_0)$. By \cite[Lemma 3.34]{Se3}, it induces a quasi-equivalence between $\overline{\mathcal{R}}^f(W,\beta_1,\phi_1)$ and $\overline{\mathcal{R}}^f(W,\beta_0,\phi_0)$. Similar considerations hold for the $\overline{\mathcal{R}}^f(W,\beta_1,\phi_1)$-bimodule $\overline{\mathcal{B}}^f(W,\beta_1,\phi_1;h)$ and the $\overline{\mathcal{R}}^f(W,\beta_0,\phi_0)$-bimodule $\overline{\mathcal{B}}^f(W,\beta_0,\phi_0;h)$. Finally, since every isotopy can be decomposed into elementary ones, Theorem \ref{thm: invariance under handleslides} follows.}

Given the two Lagrangian bases $\mathcal{B}_i$, $i=0,1$, that are related by one handleslide as above, let $R_i$ denote the $A_{\infty}$-algebra $R^f(W,\beta_i,\phi_i)$. Let $R_i-\op{Mod}$ denote the $A_\infty$-category of left $R_i$-modules and let $\op{Perf} R_i$ denote its full subcategory of perfect modules. It is known that  $\op{Perf} R_i$ is the idempotent completion of the triangulated envelope of $R_i$.
By iteratively using Theorem \ref{thm: handleslide}, we obtain the following:

\begin{cor} \label{cor: perfect modules}
For any Lagrangian $\widehat{a}_{{\bf m}_1}$ constructed from the basis $\mathcal{B}_1$ with $|{\bf m}_1|=k$, the $R_0$-module $\oplus_{|{\bf m}'_0|=k} \Hom(\widehat{a}_{{\bf m}_1},\widehat{a}_{{\bf m}'_0})$ is isomorphic to an object of $\op{Perf} R_0$ that we denote by $\mathcal{F}(\widehat{a}_{{\bf m}_1})$. Moreover, $\widehat{a}_{{\bf m}_1}$ and $\mathcal{F}(\widehat{a}_{{\bf m}_1})$ are quasi-isomorphic in $\overline{\mathcal{R}}^f(W)$. 
\end{cor}

\begin{proof}[Proof of Theorem \ref{thm: invariance under handleslides}]
%{\cg (Part I: the categories) Steps: 

%1. Modify $\overline{\mathcal{R}}^f(W,\beta_i,\phi_i)$ to be the idempotent completion of the triangulated envelope of$R^f(W,\beta_i,\phi_i)$, i.e., $\overline{\mathcal{R}}^f(W,\beta_i,\phi_i):=\op{Perf} R_i$. Note that $\op{Mod} R_i$ is the larger category of all $R_i$-modules. 

%2. Tensoring with the bimodule $M_{10}$ gives a functor $\Psi_{10}: \op{Mod} R_1 \to \op{Mod} R_0$.

%3. The key is to show that the functor $\Psi_{10}$ preserves the subcategories $\op{Perf} R_i$. (We do not have this property for the bimodule associated to the monodromy $h$.)

%4. Moreover, we need $\Psi_{10}(L_1) \cong L_1$ in the ambient Fukaya category $\overline{\mathcal{R}}^f(W)$ for any object $L_1$ in $\op{Perf} R_1$.}

%A technical result we need is Lemma A.2 in \cite{GPS1}.}

%Given two Lagrangian bases $\mathcal{B}_i$, $i=0,1$, let $R_i$ denote the $A_{\infty}$-algebra $R^f(W,\beta_i,\phi_i)$. Let $\op{Mod} R_i$ denote the $A_\infty$-category of left $R_i$-modules, and $\op{Perf} R_i$ denote its full subcategory of perfect modules. It is known that  $\op{Perf} R_i$ is the idempotent completion of the triangulated envelope of $R_i$.

Given the two Lagrangian bases $\mathcal{B}_i$, $i=0,1$, we define
\begin{gather*}
M_{10}:= \oplus_{k\geq 0} \oplus_{|{\bf m}_1|=|{\bf m}'_0|=k} \Hom(\widehat{a}_{{\bf m}_1},\widehat{a}_{{\bf m}'_0}),
\end{gather*}
where $\widehat{a}_{{\bf m}_1}$ and $\widehat{a}_{{\bf m}'_0}$ range over all $k$-tuples of disjoint Lagrangians constructed from the basis $\mathcal{B}_1$ and $\mathcal{B}_0$, respectively.  
Then $M_{10}$ is an $A_{\infty}$ $(R_0,R_1)$-bimodule. Taking the $A_\infty$-tensor product with $M_{10}$ induces a functor $\Psi_{10}: R_1-\op{Mod}  \to R_0-\op{Mod} $, which, on the level of objects, is given by $N\mapsto M_{10}\otimes_{R_1}N$.  Here $\otimes_{R_1}$ indicates the $A_\infty$-tensor product: as a vector space $M_{10}\otimes_{R_1}N=M_{10}\otimes_{\F} TR_1\otimes_{\F} N$, where $TR_1=\oplus_{i=1}^\infty R_1^{\otimes i}$, i.e., we are inserting the bar resolution of $R_1$, and the chain map is given by $A_\infty$-contractions. See for example \cite[Section 2.5]{Ga} for a detailed exposition.

For any Lagrangian $\widehat{a}_{{\bf m}_1}$ constructed from $\mathcal{B}_1$, let 
$$M(\widehat{a}_{{\bf m}_1})=\oplus_{|{\bf m}'_1|=k} \Hom(\widehat{a}_{{\bf m}_1},\widehat{a}_{{\bf m}'_1})$$ 
denote the corresponding Yoneda $R_1$-module. Then there is a natural morphism 
\begin{equation}\label{eqn: crushing}
\Psi_{10}(M(\widehat{a}_{{\bf m}_1}))= M_{10}\otimes_{\F} TR_1\otimes_{\F} M(\widehat{a}_{{\bf m}_1})  \to \oplus_{|{\bf m}'_0|=k} \Hom(\widehat{a}_{{\bf m}_1},\widehat{a}_{{\bf m}'_0})
\end{equation}
of $R_0$-modules which is given by $A_\infty$-contractions in the ambient category $\overline{\mathcal{R}}^f(W)$.

We claim that \eqref{eqn: crushing} is a quasi-isomorphism.  Following the proof of \cite[Proposition 2.2]{Ga}, consider the spectral sequence associated to the length filtration. The map on the first page reduces to the analogous map for ordinary algebras and is a quasi-isomorphism by the proof similar to the exactness of the bar resolution.  The isomorphism on the $E^1$-page then descends to an isomorphism on the $E^\infty$-level.

Moreover, $\Psi_{10}(M(\widehat{a}_{{\bf m}_1}))$ is quasi-isomorphic to $\mathcal{F}(\widehat{a}_{{\bf m}_1})$ in $\op{Perf} R_0$ by Corollary \ref{cor: perfect modules}. Hence $\Psi_{10}$ preserves the subcategories of perfect modules $\Psi_{10}: \op{Perf} R_1 \to \op{Perf} R_0$. 

Each $\op{Perf} R_i$ is a full subcategory of the ambient category $\overline{\mathcal{R}}^f(W)$. 
We have $\Psi_{10}(M(\widehat{a}_{{\bf m}_1})) \cong M(\widehat{a}_{{\bf m}_1})$ in $\overline{\mathcal{R}}^f(W)$ by Corollary \ref{cor: perfect modules}.  Hence $\Psi_{10}$ is fully faithful: 
$$\Hom(\Psi_{10}(M(\widehat{a}_{{\bf m}_1})),\Psi_{10}(M(\widehat{a}_{{\bf m}'_1}))) \cong \Hom(\Psi_{10}(M(\widehat{a}_{{\bf m}_1})),M(\widehat{a}_{{\bf m}'_1})) \cong \Hom(M(\widehat{a}_{{\bf m}_1}),M(\widehat{a}_{{\bf m}'_1})).$$
One can similarly define an $A_{\infty}$ $(R_1,R_0)$-bimodule $M_{01}$ which induces a functor $\Psi_{01}: \op{Perf} R_0 \to \op{Perf} R_1$. For any object $M(\widehat{a}_{{\bf m}_0})$ in $\overline{\mathcal{R}}_0$, we have 
$$M(\widehat{a}_{{\bf m}_0}) \cong \Psi_{01}(M(\widehat{a}_{{\bf m}_0})) \cong \Psi_{10}(\Psi_{01}(M(\widehat{a}_{{\bf m}_0}))),$$ 
in $\overline{\mathcal{R}}^f(W)$. Hence $\Psi_{10}$ is essentially surjective.  Since $\Psi_{10}$ is fully faithful and essentially surjective, $\Psi_{10}: \op{Perf} R_1 \to \op{Perf} R_0$ is a quasi-equivalence.

Recall the $R_i$-bimodule $B^f(W,\beta_i,\phi_i;h)$, $i=0,1$, from \eqref{eq bimodule for h}  associated to the symplectomorphism $h$. 
The goal is to show that there is a quasi-isomorphism of $(R_0,R_1)$-bimodules:
$$B^f(W,\beta_0,\phi_0;h) \otimes_{R_0} M_{10} \cong M_{10} \otimes_{R_1} B^f(W,\beta_1,\phi_1;h),$$
where the tensor products are still $A_\infty$-tensor products. 
For each summand $k$, the left-hand side is $$\oplus_{|{\bf m_0}|=|{\bf m_0'}|=|{\bf m_1}|=k} \Hom(\widetilde h(\wh a_{\bf m_0'}),\wh a_{\bf m_0}) \otimes_{R_0} \Hom(\widetilde h(\widehat{a}_{{\bf m}_1}), \widetilde h(\widehat{a}_{{\bf m}'_0})).$$
This tensor product is quasi-isomorphic to $$\oplus_{|{\bf m_0}|=|{\bf m_1}|=k} \Hom(\widetilde h(\wh a_{\bf m_1}),\wh a_{\bf m_0}),$$
by a proof similar to that of \cite[Proposition 2.2]{Ga}.  
%The basic idea is to view $M_{10}$ as the diagonal bimodule, and consider the bar complex of $R_0$. The $A_{\infty}$-map is given by the $A_{\infty}$-operations. 
There is a similar quasi-isomorphism for the right-hand side. This concludes the proof of Theorem~\ref{thm: invariance under handleslides}.
\end{proof}
%For any object $\widehat{a}_{{\bf m}_1}$ in $\overline{\mathcal{R}}_1$, we have 
%$$\Psi_{10}(\mathcal{F}_1(\widehat{a}_{{\bf m}_1})) \cong \mathcal{F}_1(\widehat{a}_{{\bf m}_1}) \cong \mathcal{F}_0(\Psi_{10}(\widehat{a}_{{\bf m}_1})),$$
%in $\overline{\mathcal{R}}^f(W)$. A similar argument on morphisms shows that $\Psi_{10}\circ \mathcal{F}_1 \cong \mathcal{F}_0 \circ \Psi_{10}$. We conclude that the bimodules $\overline{\mathcal{B}}^f(W,\beta_i,\phi_i;h)$ are invariant up to quasi-equivalence.

\part{Symplectic Khovanov homology of a transverse link in a fibered $3$-manifold}

\section{The Lefschetz fibration $\Sigma =W_{S,n}$}\label{section: Lefschetz fibration}

Let $S$ be a bordered surface of genus $g$ and ${\bf x}=\{x_1,\dots, x_n\}\subset \op{int}(S)$ that we call the set of {\em marked points}.

\subsection{Definition of $W_{S,n}$}

We define the Lefschetz fibration $\pi : W_{S,n} \to S$ with $n$ critical values $x_1 ,\dots ,x_n$ and regular fiber $(T^*S^1 , pdq)$.   It is a fibered sum, or more precisely a Liouville connected sum $(W_{S,n},\lambda)$, of two exact symplectic fibrations: the $2$-dimensional $A_{n}$-Milnor fiber and a trivial fibration $(T^* S^1 \times S ,p dq +\beta)$ over $S$, where $\beta$ is a Liouville form on $S$. The $2$-dimensional $A_{n}$-Milnor fiber is given by
$$\{ z_1^2 +z_2^2 +P_{n+1} (z_3)=0\} \subset \C^3,$$
where $P_{n+1}$ is a degree $n+1$ polynomial with simple roots, and admits a Lefschetz fibration over $\C$ via the map $(z_1,z_2,z_3)\mapsto z_3$, whose critical values are the roots of $P_n$. We restrict this fibration to the unit disk $D^2$ which we assume contains all the roots of $P_n$. The regular fiber is the cylinder $T^* S^1$ and the holonomy around one singular fiber is given by a positive symplectic Dehn twist about the $0$-section of $T^* S^1$; see \cite{Se1} for more details.
%It is naturally a symplectic fibration for the form $\sum dz_i \wedge d\overline{z}_i$. See 

\begin{notation}
    In what follows we abuse notation and write $\pi:(W_{S,n},\lambda)\to S$ for the {\em compact} Lefschetz fibration with regular fiber $(A:=S^1_q\times[-1,1]_p,pdq )$. At this point, the manifold $W_{S,n}$ has corners. We can round the corners of $W_{S,n}$ to turn it into a Weinstein domain; we further abuse notation and denote either the cornered or smooth version by $W_{S,n}$.
\end{notation}

\subsection{Contact structure on $\bdry W_{S,n}$}

The boundary $\bdry W_{S,n}$ is diffeomorphic to an $S^1$-bundle of Euler class $n$ over the surface $D(S)$ obtained by doubling $S$, and decomposes into the horizontal and a vertical (i.e., tangent to the fibers) parts $\bdry_h W_{S,n}$ and $\bdry_v W_{S,n}$, giving it the structure of a {\em spinal open book decomposition} of Lisi, Van Horn-Morris, and Wendl~\cite{LVW}. The pages are annular fibers $A$ with monodromy a product of $n$ positive Dehn twists along the $0$-section. The horizontal boundary $\bdry_h W_{S,n}$ is the generalized binding neighborhood consisting of two copies of $S^1 \times S$ (generalizing the binding neighborhood $S^1 \times D^2$ in the open book case). We identify 
$$\bdry_v W_{S,n}\simeq T^2_{x,q}\times[-1,1]_p=S^1_x\times S^1_q\times[-1,1]_p.$$

Let $\xi=\ker \lambda$ be the induced contact structure on $\bdry W_{S,n}$ after rounding the corners. We may take $\lambda$ such that:
\be
\item[(i)] the Reeb vector field $R_\lambda$ on $\bdry_h W_{S,n}$ is tangent to the $S^1$-fibers (in particular $\xi|_{\bdry_h W_{S,n}}$ is $S^1$-invariant and transverse to the $S^1$-fibers);
\item[(ii)]  $\lambda|_{\bdry_v W_{S,n}}=f(p)dx -g(p) dq$ for some functions $f,g:[-1,1]\to \R$ which depend on $n$ such that $(f(0),g(0))=(1,0)$, $(-g,f)\cdot (f',g')>0$, and $R_\lambda$ is parallel to $f'\bdry_q+g'\bdry_x$ and rotates from $\bdry_q$ to $-\bdry_q$ in a clockwise manner  as $p$ goes from $-1$ to $1$.
\ee
The torus $\{ p =0\}$ is linearly foliated by Legendrian circles $\{ x=\mbox{const}\}$.

\subsection{A basis of Lagrangian disks and cylinders on $W_{S,n}$}

\begin{defn}
A {\em stop} is a possibly empty finite subset $\tau \subset  \partial S$.  
%We assume that it is either empty or has nontrivial intersection with each component of $\bdry S$.
\end{defn}

\begin{defn} \label{defn: parametrization}
Assuming $\tau$ is either empty or has nontrivial intersection with each component of $\bdry S$, a {\em parametrization} of $(S,{\bf x},\tau)$ is a collection ${\bf a} =\{a_1,\dots,a_{n+s}\}$ of pairwise disjoint arcs and half-arcs of $S$, where:
\be
\item $a_1,\dots,a_n$ are half-arcs connecting $x_i$ to $\bdry S$;
\item $a_{n+1},\dots, a_{n+s}$ are properly embedded arcs;
\item if $\tau\not=\varnothing$, then $S\setminus\cup_{i=n+1}^{n+s} a_i$ is a union of polygons, each of which contains a single element of $\tau$;
\item if $\tau=\varnothing$,  then $S\setminus\cup_{i=n+1}^{n+s} a_i$ is a single polygon.
\ee
\end{defn}

Parallel transporting the vanishing cycle at $x_i$, $1\leq i\leq n$, along $a_i$ using the symplectic connection gives a properly embedded Lagrangian disk $L_i$ with Legendrian boundary $\Lambda_i$ in $W_{S,n}$, called a {\it Lagrangian thimble}. The Legendrian boundary $\Lambda_i$ is the $0$-section of the fiber over $\partial a_i \cap \partial S$, contained in $\partial_v W_{S,n}$. Over the arc $a_i$, $i>n$, we parallel transport the $0$-section of the fiber to form a Lagrangian annulus $L_i$ with Legendrian boundary $\Lambda_i$ (which consists of two circles).

The collection $\{L_1,\dots, L_{n+s}\}$ of Lagrangian disks and cylinders on $W_{S,n}$ will be called a {\em Lagran\-gian basis}.

\subsection{Description of the Reeb chords}

We briefly relate the Reeb chords between $\Lambda_i$ and $\Lambda_j$ to the Reeb chords between $a_i$ and $a_j$.  If $1\leq i,j\leq n+s$, then over each chord $c$ from $\partial a_i$ to $\partial a_j$, there is an $S^1$-family $S^1(c)$ of chords from $\Lambda_i$ to $\Lambda_j$ since $\Lambda_i$ and $\Lambda_j$ are the $0$-sections of the fibers over $\partial a_i \cap \partial S$ and $\partial a_j \cap \partial S$. Applying Morse-Bott theory, we may perturb $S^1(c)$ into two chords $\check c$ and $\hat c$ corresponding to the maximum and the minimum of the Morse function on $S^1$ with lower and higher cohomological degree, respectively.

\section{The $\ai$-algebras $R^p(S,n,{\bf a};k)$ and $R^f(S,n,{\bf a};k)$} \label{section: A infty algebras}

\subsection{Definitions}

Let $S$ be a bordered surface, let ${\bf x}=\{x_1,\dots,x_n\}$ be the set of marked points on $S$, and let $\tau$ be a stop which is either empty or has nontrivial intersection with each component of $\bdry S$. We consider a parametrization ${\bf a}=\{a_1,\dots,a_{n+s}\}$ of $(S,{\bf x},\tau)$.

We will take parallel copies of arcs $a_i$, $i> n$, {\em but not of half-arcs.} Given a nonnegative integer $k$, let ${\bf m}=(m_1,\dots,m_{n+s})$ be a partition of $k$, i.e., $k=m_1+\dots+m_{n+s}$, subject to the condition that $m_i=0$ or $1$ whenever $i\leq n$. We write $|{\bf m}|= \sum_{i=1}^{n+s} m_i=k$. Let $a_{\bf m}$ be the union of $k$ disjoint arcs/half-arcs consisting of $m_i$ (nonintersecting) copies of $a_i$ and let $\widehat{L}_{\bf m}$ be the collection of Lagrangian disks and cylinders sitting above the cylindrical completion $\widehat{a}_{\bf m}$ of $a_{\bf m}$ in $S\cup ([0,\infty)\times \bdry S)$. We sometimes write $a_{\bf m}$ multiplicatively as $\prod_{i=1}^{n+s}a_i^{m_i}$; then expressions such as $a_i\mid a_{\bf m}$ naturally make sense.

Next we define $\widehat{a}^\tau_{\bf m}$, the result of applying $\tau$-wrapping to $\widehat{a}_{\bf m}$, as follows: If $\tau$ nontrivially intersects each component of $\bdry S$, then wrap $\widehat{a}_{\bf m}$ in the positive direction on $[0,\infty)\times \bdry S$ until all their ends come close to the stop $[0,\infty)\times \tau$.  If $\tau=\varnothing$, then apply full wrapping to $\widehat{a}_{\bf m}$.  Then let $\widehat{L}_{\bf m}^\tau$ be the collection of Lagrangians sitting above $\widehat{a}^\tau_{\bf m}$ and defined in a manner analogous to $\widehat{L}_{\bf m}$. 

We then define
\begin{gather*}
R^\tau(S,n,{\bf a};k)=\oplus_{|{\bf m}|= |{\bf m'}|=k}\op{Hom} (\widehat{L}_{\bf m},\widehat{L}_{\bf m'})= \oplus_{|{\bf m}|= |{\bf m'}|=k} \widehat{CF}(\widehat{L}_{\bf m}^\tau,\widehat{L}_{\bf m'}),\\
R^\tau(S,n,{\bf a}) = \oplus_{k\geq 0} R^\tau(S,n,{\bf a};k),
\end{gather*}
where $\widehat{CF}$ is the usual HDHF chain complex as in \cite[Section 5.6 and Notation 5.6.1]{CHT}.

Viewing $R^\tau(S,n,{\bf a})$ as an $\ai$-category $\mathcal{R}^\tau(S,n,{\bf a})$ whose objects are $\widehat{L}_{\bf m}$ (or interchangeably $a_{\bf m}$) and whose morphisms are $\Hom (\widehat{L}_{\bf m},\widehat{L}_{\bf m'})$, we write $\overline {\mathcal{R}}^\tau (S,n,{\bf a})$ for its triangulated envelope (i.e., the smallest triangulated $\ai$-category containing $\mathcal{R}^\tau(S,n,{\bf a})$).

\subsection{Proof of Theorem~\ref{thm: invariance}} \label{subsection: proof of theorem}

The goal of this subsection is to prove Theorem~\ref{thm: invariance} for $\star=\tau$.

There are several cases of handleslides to consider:
\be
\item half-arcs over half-arcs;
\item arcs over half-arcs;
\item half-arcs over arcs;
\item arcs over arcs.
\ee
Any two parametrizations are related by a finite sequence of handleslides of type (1)--(4). 

We give a brief outline of the proof before discussing the individual cases in Sections~\ref{subsubsection: case 1}--\ref{subsubsection: case 2}.

Cases (1) and (2) can both be split into two subcases. This extra complexity is due to the fact that we insist on having at most one copy of each half-arc in our objects $a_{\bf m}$, whereas having a proof similar to Cases (3) and (4) only involving a version of Theorem \ref{thm: invariance under handleslides} would sometimes require us to have multiple copies, as is the case for arcs.

In Case (1), starting from a collection $a_{\bf m}$, without loss of generality we handleslide a half-arc $a_1 \mid a_{\bf m}$ over another half-arc $a_2$. 

If $a_2 \mid a_{\bf m}$, then Theorem \ref{thm: invariance under handleslides} is not applicable because it would involve the cone of some $\Theta\in \op{Hom}(\widehat{L}_{\bf m}, \widehat{L}_{\bf m'})$ where $a_2^2\mid  a_{\bf m'}$, and such an object $\widehat{L}_{\bf m'}$ is not an object of the category.  Instead, we will show that $\widehat{L}_{\bf m}$ and $\widehat{L}_{\bf m'}$ are quasi-isomorphic in the sense of Claim~\ref{claim: quasi-isomorphic}, using the proof strategy of \cite[Theorem 9.3.1]{CHT}.

%{\cg gives a count of one holomorphic curve passing through a pair of points in $a_1$ and $a_2$.}
%{\cg The chain maps giving the quasi-isomorphisms are $A_\infty$-morphisms of $A_\infty$-bimodules.} 

If $a_2\nmid a_{\bf m}$, then Case (1) consists of two consecutive handleslides of a disk over a disk (Proposition~\ref{prop: half-arc over a half-arc}) and the handleslide invariance will come from two consecutive applications of Theorem \ref{thm: invariance under handleslides}. 

There is a similar splitting into two subcases in Case (2), depending on whether the half-arc over which we handleslide the arc is in $a_{\bf m}$ or not.

For Cases (3) and (4), the proof strategy is the same as that of Theorem \ref{thm: invariance under handleslides}, except Lagrangian disks are replaced by Lagrangian annuli if arcs (instead of half-arcs) are involved. The maps $\Phi$ and $\Psi$ are defined as in the proof of Theorem~\ref{thm: invariance under handleslides}. 
For the half-arcs $a_i$ that are untouched by the handleslide, one takes a pushoff $a_i'$ in the positive Reeb direction which intersects $a_i$ once.  Write $c_{ii}$ for the short Reeb chord in $\bdry S$ from $a_i$ to $a_i'$. Corresponding to each Reeb chord $c$ in $\bdry S$ there exists an $S^1$-Morse-Bott family of chords in $\bdry W$; in the perturbed version we consider the longer and shorter Reeb chords $\check{c}$ and $\hat{c}$ which are the bottom and top generators with respect to the cohomological grading. Let $\theta^{\op{bot}}_i:=\check{c}_{ii}$.  The case of arcs $a_i$ is similar, except there are two chords $c_{ii}$ and $c_{ii}'$ and we let $\theta^{\op{bot}}_i:=\check{c}_{ii}+\check{c}_{ii}'$. 
In view of \cite[Lemma 1.16]{Se3}, the key calculations are given by Lemmas~\ref{lemma: triangle1} and \ref{lemma: triangle2}, i.e., if we handleslide $a_1$ over an arc $a_2$ to obtain $a_1'$, there is a count of one holomorphic triangle passing through a point in $L_2$ and asymptotic to the relevant top/bottom generators of the chords between any two of $\partial L_1$, $\partial L_2$ and $\partial L_1'$.

\subsubsection{Case (1): half-arcs over half-arcs}\label{subsubsection: case 1}$\mbox{}$

\s\n
(i) We start with the subcase where we handleslide a half-arc $a_1\mid a_{\bf m}$ over a half-arc $a_2\mid a_{\bf m}$, yielding $a_1'$.
The proof of Theorem 9.3.1 
%--- strictly speaking we only need the proof of Theorem 9.3.6 --- 
for the hat version in \cite{CHT} applies in a straightforward manner to our wrapped case as follows:
%the wrapped versions and implies that handlesliding a half-arc over a half-arc induces an isomorphism on the cohomological level in the wrapped versions. {\cg The chain maps are moreover $A_\infty$-morphisms which gives the required equivalence.}

%Let $a_1'$ be the half-arc obtained by handlesliding $a_1$ over $a_2$. 
The half-arc $a_2$ ends at the critical value $x_2$ of $\pi$ and is contained in the strip bounded by $a_1$ and $a_1'$ as given in Figure~\ref{fig: An-arcslide}. For $i=2,\dots, k$, let $a_i'$ be the pushoff of $a_i$ which intersects $a_i$ at the same endpoint $x_i$. 
\begin{figure}[ht]
	\begin{overpic}[scale=1]{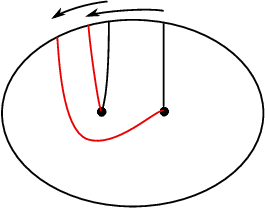}
		\put(41,35){\tiny $x_2$}
		\put(65,35){\tiny $x_1$}
		\put(28,51){\tiny $a_2'$}
		\put(33,19){\tiny $a_1'$}
		\put(42,51){\tiny $a_2$}
		\put(63,51){\tiny $a_1$}
		\put(50,77){\tiny $c_{12}$}
		\put(25,79){\tiny $c_{21}$}
	\end{overpic}
	\caption{}
	\label{fig: An-arcslide}
\end{figure}
We write $a_{\bf m}'$ for $a_{\bf m}$ with $a_1,a_2$ replaced by $a_1',a_2'$ and $\widehat{L}_{\bf m}'$ corresponding to $a_{\bf m}'$. For each pair $(i,j)$ with $i,j=1$ or $2$, let $c_{ij}$ be the shortest counterclockwise chord in $\bdry S$ from $a_i$ to $a_j'$. 
%Corresponding to each chord $c_{ij}$ there exists an $S^1$-Morse-Bott family of chords in $\bdry W$; in the perturbed version we consider the longer and shorter Reeb chords $\check c_{ij}$ and $\hat c_{ij}$. 
We write
$${\Theta_c}= \{\check c_{12},\check c_{21}, \theta_{3}^{\op{bot}} ,\dots,\theta_{k}^{\op{bot}}\}\quad \mbox{and} \quad \Theta = \{\theta_1^{\op{top}}, \theta_2^{\op{top}} ,\theta_{3}^{\op{top}} ,\dots,\theta_{k}^{\op{top}}\},$$ 
where $\theta_i^{\op{top}}$ is the intersection of $L_i$ and $L_i'$ if they intersect over $x_i$ and the top (= higher cohomological degree) generator of the $S^1$-Morse-Bott family of intersections if $a_i$ is an arc. It is easy to verify that $\Theta_c$ and $\Theta$ are cycles.

\begin{claim} \label{claim: quasi-isomorphic}
    $\widehat{L}_{\bf m}'$ and $\widehat{L}_{\bf m}$ are quasi-isomorphic in the category of twisted complexes of $\mathcal{R}^\tau(S,n,{\bf a})$-modules.
\end{claim}

\begin{proof}[Sketch of proof of Claim~\ref{claim: quasi-isomorphic}]
This follows the proof of \cite[Theorem 9.3.1]{CHT}, which extends from the hat to the wrapped version. The leading term of the quasi-isomorphism is given by:
\begin{gather} \label{eqn: chain map Phi}
\Phi: \Hom(\widehat{L}_{\bf m'},\widehat{L}_{\bf m})\to \Hom(\widehat{L}_{\bf m'},\widehat{L}_{\bf m}'),\\
\nonumber {\bf y}\mapsto \mu_2(\Theta_c \otimes {\bf y}),
\end{gather}
and the leading term of its inverse is:
\begin{gather} \label{eqn: chain map Psi}
\Psi: \Hom(\widehat{L}_{\bf m'},\widehat{L}_{\bf m}')\to \Hom(\widehat{L}_{\bf m'},\widehat{L}_{\bf m}),\\
\nonumber {\bf y}\mapsto \mu_2(\Theta \otimes {\bf y}).
\end{gather}

By \cite[Theorem 9.3.1]{CHT} and the count with point constraints from \cite[Theorem 9.3.6]{CHT}, the composition $\Psi\circ \Phi$ is chain homotopic to $\hbar^2\cdot p(\hbar)$ times the identity, where $p(\hbar)$ is a power series in $\hbar$ whose constant term is $\pm 1$.  The situation for $\Phi\circ \Psi$ is analogous.
\end{proof}

\s\n
(ii) Next we treat the subcase where we handleslide a half-arc $a_1\mid a_{\bf m}$ over a half-arc $a_2\nmid a_{\bf m}$. The main technical statement is the following:

\begin{prop} \label{prop: half-arc over a half-arc}
    Handlesliding $a_1$ over $a_2$ corresponds to two consecutive Lagrangian handleslides in $W_{S,n}$.
\end{prop} 

\begin{proof}[Proof of Proposition~\ref{prop: half-arc over a half-arc}]
For $i\leq n$, we take a small tubular neighborhood $N(a_i)$ of $a_i$ in $S$ and consider the closure $\widetilde{a}_i=cl(\bdry N(a_i)\cap \op{int}(S))$, which is a properly embedded arc surrounding the point $x_i$.
Let $\widetilde{L}_i$ be the Lagrangian disk obtained by parallel transporting a fiber $\{pt\}\times [-1,1]\subset A$ along $\widetilde{a}_i$ using the symplectic connection 
%(by abuse of notation we often denote it by $\{pt\}\times [-1,1] \times \widetilde{a}_i$) 
and rounding $\partial W_{S,n}$ to make $\bdry \widetilde{L}_i$ Legendrian.  The rounding of $\bdry \widetilde{L}_i$ can be done as follows: Let $N_0(a_i) \subset N(a_i)$ be a slightly smaller tubular neighborhood of $a_i$ in $S$.  Over $S\setminus \cup_{i=1}^n N_0 (a_i)$ the Liouville form $\lambda$ can be written as $pdq +\beta$, where $\beta$ is a Liouville $1$-form on $S$ which can be chosen so that $\widetilde{a}_i$ is tangent to its kernel. Then $\partial \widetilde{L}_i$ is Legendrian before rounding and can be kept Legendrian when rounding.
	
Now let $N'(a_i) \supset N(a_i)$ be a slightly enlarged tubular neighborhood of $a_i$ in $S$, so that $W_{S,n}$ is the Liouville connected sum of $\pi^{-1} (N'(a_i))$ and $\pi^{-1}(S\setminus \op{int}(N'(a_i))$\footnote{Here we must perturb the $\lambda$ and $\beta$ rel boundary in the previous paragraph.} along $\pi^{-1}(b)$, where $b=cl(\partial N'(a_i)\cap \op{int}(S))$.
The manifold $\pi^{-1} (N'(a_i))$ is the $4$-ball and the fibration $\pi : \pi^{-1} (N'(a_i)) \to N'(a_i)$ induces an open book decomposition supporting the standard contact structure $\ker \lambda$ on its boundary $3$-sphere. The pages of the open book decomposition are annuli (= the fibers of $\pi$ along $\partial N'(a_i)$), the monodromy is a positive Dehn twist along the core of the annulus, and the binding has two components, a neighborhood of which is given by a horizontal boundary component of $\pi^{-1} (N'(a_i))$.
	
\begin{claim}\label{claim: alternative}
	The disk $\widetilde{L}_i$ is Lagrangian isotopic to the thimble $L_i$ over $a_i$ through Lagrangian disks with Legendrian boundary in $W_{S,n}$.
\end{claim}

\begin{proof}[Proof of Claim~\ref{claim: alternative}]
We first work over $\pi^{-1} (N'(a_i))$. The boundary $\Lambda_i=\bdry L_i$ is the core Legendrian $0_{T^* S^1}$ contained in the vertical fiber over the point $a_i \cap \partial S$.
The boundary $\widetilde{\Lambda}_i =\partial \widetilde{L}_i$ decomposes into two vertical Legendrian arcs contained in two different pages and two horizontal arcs contained in slices $\{pt\} \times \{\pm 1\} \times D^2$ of neighborhoods of the two components of the binding. Now, in the rounded boundary $(\partial \pi^{-1} (N'(a_i)),\xi)\simeq (S^3,\xi_{std})$, the Legendrians $\Lambda_i$ and $\widetilde{\Lambda}_i$ are Legendrian unknots that are Legendrian isotopic in the complement of $\pi^{-1}(b)$ and hence are also Legendrian isotopic in the rounded $\partial W_{S,n}$. This is because the complement of a small open neighborhood of $\pi^{-1}(b)$ in the rounded $\partial \pi^{-1} (N'(a_i))$ is a standard solid torus, in which both Legendrians are isotopic to a standard longitude.
		
We then argue that (i) since the trace of a Legendrian isotopy in a contact manifold can be lifted into a Lagrangian cylinder in its symplectization by Ekholm-Honda-K\'alm\'an \cite{EHK}, the disk $\widetilde{L}_i$ is Lagrangian isotopic in $\partial \pi^{-1} (N'(a_i)) \setminus \pi^{-1} (b)$ (through Lagrangian disks with Legendrian boundary) to a Lagrangian disk $L_i'$ with Legendrian boundary $\Lambda_i =\partial L_i$, and that (ii) $L_i'$ and $L_i$ are then Lagrangian isotopic in the symplectic $4$-ball $\pi^{-1} (N'(a_i))$ by the uniqueness of the Lagrangian disk bounded by the standard Legendrian unknot due to Eliashberg-Polterovich \cite{EP}. Hence $\widetilde{L}_i$ and $L_i$ are Lagrangian isotopic in $\pi^{-1} (N'(a_i)) \setminus \pi^{-1} (b)$ and {\it a fortiori} in $W_{S,n}$.		
%{\it Remark: We could probably argue that the disks $L_i$ and $\widetilde{L}_i$ are conical over their boundaries and get a Lagrangian isotopy of conical Lagrangian disks sitting over the Legendrian isotopies of their boundaries.}
\end{proof}

Continuing with the proof of Proposition~\ref{prop: half-arc over a half-arc}, by Claim~\ref{claim: alternative} we can replace $L_i$, $i\leq n$, by $\widetilde{L}_i$ in the Lagrangian basis $\{L_i\}_{1\leq i\leq n+2g}$. Then $\widetilde{L}'_1$ corresponding to $a'_1$ is obtained from $\widetilde{L}_1$ by two consecutive handleslides over $\widetilde{L}_2$. Indeed, in the base $S$, we pass from the arc $\widetilde{a}_1$ to the arc $\widetilde{a}'_1$ by consecutively handlesliding {\em both} ends of $\widetilde{a}_1$ over $\widetilde{a}_2$. By the first handleslide model from Section \ref{subsub: model1} when $\op{dim} W=2\nu=4$, $\widetilde{L}'_1$ is similarly obtained by two consecutive handleslides of  $\widetilde{L}_1$ over $\widetilde{L}_2$. Viewed directly on $L_1$ and $L_2$, the handleslides required to pass from $L_1$ to $L'_1$ are along the two short chords from $\partial L_1$ to $\partial L_2$.
This completes the proof of Proposition~\ref{prop: half-arc over a half-arc}.
\end{proof}

% {\em A disk slide induces an equivalence between $\ai$-categories of twisted complexes by Theorem \ref{thm: invariance under handleslides}.} 
This completes the proof of Case (1).

\subsubsection{Case (4): arcs over arcs}\label{subsection: arcs over arcs}

Let $a_1$, $a_2$ be arcs in $S$ and $a_1'$ be the arc obtained by handlesliding $a_1$ over $a_2$. The arcs $a_1$, $a_2$ and $a_1'$, together with $\partial S$, bound a hexagon $H$ in $S$ such that $\pi^{-1}(H)\simeq H\times A\subset W_{S,n}$ (recall $A$ is an annulus in $T^*S^1$) and $H\cap {\bf x}=\varnothing$. %Note that $H$ does not contain any critical value $x_i$ of $\pi : W_{S,n} \to S$. 
The Lagrangians under consideration are $L_1$, $L_2$, and $L_1'$, given respectively by $a_1 \times 0_{T^*S^1}$, $a_2\times 0_{T^*S^1}$, and $a_1'\times 0_{T^*S^1}$. 

We do a holomorphic curve computation on $H\times A$, where we can use the split symplectic structure and split almost complex structure $J$, so that the projections of $H\times A$ to $H$ and $A$ are $J$-holomorphic.  We also partially wrap the arcs to transform the short chords into intersection points; see the left-hand side of Figure~\ref{triangle2}. Over each intersection point in $S$ between the arcs $a_1,a_2,a_1'$, there is a clean circle of intersections corresponding to the $0$-section.  We then slightly perturb $L_1,L_2,L_1'$ (while retaining their names) so they are still products of $a_1,a_2,a_1'$ with Hamiltonian deformations of the zero section in $A$ that intersect each other transversely and at exactly two points; see the right-hand side of Figure~\ref{triangle2}.

We consider the bottom generator $\theta$ of $CF(L_1,L_2)$ and the bottom generators $c$ and $d$ of $CF(L_1',L_1)$ and $CF(L_2,L_1')$, respectively.

As discussed at the beginning of Section~\ref{subsection: proof of theorem}, the desired curve count is:

\begin{lemma}\label{lemma: triangle1}
The algebraic count of holomorphic triangles $u: D_2\to W_{S,n}$, $D_2\in \mathcal{A}_2$, which:
\be
\item map $\bdry_0 D_2,\bdry_1 D_2,\bdry_2 D_2$ to $L_1$, $L_2$, $L_1'$, respectively; 
\item limit to intersection points $\theta$, $c$, $d$; and
\item pass through a generic point $z$ in $L_1$,
\ee
is $\pm 1$.
\end{lemma} 

\begin{proof} 
The holomorphic map $u$ is completely determined by its two projections $u_1$ and $u_2$ to $H$ and $A$ and must pass through $z=(z_1,z_2)\in L_1$. The first projection fixes the cross ratio of $z_1$. 
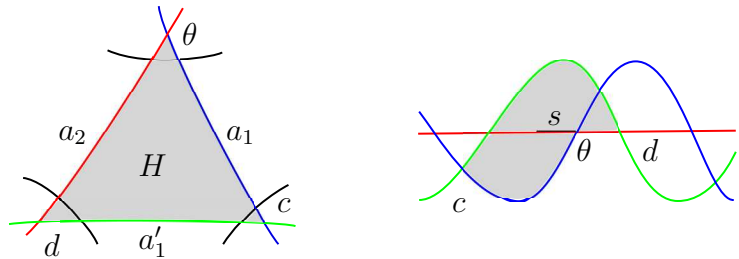
\begin{figure}[ht]
	\begin{overpic}[scale=1]{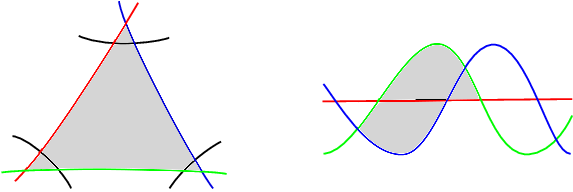}
	\put(18,10){$H$}
    \put(37,5){$c$}\put(24,28){$\theta$}\put(5,-1){$d$}
	\put(30,15){$a_1$}\put(7,15){$a_2$}\put(18,0){$a_1'$}
	\put(61,5){$c$}\put(78,12){$\theta$}\put(87,12){$d$}\put(74,17){$s$}
	\end{overpic}
	\caption{The projections $u_1$ and $u_2$ of the triangle in $H$ and $A$, respectively.}
	\label{triangle2}
\end{figure}
In $A$, there is a $1$-parameter family of holomorphic triangles, indexed by the size of the slit $s$ (see Figure \ref{triangle2}), which allows us to adjust the cross ratio for $z_2$. This gives a count of $1$ holomorphic triangle.
\end{proof}

\subsubsection{Case (3): half-arcs over arcs}

Let $a_1$ be a half-arc that we slide over an arc $a_2$ to obtain a half-arc $a_1'$. Let $L_1$ and $L_1'$ be Lagrangian thimbles over $a_1$ and $a_1'$ that intersect at a point $d$ over the critical value $x_1$ of $\pi$, and let $L_2$ be the Lagrangian annulus over $a_2$. We also partially wrap the arcs/half-arcs to transform the short chords into intersection points. Assume that $a_1, a_2,a_1'$ cobound a triangle in counterclockwise order. Let $\theta$ (resp.\ $c$) be the bottom generator of $CF(L_1,L_2)$ (resp.\ $CF(L_2, L_1')$).

The desired curve count in our case is the following:

\begin{lemma}\label{lemma: triangle2} 
The algebraic count of holomorphic triangles $u: D_2\to W_{S,n}$, $D_2\in \mathcal{A}_2$, which:
\be
\item map $\bdry_0 D_2,\bdry_1 D_2,\bdry_2 D_2$ to $L_1$, $L_2$, $L_1'$, respectively; 
\item limit to intersection points $\theta$, $c$, $d$; and
\item pass through a generic point $z$ in $L_1$,
\ee
is $\pm 1$.
\end{lemma}

\begin{proof} 
We can resolve the intersection $d$ to obtain a Lagrangian cylinder $L_1''$ that is isotopic to $L_2$. By Lemma \ref{lemma: FO3}, the count of triangles through $z$ is the same as the count of strips between $L_2$ and $L_1''$ that pass through $z$ and are asymptotic to $\theta$ and $c$.  The latter count is algebraically $1$ by a standard computation.
\end{proof}

\subsubsection{Case (2): arcs over half-arcs} \label{subsubsection: case 2}

In this case, we slide an arc $a_1$ over a half-arc $a_2$ to obtain the arc $a_1'$. The half-arc $a_2$ ends at a critical value $x_2$ of $\pi$ which is contained in the strip bounded by $a_1$ and $a_1'$.

\s\n
(i) Consider the subcase where the half-arc $a_2\mid a_{\bf m}$. Let $a_{\bf m}'$ be $a_{\bf m}$ with $a_1$ replaced by $a_1'$ and let $\widehat{L}_{\bf m}$ and $\widehat{L}_{\bf m}'$ be the Lagrangians over $a_{\bf m}$ and $a_{\bf m}'$.  

Let $c_{12}$ be the shortest counterclockwise chord in $\bdry S$ from $a_1$ to $a_2$, $c_{21}$ the one from $a_2$ to $a_1'$, and $c_{11}$ the one from $a_1$ to $a_1'$ that passes over $a_2$.  
%To each chord $c_{ij}$ there corresponds an $S^1$-Morse-Bott family of chords in $\bdry W$ and in the perturbed version we consider chords $\check c_{ij}$ and $\hat c_{ij}$ which are the longer and shorter Reeb chords. 
The situation is similar to that of Figure~\ref{fig: An-arcslide} except that $a_1$ and $a_1'$ are arcs which intersect at a point instead of at $x_1$.

We write 
$${\Theta_c}= \{\check c_{12}, \check c_{21},\theta_{3}^{\op{bot}} ,\dots,\theta_{k}^{\op{bot}}\}\quad \mbox{and}\quad \Theta = \{\theta_1^{\op{top}}, \theta_2^{\op{top}} ,\theta_{3}^{\op{top}} ,\dots,\theta_{k}^{\op{top}}\},$$
where $\theta_i^{\op{top}}$ is the intersection of $L_i$ and $L_i'$ if they intersect over $x_i$ and the top generator of the $S^1$-Morse-Bott family of intersections if $a_i$ is an arc. 

\begin{claim}\label{claim: quasi-isomorphic 2}
    $\widehat{L}_{\bf m}'$ and $\widehat{L}_{\bf m}$ are quasi-isomorphic in the category of twisted complexes of $\mathcal{R}^\tau(S,n,{\bf a})$-modules.
\end{claim}

\begin{proof}[Proof of Claim~\ref{claim: quasi-isomorphic 2}]
    This is similar to the proof of Claim~\ref{claim: quasi-isomorphic} and uses the curve count with point constraints from \cite[Theorem 9.3.1]{CHT}.  The only difference is that $L_1$ and $L_1'$ intersect along an $S^1$-Morse-Bott family over $a_1\cap a_1'$ as opposed to Case (1) where $L_1^\star$ and $(L_1^\star)'$ over half-arcs $a_1^\star$ and $(a_1^\star)'$ intersect at one point $\theta_1^\star$.  To go between the two cases, we resolve the intersection $\theta_1^\star$ and then transform it into a clean intersection as described in Figure \ref{fig: lagrangian-surgery2}. 
\begin{figure}[ht]
	\begin{overpic}[scale=1]{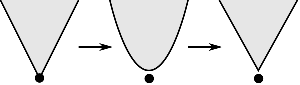}
	\end{overpic}
	\caption{}
	\label{fig: lagrangian-surgery2}
\end{figure}
By \cite[Lemma 2.3.5]{FO3}, there is a bijection between holomorphic curves passing through $\theta_1^\star$ and those passing through the clean intersection.  The rest of the proof of Claim~\ref{claim: quasi-isomorphic} goes through.
% curves that are unconstrained along this clean intersection correspond to curves counted in \cite[Theorem 9.3.6]{CHT} going to $\theta_1$ before resolving, 
\end{proof}

\s\n
(ii) Consider the subcase where $a_2\nmid a_{\bf m}$. If $a_2'$ is the usual pushoff of $a_2$, then the resolution of $L_2$ and $L_2'$ along their intersection point $d$ over $x_2$ is a Lagrangian annulus $\widetilde{L}_2$ that is the lift of an arc $\widetilde{a}_2$ surrounding the point $x_2$. Then $\widetilde L_2$ is quasi-isomorphic in the category of twisted complexes to the cone of $d \in \op{Hom}(L_2, L_2')$ and $L_{\bf m}\cdot \widetilde L_2/L_1$ is quasi-isomorphic to the cone of $\{ d, \theta_2^{\op{bot}},\theta_3^{\op{bot}},\dots,\theta_k^{\op{bot}}\} \in \Hom (L_{{\bf m}}\cdot L_2/L_1 ,L_{\bf m}\cdot L_2'/L_1)$, where the $\theta_i^{\op{bot}}$ are the bottom generators of the terms of $a_{\bf m}/a_1$ and their pushoffs.
%$a_1'$ is obtained by handlesliding the arc $a_1$ over $\widetilde{a}_2$ and, 
On the other hand, by Section~\ref{subsection: arcs over arcs},
%trading $L_1$ with $L_1'$ gives an equivalence between the corresponding triangulated envelopes. Then 
$L_{\bf m}':= L_{\bf m} \cdot L_1'/L_1$ is quasi-isomorphic to the cone of $\{ \check c_{12}, \theta_2^{\op{bot}},\dots,\theta_k^{\op{bot}}\}\in \op{Hom}(L_{\bf m}, L_{\bf m}\cdot \widetilde L_2/L_1)$, where $c_{12}$ is the shortest Reeb chord from $a_1$ to $\widetilde a_2$.   Combining the last two statements yields Case (2).

\section{Symplectic Khovanov homology of a braid closure with respect to a fibration over $S^1$} \label{section: Khovanov homology of braid in fibration}

Let $\pi : M\to S^1$ be a fibration over $S^1$ and let $\ell\subset M$ be a link  transverse to the fibers.  The goal of this section is to define an invariant of $\ell$ {\em with respect to the fibration $\pi$}.  We pick an auxiliary knot $K \in M$ which transversely intersects each fiber at one point and assume that $\ell \cap K=\varnothing$. By drilling out a tubular neighborhood $N(K)$ of $K$ from $M$, we can identify
$$M\setminus \op{int} (N(K))\simeq N_{(S,h)}=(S\times[0,1])/ (x,1)\sim (h(x),0),$$
where $S$ is a fiber $\pi^{-1}(pt)$ minus an open disk and $h\in \op{Diff}^+(S,\bdry S)$.

We can cut $\ell$ along $S$ to obtain a braid $\mathbb{b}$ in $S\times  [0,1]$.  We assume that $h$ has been isotoped in $\op{Diff}^+(S,\bdry S)$ so that it takes ${\bf x}= \{x_1,\dots,x_n\}\subset \op{int}(S)$ to itself
 pointwise; we write it as $h\in \op{Diff}^+(S,\bdry S,{\bf x})$. We also assume that the endpoints of $\mathbb{b}$ are ${\bf x}\times \{0,1\}$.

We denote the Hochschild homology of the bimodule $B^f(S,n, {\bf a};\mathbb{b}, h)$ by $HH^f(S,n, {\bf a};\mathbb{b}, h)$. 

\begin{q}
What is the relationship between $HH^f(S,n, {\bf a};\mathbb{b}, h)$ and Rozansky's link invariant for $S^1\times S^2$ \cite{Ro} when we take $S$ to be annulus and $h=\op{id}$? Also compare with Willis' generalization to $\#^r S^1\times S^2$ \cite{Wi}.
\end{q}

\begin{q}
What is the relationship between $HH^f(S,n, {\bf a};\mathbb{b}, h)$ and the invariants of Asaeda-Przytycki-Sikora~\cite{APS} on $\overline{S}\times [0,1]$ when $\ell$ can be isotoped into $\overline{S}\times[0,1]$?
\end{q}

Let $h_{\mathbb{b}}\in \op{Diff}^+(S,\bdry S,{\bf x})$ be a diffeomorphism isotopic to the identity whose trace is $\mathbb{b}$.  By this we mean the following:  There exists $h_t\in \op{Diff}^+(S,\bdry S)$, $t\in[0,1]$, such that $h_0=\op{id}$ and $h_1=h_{\mathbb{b}}$ and such that $\{(h_t({\bf x}), t)~|~t\in[0,1]\}$ is the braid $\mathbb{b}$.

We then define
$$B^\tau(S,n,{\bf a};\mathbb{b},h):=\oplus_{k\geq 0,|{\bf m}|=|{\bf m'}|=k} \widehat{CF}^\tau(L_{\bf m}, h_{\mathbb{b}}\circ h(L_{\bf m'})).$$
It is an $\ai$-bimodule over $R^\tau(S,n,{\bf a})=\oplus_{k\geq 0} R^\tau(S,n,{\bf a};k)$. The coefficients are taken in 
$\F[\mathcal{A}]\llbracket \hbar, \hbar^{-1}]$ as before.

We denote the Hochschild homology of the bimodule $B^\tau(S,n,{\bf a};\mathbb{b},h)$ by $HH^\tau(S,n, {\bf a};\mathbb{b}, h)$.

\begin{proof}[Proof of Theorem~\ref{thm: braid in fibration}]
$HH^f(S,n, {\bf a}; \mathbb{b},h)$ is an invariant of the braid $\ell$ in the mapping torus of $(S,h)$ by Theorem~\ref{thm: invariance}.  Note it is invariant under conjugations of $h$ and thus depends only on the fibration $\pi: N_{(S,h)}\to S^1$.

It remains to see that it is invariant under an ``Alexander move'', i.e., viewing $K$ as fixed and isotoping a strand of $\ell$ across $K$ through transverse links. Such a move corresponds to trading $h$ for $\tau \circ h$, where $\tau$ corresponds to a braid of the following type:
\be
\item trivial strands from $x_i$ to $x_i$ for all $x_i\not= x_j$; and
\item the strand starting at $x_j$ follows an arc $c$ to $\bdry S$, goes once around $\bdry S$, and then back to $x_j$. 
\ee
We have the following isomorphism
$$HH^f(S,n, {\bf a}; \mathbb{b},h)\simeq HH^f(S,n, {\bf a}; \mathbb{b},\tau\circ h).$$
which is a consequence of the full wrapping.
\end{proof}

\begin{conj}
    More generally, one could expect to define a Khovanov-type homology associated to a link in a fibered $3$-manifold by decomposing it into elementary pieces in consecutive products $S\times [t_k,t_{k+1}]$ and tensoring the corresponding elementary braids-, cap- and cup- bimodules as in \cite{AS}. The Hochschild homology of the resulting bimodule should be our invariant.
\end{conj}

\part{Combinatorial description of surface categories}

The goal of Part 2 is to give (conjectural) combinatorial descriptions of the partially wrapped $\ai$-algebras $R^p(S,n,{\bf a})$ in terms of dgas $R(S,n,{\bf a})$. Our dgas can be viewed as higher-dimen\-sional analogs of the strands algebras in bordered Heegaard Floer homology, due to Lipshitz-Ozsv\'ath-Thurston.  

There are three types of arcs on a surface $(S, n, \tau)$: 
\be 
\item an arc connecting two points on different components of the boundary $\bdry S\backslash \tau$;
\item a half-arc from a point on the boundary to a singular point; and
\item an arc connecting two points on the same boundary component.
\ee
It gives a Lagrangian in $W_{S,n}$ which is a disk in Case (2), and an annulus in Cases (1) and (3). 
We present a combinatorial description of endomorphism algebras of the products of these Lagrangians.    
The description will be carried out in four steps:
\be
\item[(a)] the local model $S=D^2, n=0, |\tau|=2$ in Section \ref{sec local}: a single arc of type (1);
\item[(b)] the $A_n$-singularity $S=D^2, |\tau|=1$ in Section \ref{sec An}: a collection of arcs of type (2);
\item[(c)] an annulus with $|\tau|=1$ in Section \ref{ssec annulus}: a single arc of type (3); and
\item[(d)] a parametrized surface in Section \ref{ssec surface}.
\ee

\section{The local model} \label{sec local}

The surface is a disk without singular points and has a stop consisting of two points on the boundary, i.e., $S = D^2, n = 0, |\tau| = 2$, and $W_{S,n}=D^2\times A=D^2\times S^1\times[-1,1]$. Let $a\subset D^2$ be the arc such that each component of $D^2\setminus a$ nontrivially intersects $\tau$. Let $L\subset W_{S,n}$ be the Lagrangian $a\times S^1\times\{0\}$. Then $\rk$ is the conjectural algebraic description of $\End(L_k)$, where $L_k$ is the family of Lagrangian cylinders sitting above $k$ parallel copies of $a$.

\subsection{The basic version $\rk$}  \label{ssec rk}

In this subsection we define the dga $\rk$ and analyze its properties.   As a graded algebra, $\rk$ is isomorphic to a twisted tensor product $\slk$, where $H_k$ is the Hecke algebra of the symmetric group $S_k$ and $\lk$ is the exterior algebra on $k$ generators of degree $1$.
We prove that the cohomology algebra $H(\rk)$ is isomorphic to $\lk$ (see Proposition \ref{prop H(Rk)}) and conjecture that $\rk$ is a formal dga, i.e., quasi-isomorphic to $H(\rk)$.

\subsubsection{Definition of $\rk$}

Let $H_k$ be the Hecke algebra of $S_k$ and $\lk$ be the exterior algebra on $k$ generators.
The twisted tensor product $\slk$ is generated by $T_i, \xi_j$, $1 \leq i \leq k-1, 1 \leq j \leq k$, and satisfies the following defining relations:
\begin{gather} \label{rel Rk}
T_i^2=1+\hbar T_i, \quad T_iT_{i+1}T_i=T_{i+1}T_iT_{i+1}, \quad  T_iT_{i'}=T_{i'}T_i ~\mbox{for}~|i-i'|>1; \\
\nonumber \xi_j^2=0, \quad  \xi_j\xi_{j'}=-\xi_{j'}\xi_j ~\mbox{for}~j \neq j';   \\
\nonumber \xi_iT_i=T_i\xi_{i+1},  \quad \xi_{i+1}T_i=T_i\xi_{i}-\hbar(\xi_{i}-\xi_{i+1}), \quad  \xi_jT_i=T_i\xi_j ~\mbox{for}~j \neq i,i+1.
\end{gather}
%The generators $\xi_j$ are of degree $1$, and $T_i$ are of degree $0$.
The parameter $\hbar$ is a nonzero element of the ground field $\mathbb{F}$. 
The subalgebra generated by the $T_i$ is isomorphic to the Hecke algebra $H_k$ with a change of parameter $\hbar=q-q^{-1}$. 
The generator $T_i$ is invertible: $T_i^{-1}=T_i-\hbar$.
The subalgebra generated by the $\xi_j$ is isomorphic to $\lk$. 
Let $\mb_k$ denote the unit of $\slk$.

\begin{lemma}
The graded algebra $\slk$ is well-defined.
\end{lemma}

\begin{proof}
We verify that the twisting operations given by the third line of \eqref{rel Rk} are compatible with the defining relations for $H_k$ and $\lk$ given by the first two lines of \eqref{rel Rk}. 
For instance, 
\begin{align*}
\xi_iT_i^2=&T_i\xi_{i+1}T_i=T_i(T_i\xi_{i}-\hbar(\xi_{i}-\xi_{i+1}))=(T_i^2-\hbar T_i)\xi_i+\hbar T_i\xi_{i+1}=\xi_i(1+\hbar T_i),
\end{align*}
and 
\begin{align*}
\xi_{i+2}T_iT_{i+1}T_i=&T_i\xi_{i+2}T_{i+1}T_i=T_i((T_{i+1}-\hbar)\xi_{i+1}+\hbar\xi_{i+2})T_i \\
=&T_i(T_{i+1}-\hbar)(T_{i}-\hbar)\xi_i+\hbar T_i(T_{i+1}-\hbar)\xi_{i+1}+\hbar T_i^2\xi_{i+2}, \\
\xi_{i+2}T_{i+1}T_iT_{i+1}=&((T_{i+1}-\hbar)\xi_{i+1}+\hbar\xi_{i+2})T_iT_{i+1}\\
=&(T_{i+1}-\hbar)((T_{i}-\hbar)\xi_i+\hbar\xi_{i+1})T_{i+1}+\hbar T_i\xi_{i+2}T_{i+1} \\
= & (T_{i+1}-\hbar)(T_{i}-\hbar)T_{i+1}\xi_i+\hbar(T_{i+1}-\hbar)T_{i+1}\xi_{i+2}+\hbar T_i(T_{i+1}-\hbar)\xi_{i+1}+\hbar^2T_i\xi_{i+2} \\ 
= & (T_{i+1}-\hbar)(T_{i}-\hbar)T_{i+1}\xi_i+\hbar T_i(T_{i+1}-\hbar)\xi_{i+1}+\hbar T_i^2\xi_{i+2},
\end{align*}
so that $\xi_{i+2}T_iT_{i+1}T_i=\xi_{i+2}T_{i+1}T_iT_{i+1}.$ 
The other relations are similar and left to the reader. 
\end{proof}

The Hecke algebra $H_k$ has a standard basis $\{T_w~|~w\in S_k\}$, where $T_w=T_{i_1}\cdots T_{i_l}$ if $w=s_{i_1}\cdots s_{i_l}$ is a reduced expression. 
The algebra $\slk$ has a PBW type basis 
$$\{T_w\xi~|~w \in S_k, \xi ~\mbox{monomial in}~\lk\}.$$
Hence it is a finite-dimensional algebra of rank $2^k k!$.
If $\hbar=0$, then $\slk$ reduces to the semidirect product $\F[S_k] \rtimes \lk$, where $S_k$ acts on $\lk$ by permutations.

\begin{defn} \label{def Rk}
The dga $\rk=(\slk, d)$ has differential $d$, defined on the generators as
$$d(\mb_k)=0, \qquad d(T_i)=\xi_i-\xi_{i+1}, \qquad d(\xi_j)=0,$$ 
and extended by the Leibniz rule: $d(ab)=d(a)b+(-1)^{\deg(a)}a~d(b)$.
The generators have degree $\deg(\xi_j)=1$ and $\deg(T_i)=0$.
\end{defn}

\begin{rmk}
The HDHF differential is given by $d(T_i)=\hbar(\xi_i-\xi_{i+1})$; see \cite[Lemma 9.4.1]{CHT}. We normalize the generators $\xi_i$ so that $d(T_i)=\xi_i-\xi_{i+1}$ in our algebraic presentation. 
\end{rmk}

A diagrammatic presentation of $\rk$ is given by Figure~\ref{fig alg1}.
The generator $T_i$ of $H_k$ is drawn as a braid with a single positive crossing.
The generator $\xi_j$ of $\lk$ is drawn by putting a dot on the $j$th horizontal strand.
Some of the relations can be interpreted as:
\be
\item[(R1)] disjoint diagrams supercommute; 
\item[(R2)] a diagram containing a strand with more than one dot is zero; and 
\item[(R3)] a dot can freely slide through a crossing from above.
\ee

\begin{figure}[ht]
\begin{overpic}
[scale=0.25]{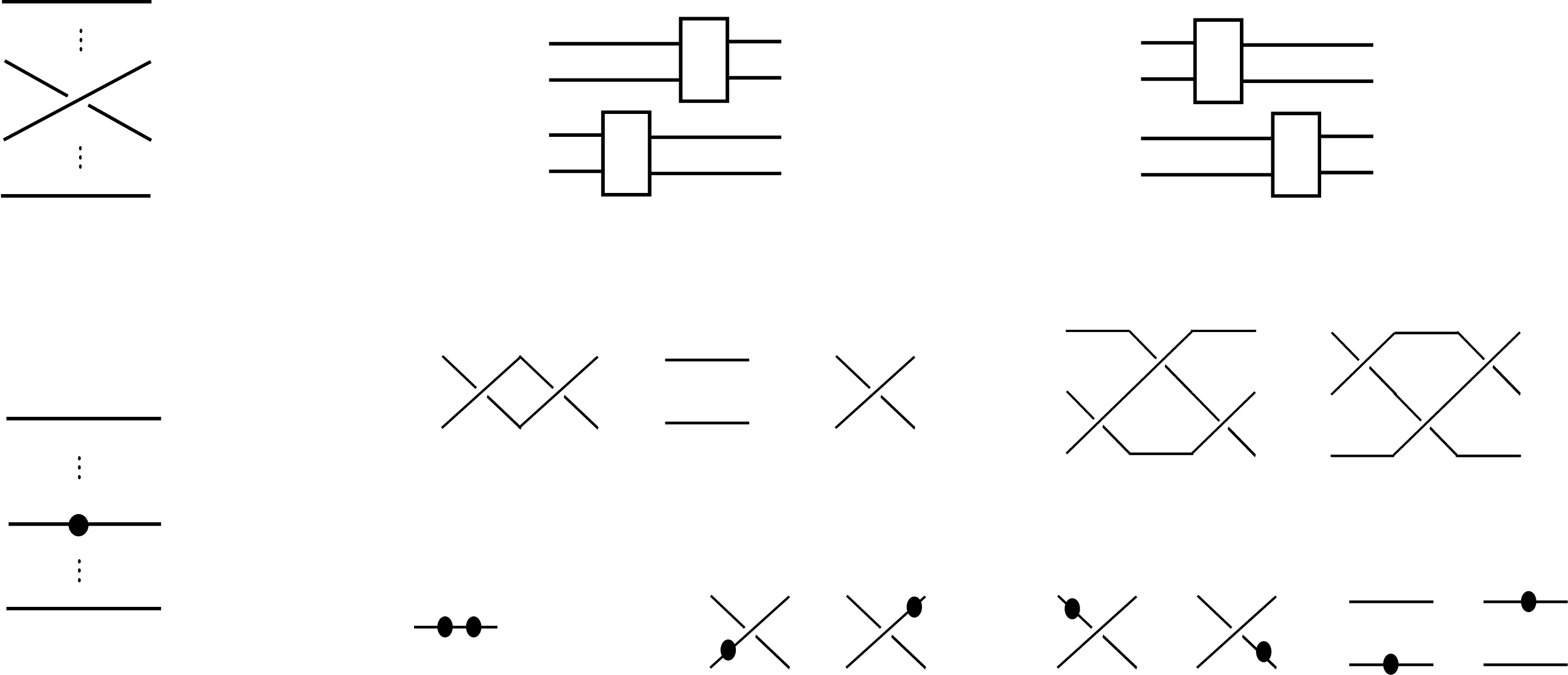}
\put(4,26){$T_i$}
\put(5,0){$\xi_j$}
\put(-3,4){${\scriptstyle 1}$}
\put(-3,10){${\scriptstyle j}$}
\put(-3,16){${\scriptstyle k}$}
\put(-3,31){${\scriptstyle 1}$}
\put(-3,35){${\scriptstyle i}$}
\put(-5,40){${\scriptstyle i+1}$}
\put(-3,44){${\scriptstyle k}$}
\put(39.5,33){${\scriptstyle a}$}
\put(82,33){${\scriptstyle a}$}
\put(44.5,39){${\scriptstyle b}$}
\put(77,39){${\scriptstyle b}$}
\put(51,36){${\scriptstyle =(-1)^{\deg(a)\deg(b)}}$}
\put(39,17){${\scriptstyle =}$}
\put(49,17){${\scriptstyle + \hbar}$}
\put(81,17){${\scriptstyle =}$}
\put(73,2){${\scriptstyle =}$}
\put(81.5,2){${\scriptstyle - \hbar(}$}
\put(92,2){${\scriptstyle -}$}
\put(101,2){${\scriptstyle )}$}
\put(51,2){${\scriptstyle =}$}
\put(33,2){${\scriptstyle =0}$}
\end{overpic}
\caption{The generators $T_i, \xi_j$ of $\rk$ are on the left and the relations are on the right.}
\label{fig alg1} 
\end{figure}
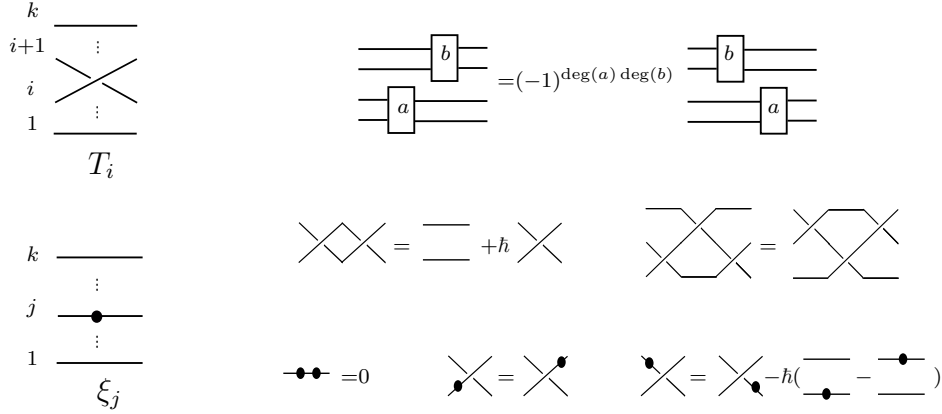

The differential of a single crossing has two terms, given by resolving the crossing and adding a dot on one of the two strands up to signs; see Figure~\ref{fig alg1n}.

\begin{figure}[ht]
\begin{overpic}
[scale=0.25]{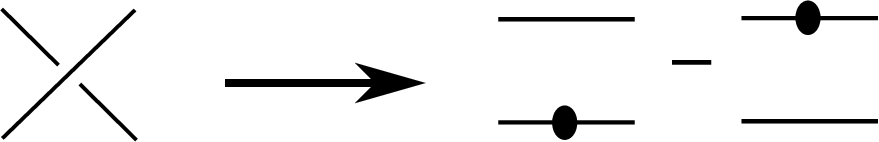}
\put(32,10){$d$}
\end{overpic}
\caption{The differential of a crossing.}
\label{fig alg1n}
\end{figure}

\begin{rmk} \label{rmk Rk LOT}
The dga $\rk$ is the higher-dimensional analog of the LOT dga. 
They diagrammatically differ in the following ways:
\be
\item the degree of a single crossing is zero in $\rk$, while it is minus one in the LOT algebra;
\item the double crossing satisfies the quadratic relation in $\rk$, while it is zero in the LOT algebra;
\item the dot generators of $\rk$ do not exist in the LOT algebra;
\item the differential of a single crossing has two terms with dots, while it has one term in the LOT algebra.
\ee
\end{rmk}

\begin{lemma} \label{lem Rk well}
The dga $\rk$ is well-defined.
\end{lemma}

\begin{proof}
It suffices to prove that the differential $d$ preserves the defining relations of $\slk$.
We verify some of them and leave the remaining verifications to the reader:
\begin{align*}
d(\xi_iT_i)=&-\xi_i(\xi_i-\xi_{i+1})=\xi_i\xi_{i+1}=(\xi_i-\xi_{i+1})\xi_{i+1}=d(T_i\xi_{i+1}), \\
d(T_i^2)=&(\xi_i-\xi_{i+1})T_i+T_i(\xi_i-\xi_{i+1})=\hbar (\xi_i-\xi_{i+1})=d(1+\hbar T_i),\\
d(T_iT_{i+1}T_i)=&(\xi_i-\xi_{i+1})T_{i+1}T_i+T_i(\xi_{i+1}-\xi_{i+2})T_i+T_iT_{i+1}(\xi_i-\xi_{i+1}) \\
=& T_{i+1}T_i(\xi_{i+1}-\xi_{i+2})+T_i(T_i-\hbar)\xi_{i+1}+\hbar T_i\xi_{i+1}-T_i^2\xi_{i+2}+T_iT_{i+1}(\xi_i-\xi_{i+1})\\
=& T_{i+1}T_i(\xi_{i+1}-\xi_{i+2})+T_i^2\xi_{i+1}-T_i^2\xi_{i+2}+T_iT_{i+1}(\xi_i-\xi_{i+1}),
\end{align*}
\begin{align*}
d(T_{i+1}T_iT_{i+1})=&(\xi_{i+1}-\xi_{i+2})T_iT_{i+1}+T_{i+1}(\xi_i-\xi_{i+1})T_{i+1}+T_{i+1}T_i(\xi_{i+1}-\xi_{i+2}) \\
=&((T_i-\hbar)\xi_i+\hbar\xi_{i+1})T_{i+1}-T_i(T_{i+1}-\hbar)\xi_{i+1}-\hbar T_i\xi_{i+2}+\\
&T_{i+1}(T_{i+1}-\hbar)\xi_{i+1}+\hbar T_{i+1}\xi_i-T_{i+1}^2\xi_{i+2}+T_{i+1}T_i(\xi_{i+1}-\xi_{i+2}) \\
=&(T_i-\hbar)T_{i+1}\xi_i+\hbar T_{i+1}\xi_{i+2}-T_i(T_{i+1}-\hbar)\xi_{i+1}-\hbar T_i\xi_{i+2}+\\
&T_{i+1}(T_{i+1}-\hbar)\xi_{i+1}+\hbar T_{i+1}\xi_i-T_{i+1}^2\xi_{i+2}+T_{i+1}T_i(\xi_{i+1}-\xi_{i+2}) \\
=&T_iT_{i+1}(\xi_i-\xi_{i+1})+T_i^2\xi_{i+1}-T_i^2\xi_{i+2}+T_{i+1}T_i(\xi_{i+1}-\xi_{i+2}),
\end{align*}
so that $d(T_iT_{i+1}T_i)=d(T_{i+1}T_iT_{i+1})$. 
\end{proof}

%\begin{figure}[ht]
%\begin{overpic}
%[scale=0.25]{alg2.eps}
%\put(30,9){$d$}
%\put(30,30){$d$}
%\put(30,52){$d$}
%\put(51,48){$-$}
%\put(68,48){$+$}
%\put(84,48){$-$}
%\put(101,48){$=0$}
%\put(40,26){$-$}
%\put(58,26){$+$}
%\put(78,26){$=$}
%\put(58,5){$-$}
%\put(78,5){$=$}
%\end{overpic}
%\caption{The differential $d$ preserves some relations of $\rk$.}
%\label{fig alg2}
%\end{figure}

\begin{example} \label{ex Rk} $\mbox{}$
\begin{enumerate}
    \item The differential on $R_1$ is trivial so that $R_1 \cong H(R_1) \cong \Lambda_1$.
    \item The cohomology $H(R_2)$ has a linear basis
    $$\{[\mb_2], [\xi_1]=[\xi_2], [T_1(\xi_1-\xi_2)], [T_1\xi_1\xi_2]\},$$
    and is isomorphic to $\lm_2$ as a graded algebra, where $[T_1\xi_1\xi_2]=[T_1(\xi_1-\xi_2)][\xi_1]$. Moreover, the map $H(R_2) \ra R_2$ given by $[a] \mapsto a$ for $a \in \{\mb_2, \xi_1, T_1(\xi_1-\xi_2), T_1\xi_1\xi_2\}$ is a quasi-isomorphism of dgas. Hence the dga $R_2$ is formal.
\end{enumerate}
\end{example}

\subsubsection{The cohomology $H(\rk)$}

In this subsection we compute $H(\rk)$ as a graded vector space.  
There is an inclusion $R_{k-1} \hookrightarrow \rk$ given by tensoring with the $k$th horizontal strand at the top.
We fix this inclusion and view $R_{k-1}$ as a subalgebra of $\rk$.

\begin{lemma} \label{lem H(Rk)}
The cohomology $H(\rk)$ is isomorphic to the exterior algebra $\lk$ as a graded vector space, where the class $[\mb_k]$ corresponds to the identity of $\lk$.
\end{lemma}

\begin{proof}
The proof is by induction on $k$. The $k=1$ case is immediate.

Assume $k \ge 2$.
Let $G^p$ denote the subspace of $\rk$ spanned by diagrams containing a strand which connects the $k$th endpoint on the left to the $p$th endpoint on the right, where $1 \le p \le k$.
It is a left free $R_{k-1}$-module of rank $2$ with a basis $\{T_{k-1}\cdots T_{p}, T_{k-1}\cdots T_{p}\xi_p\}$, given by the left-hand side of Figure \ref{fig alg3} with the boxes removed.
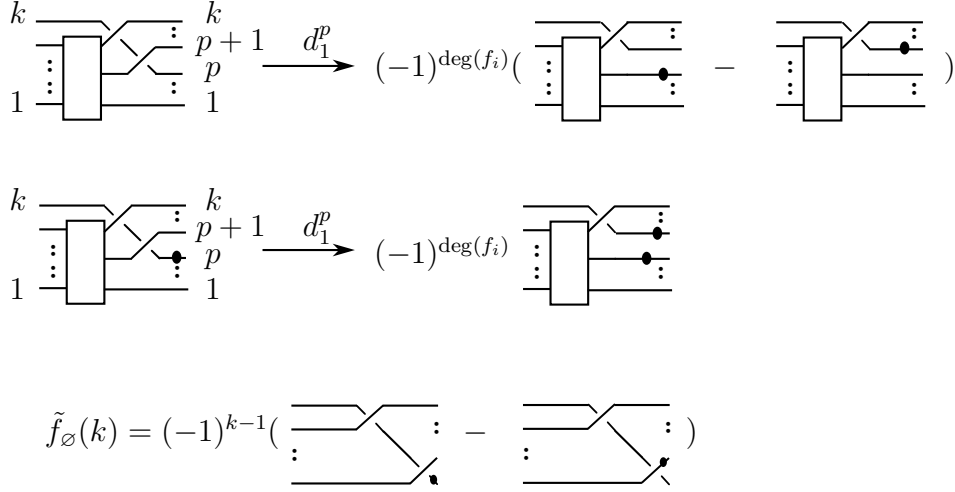
\begin{figure}[ht]
\begin{overpic}
[scale=0.3]{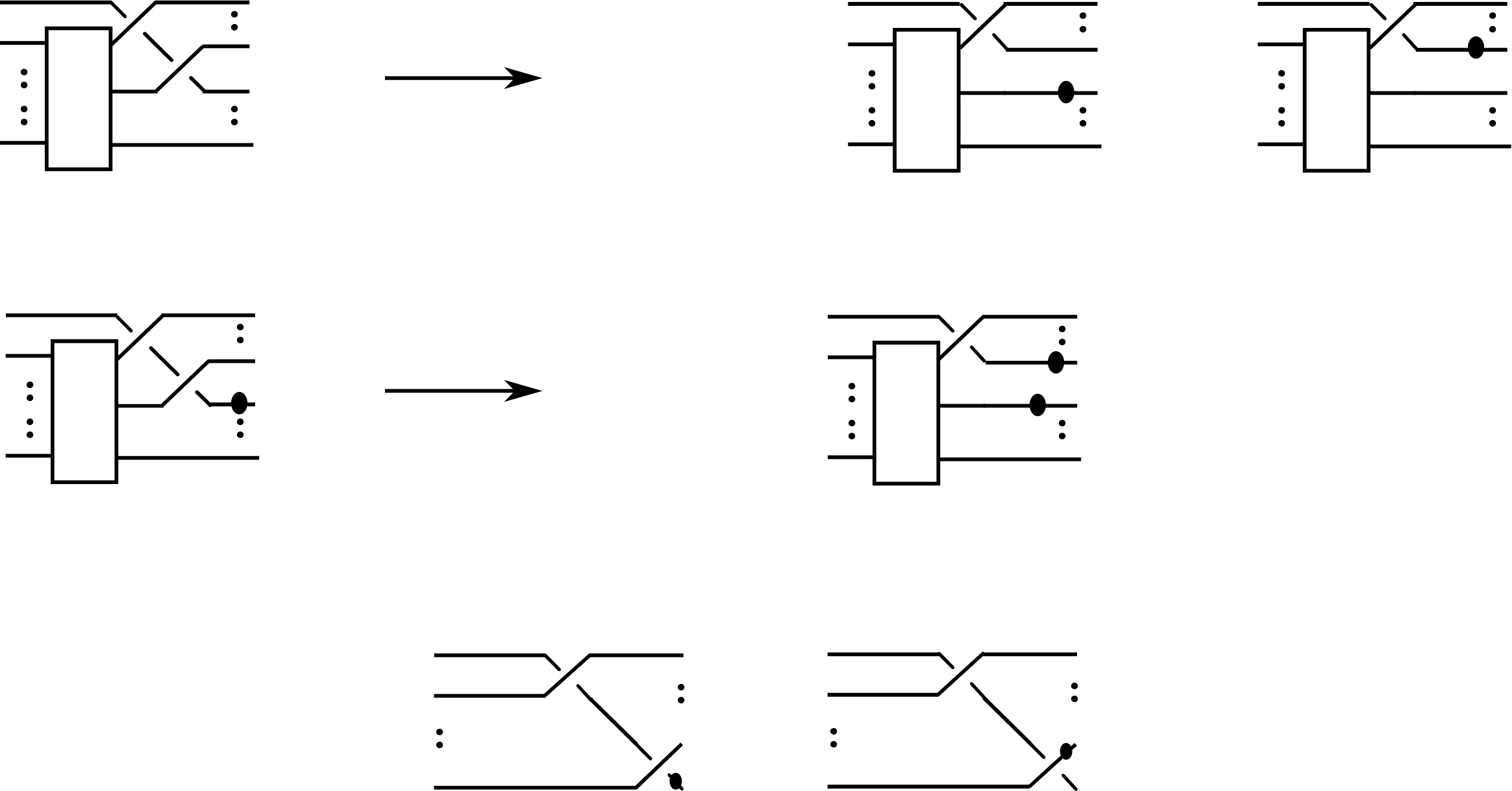}
\put(-3,21){$1$}
\put(-3,31){$k$}
\put(-3,42){$1$}
\put(-3,52){$k$}
\put(19,21){$1$}
\put(19,31){$k$}
\put(19,42){$1$}
\put(19,52){$k$}
\put(19,25){$p$}
\put(18,28){$p+1$}
\put(19,46){$p$}
\put(18,49){$p+1$}
\put(30,28){$d_1^p$}
\put(30,49){$d_1^p$}
\put(38,46){$(-1)^{\deg(f_i)}($}
\put(76,46){$-$}
\put(102,46){$)$}
\put(38,25){$(-1)^{\deg(f_i)}$}
\put(1,5){$\ti{f}_{\es}(k)=(-1)^{k-1}($}
\put(48,5){$-$}
\put(73,5){$)$}
\end{overpic}
\caption{The top two lines depict the basis $\{[f_iT_{k-1}\cdots T_{p}], [f_iT_{k-1}\cdots T_{p}\xi_p] \}$ of $E_1^p$ and their images under $d_1^p$, where each box denotes $f_i \in R_{k-1}$.}
\label{fig alg3}
\end{figure}

Consider the finite filtration $0 \subset F^k \subset \cdots \subset F^1 =\rk$, where $F^p=\textstyle{\oplus_{q \ge p}}G^q$ is a subcomplex of $\rk$. The corresponding finite spectral sequence starts with $E_0^p=F^p / F^{p+1} \cong G^p$. Since the differential $d_0^p$ does not change the crossings $T_{k-1}, \dots, T_{p}$ for any element of $G^p$, the complex $(E_0^p, d_0^p)$ is isomorphic to $R_{k-1} \oplus R_{k-1}[-1]$. By the inductive step the first page is:
\begin{equation} \label{eq iso E1}
E_1^p=H(E_0^p, d_0^p) \cong \lm_{k-1} \oplus \lm_{k-1}[-1],
\end{equation}
Now choose a linear basis $\{[f_i(k-1)]\}$ of $H(R_{k-1})$, where $f_i(k-1) \in R_{k-1}$ is homogeneous with a distinguished element $f_{\es}(k-1)=\mb_{k-1}$. We simply write $f_i$ for $f_i(k-1)$ when there is no confusion. Then $E_1^p$ has a basis $\{[f_iT_{k-1}\cdots T_{p}], [f_iT_{k-1}\cdots T_{p}\xi_p] \}$ and the differential $d_1^p: E_1^p \ra E_1^{p+1}$ resolves the rightmost crossing $T_p$:
\begin{align}
\label{panda1} d_1^p[f_iT_{k-1}\cdots T_{p}]&=(-1)^{\deg(f_i)}([f_i T_{k-1}\cdots T_{p+1}\xi_p]-[f_i T_{k-1}\cdots T_{p+1}\xi_{p+1}]), \\
\label{panda2} d_1^p[f_iT_{k-1}\cdots T_{p}\xi_p]&=(-1)^{\deg(f_i)}[f_i T_{k-1}\cdots T_{p+1}\xi_p\xi_{p+1}].
\end{align}
The terms $[f_i T_{k-1}\cdots T_{p+1}\xi_p], [f_i T_{k-1}\cdots T_{p+1}\xi_{p+1}]$  
on the right-hand side of \eqref{panda1} form a basis of $E_1^{p+1}$. Hence $\{d_1^p[f_iT_{k-1}\cdots T_{p}]\}$ is linearly independent and $\dim \im(d_1^p) \ge 2^{k-1}$ for $1 \le p \le k-1$.
This in turn implies that $\dim \Ker(d_1^p) \le 2^{k-1}$ for $1 \le p \le k-1$ since $\dim E_1^p =2\cdot2^{k-1}$.
Moreover, $\im(d_1^p) \subset \Ker(d_1^{p+1})$ since $d_1^{p+1} \circ d_1^p=0$.
%\marginpar{\cg Yin do really mean $d_1^{p+1}$?} 
Hence we have
$$\dim \im(d_1^p) = \dim \Ker(d_1^p) = 2^{k-1}, \quad 1 \le p \le k-1.$$
Therefore, the complex $(E_1^p, d_1^p)$ is exact except at $p=1, k$, and
$$\dim E_2^p=\left\{ 
\begin{array}{cl}
2^{k-1}, & \mbox{for } p=1,k; \\
0, & \mbox{otherwise}.
\end{array}
\right.$$

A basis of $E_2^1=\Ker(d_1^1)$ is given by $\{[\ti{f}_i]\}$, where $\ti{f}_i=f_iT_{k-1}\cdots T_{1}(\xi_1-\xi_2)$.
The distinguished element is
\begin{gather} \label{eq def fes}
\ti{f}_{\es}(k)=(-1)^{k-1}T_{k-1}\cdots T_{1}(\xi_1-\xi_2) \in \rk,
\end{gather}
which is of degree $1$.
A basis of $E_2^k=\op{Coker}(d_1^{k-1})$ is given by $\{[f_i]\}$.

Since $E_2^p=0$ for $p \neq 1,k$, we have $d_r^p=0$ for $2 \le r \le k-2$, and $E_{k-1}^1\cong E_{2}^1, E_{k-1}^k \cong E_{2}^k$.
The only possibly nontrivial differential is $d_{k-1}^1: E_{k-1}^1 \ra E_{k-1}^k$ such that 
$$d_{k-1}^1([\ti{f}_i])=(-1)^{\deg(f_i)}[f_i]d_{k-1}^1([\ti{f}_{\es}]).$$

We claim that $d_{k-1}^1=0$. It suffices to show that there exists $h_k \in \rk$ for $k \ge 2$ which satisfies
\begin{gather} \label{eq prop hk}
d(h_k)=0, \quad \deg(h_k)=1, \quad h_k-\ti{f}_{\es}(k) \in F^2,
\end{gather}
so that $[h_k]=[\ti{f}_{\es}(k)] \in E_2^1 \cong E_{k-1}^1$.
The element $h_k\in \rk$ will be constructed below.
Assuming that $d_{k-1}^1=0$, the last page $E_k^p$ is isomorphic to $E_2^p$.
The finite spectral sequence converges to $H(\rk)$, which has a basis $\{[f_i], [f_ih_k]\}$, where $\{[f_i]\}$ is a basis of $H(R_{k-1}) \cong \lm_{k-1}$ and $\deg(h_k)=1$. Hence $H(\rk)\cong \lk$ as graded vector spaces with a distinguished class $[f_{\es}]=[\mb_k]$.
\end{proof}

We construct $h_k$ (via elements $h_{k,i}\in \rk$) as follows:

\begin{defn} \label{def hk}
For $k=1$, define $h_{1,1}=\xi_1$, and $h_{1,i}=0$ for $i \neq 1$. 
For $k>1$, define $h_{k,i} \in \rk$ inductively as follows:
$$h_{k,i}=h_{k-1,i}-T_{k-1}^{-1}h_{k-1,i-1}+h_{k-1,i-1}T_{k-1}-T_{k-1}^{-1}h_{k-1,i-2}T_{k-1},$$
for $i \in \Z$. See Figure \ref{fig alg4a}, where $T_{k-1}^{-1}$ is drawn as a negative crossing. We view $R_{k-1} \subset \rk$ via the natural inclusion.
\end{defn}

\begin{figure}[ht]
\begin{overpic}
[scale=0.4]{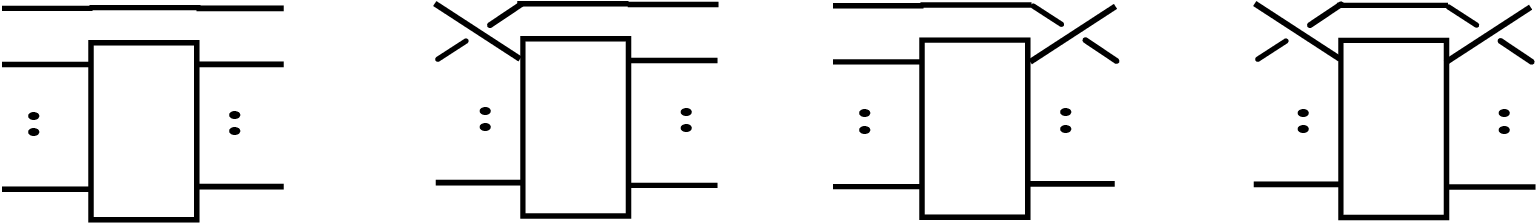}
\put(-15,4){${h_{k,i}=}$}
\put(8.8,4){${\scriptstyle i}$}
\put(35,4){${\scriptstyle i-1}$}
\put(60.8,4){${\scriptstyle i-1}$}
\put(88,4){${\scriptstyle i-2}$}
\put(22,4){${\scriptstyle -}$}
\put(49,4){${\scriptstyle +}$}
\put(75,4){${\scriptstyle -}$}
\end{overpic}
\caption{The definition of $h_{k,i} \in \rk$, where a box with $j$ represents $h_{k-1,j} \in R_{k-1}$.}
\label{fig alg4a}
\end{figure}

The elements $h_{k,i}$ are all of degree one. By definition, $h_{k,i}=0$ for $i \le 0$ or $ i \ge 2k$. Moreover, $h_{k,i}=h_{k-1,i} \in R_{k-1} \subset \rk$ for $1 \le i \le k-1$ since $T_{k-1}-T_{k-1}^{-1}=\hbar$.

\begin{lemma} \label{lem dhk}
The element $h_{k, i} \in \rk$ is closed for $1 \le i \le k$.
\end{lemma}

\begin{proof}
We prove that $d(h_{k,i})=\xi_kh_{k,i-1}+h_{k,i-1}\xi_k,$ for all $k,i$ by induction on $k$.
The graphical calculation is given in Figure~\ref{fig alg4}, where the dot can slide through the negative crossing $T_{k-1}^{-1}$ from above since $\xi_kT_{k-1}^{-1}=T_{k-1}^{-1}\xi_{k-1}$.
Hence $d(h_{k,i})=0$ for $1 \le i \le k-1$ since $h_{k,i} \in R_{k-1} \subset \rk$.
\end{proof}
\begin{figure}[ht]
\begin{overpic}
[scale=0.3]{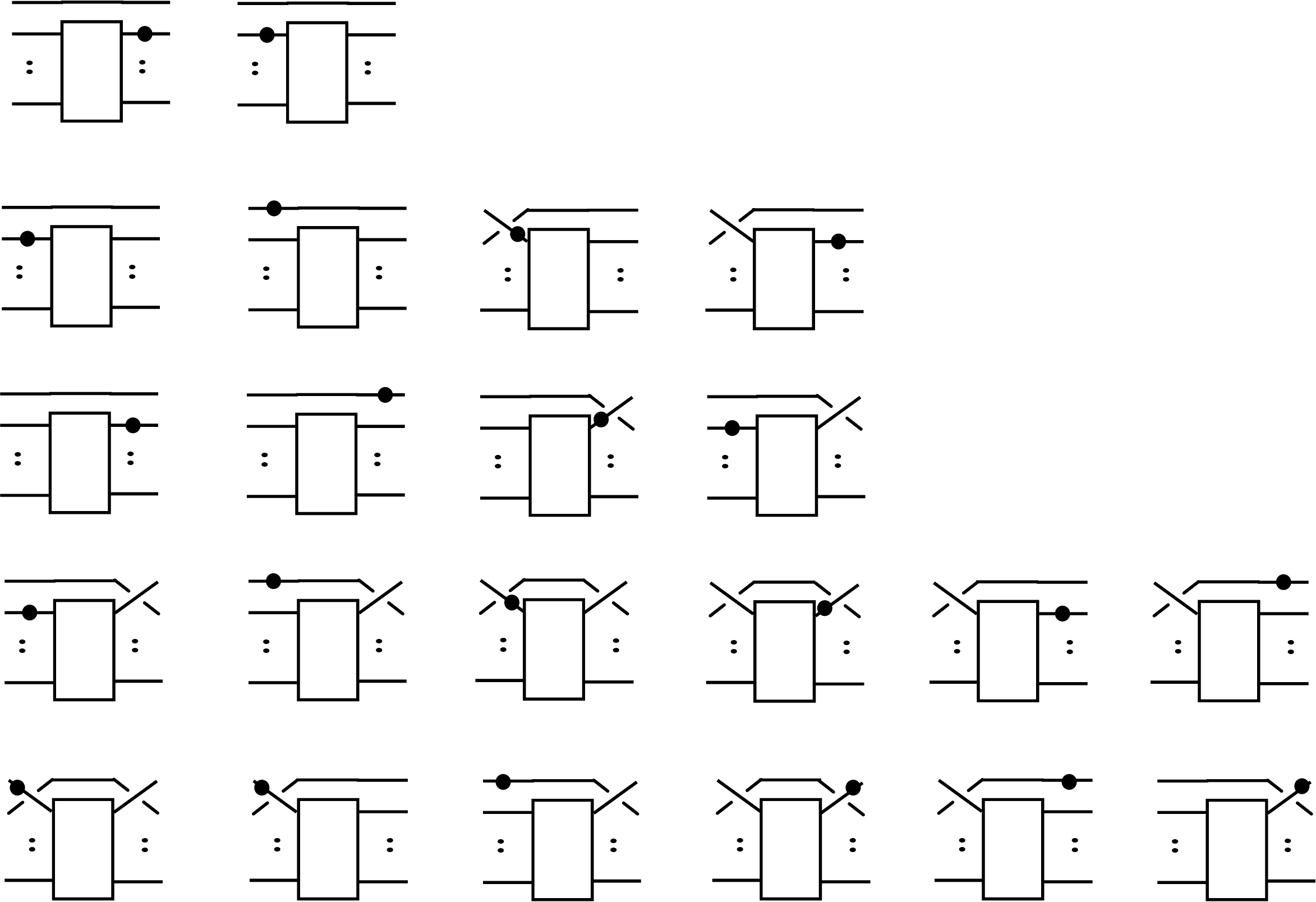}
\put(-14,60){${d(h_{k,i})=}$}
\put(5,62){${\scriptscriptstyle i-1}$}
\put(22.5,62){${\scriptscriptstyle i-1}$}
\put(14.3,60){${\scriptstyle +}$}
\put(-3,47){${\scriptstyle -}$}
\put(14.3,47){${\scriptstyle +}$}
\put(32.7,47){${\scriptstyle -}$}
\put(50,47){${\scriptstyle -}$}
\put(4.5,47){${\scriptscriptstyle i-1}$}
\put(23.2,47){${\scriptscriptstyle i-1}$}
\put(40.5,47){${\scriptscriptstyle i-2}$}
\put(57.8,47){${\scriptscriptstyle i-2}$}
\put(-3,33){${\scriptstyle -}$}
\put(14.3,33){${\scriptstyle +}$}
\put(32.7,33){${\scriptstyle +}$}
\put(50,33){${\scriptstyle +}$}
\put(4.5,33){${\scriptscriptstyle i-1}$}
\put(23.3,33){${\scriptscriptstyle i-1}$}
\put(40.8,33){${\scriptscriptstyle i-2}$}
\put(58,33){${\scriptscriptstyle i-2}$}
\put(-3,20){${\scriptstyle -}$}
\put(14.3,20){${\scriptstyle +}$}
\put(32.7,20){${\scriptstyle -}$}
\put(50,20){${\scriptstyle -}$}
\put(67.5,20){${\scriptstyle +}$}
\put(84.3,20){${\scriptstyle -}$}
\put(4.5,19){${\scriptscriptstyle i-2}$}
\put(23,19){${\scriptscriptstyle i-2}$}
\put(40.5,19){${\scriptscriptstyle i-3}$}
\put(58,19){${\scriptscriptstyle i-3}$}
\put(75,19){${\scriptscriptstyle i-2}$}
\put(92,19){${\scriptscriptstyle i-2}$}
\put(-7,4){$= {\scriptstyle -}$}
\put(14.3,4){${\scriptstyle -}$}
\put(32.7,4){${\scriptstyle +}$}
\put(50,4){${\scriptstyle -}$}
\put(67.65,4){${\scriptstyle -}$}
\put(84.3,4){${\scriptstyle +}$}
\put(4.5,4){${\scriptscriptstyle i-3}$}
\put(23,4){${\scriptscriptstyle i-2}$}
\put(41,4){${\scriptscriptstyle i-2}$}
\put(58.3,4){${\scriptscriptstyle i-3}$}
\put(75.3,4){${\scriptscriptstyle i-2}$}
\put(92.3,4){${\scriptscriptstyle i-2}$}
\end{overpic}
\caption{The computation $d(h_{k,i})=\xi_kh_{k,i-1}+h_{k,i-1}\xi_k$ by induction on $k$.}
\label{fig alg4}
\end{figure}

\begin{lemma} \label{lem hk}
The element $h_k=h_{k,k} \in \rk$ for $k \ge 2$ satisfies (\ref{eq prop hk}).
\end{lemma}

\begin{proof}
We have $d(h_k)=0$ by Lemma \ref{lem dhk}.
For $k=2$, 
$$h_2=-T_1^{-1}\xi_1+\xi_1T_1=T_1(-\xi_1+\xi_2)+\hbar\xi_1\equiv\ti{f}_{\es}(2) \mbox{ mod } F^2,$$ 
as in (\ref{eq def fes}).
In general, 
\begin{align*}
    h_k &\equiv -(T_{k-1}^{-1}h_{k-1,k-2}T_{k-1}+T_{k-1}^{-1}h_{k-1,k-1}) \equiv -T_{k-1}h_{k-1,k-1} \\
    &\equiv -T_{k-1}\ti{f}_{\es}(k-1)\equiv \ti{f}_{\es}(k) \mbox{ mod } F^2,
\end{align*}
where 
\begin{gather*}
T_{k-1}^{-1}h_{k-1,k-2}T_{k-1}=h_{k-1,k-2} \in R_{k-2},\\ 
T_{k-1}h_{k-1,k-1}-T_{k-1}^{-1}h_{k-1,k-1}=\hbar \cdot h_{k-1,k-1} \in R_{k-1},
\end{gather*}
and the third equality is obtained by induction on $k$.
\end{proof}

We now compute $H(\rk)$ as an algebra. 

\begin{prop} \label{prop H(Rk)}
The cohomology algebra $H(\rk)$ is isomorphic to an exterior algebra generated by the classes of $h_{k,1}, \dots, h_{k,k}$ if $\op{char}(\F) \neq 2$.
\end{prop}

\begin{proof}
The proof is by induction on $k$. The $k=1$ case is immediate.  

Arguing by induction, suppose that $H(R_{k-1}) \cong \lm_{k-1}$ as algebras. The space $H(\rk)$ has a linear basis $\{[f_i], [f_ih_{k,k}]\}$, where $\{[f_i]\}$ is a basis of $H(R_{k-1})$. It suffices to show that
\begin{gather} \label{eq HRk}
[f_ih_{k,k}]=(-1)^{\deg(f_i)}[h_{k,k}f_i], \quad [h_{k,k}]^2=0,
\end{gather}
which will be proved using the following two claims:

\s\n {\bf Claim 1.} For any closed $h \in R_{k-1}$, there exists a closed $h' \in R_{k-1}$ such that $[h' \ot 1]=[1 \ot h] \in H(\rk)$, where
 $h' \ot 1$ and $1 \ot h$ are given by stacking a trivial strand above $h'$ and below $h$, respectively.

The spectral sequence from the proof of Lemma \ref{lem H(Rk)} gives a short exact sequence $0 \ra E_2^k \ra H(\rk) \ra E_2^1 \ra 0$.
The element $1\ot h$ lies in $F^2$ and the image of its class $[1\ot h]$ in $H(\rk)$ lies in $E_2^k$ under the inclusion.
Then observe that $E_2^k$ has a basis $\{[f_i]=[f_i \ot 1]\}$.

\s\n {\bf Claim 2.} The inclusion $\imath_{k-1}: R_{k-1} \ra \rk$, $\imath_{k-1}(f)=1 \ot f$ induces an inclusion $H(R_{k-1})\to H(\rk)$.
This is immediate from the above bases of $H(R_{k-1})$ and $H(\rk)$.

\s
By a repeated application of Claim 2, we have an inclusion $H(\rk) \ra H(R_{2k})$ given by $f \mapsto \mb_{k} \ot f$.
Hence it suffices to show that the image of \eqref{eq HRk} holds in $H(R_{2k})$ under the inclusion.  We have
$$ [\mb_{k} \ot h_{k,k}]^2=[\mb_{k} \ot h_{k,k}][h_{k,k}' \ot \mb_k]=-[h_{k,k}' \ot \mb_k][\mb_{k} \ot h_{k,k}]=-[\mb_{k} \ot h_{k,k}]^2,$$
where the first equation is from Claim 1, and the second one is from super-commutativity of disjoint diagrams.
Hence $[h_{k,k}]^2=0$ if $\op{char}(\F) \neq 2$.
The proof for the other case is similar.
\end{proof}

%Although the proof does not go through when $\op{char}(\F) = 2$, we still conjecture that the same result holds.
%We construct a family of closed elements $h_1, \dots, h_k$ of degree $1$ in the proof of Lemma \ref{lem hk}.
%Note that $h_i \in R_i \subset \rk$ for $1 \le i \le k$.

Finally we conjecture that the dga $\rk$ is formal.

\begin{conj} \label{conj Rk formal}
The elements $h_{k,1}, \dots, h_{k,k}$ of $\rk$ satisfy 
$$h_{k,i}^2=0; \quad h_{k,i}h_{k,j}=-h_{k,j}h_{k,i}, i \neq j.$$
Hence the map $H(\rk) \ra \rk$ given by $[h_{k,i}] \mapsto h_{k,i}$ is a quasi-isomorphism.

%\n (2) The Hochschild homology $HH(\rk) \cong HH(\lk) \cong \lk \ot P_k$, where $P_k=\F[y_1,\dots,y_k]$ is the polynomial ring.
\end{conj}

\subsection{The $q$-graded version} \label{ssec rkkh} 

\subsubsection{Definition of $\rkkh$} 

Recall the algebra $\rk$ has a linear basis 
$$\{T_w\xi~|~ w \in S_k, \xi ~\mbox{monomial in}~\lk\}.$$ 
We introduce an additional $q$-grading on the underlying vector space as follows:
\begin{equation}
\deg_{q}(T_w \xi)=-2(l(w)+\deg(\xi)).
\end{equation}
Here $\deg(\xi)$ is the homological grading of $\xi$ and $l(w)$ is the length of $w \in S_k$. It is easy to verify that $\rk$ is a filtered algebra with respect to the $q$-grading. We denote the images of $T_i$ and $\xi_j$ in the associated graded algebra by $s_i$ and $\xi_j$. Then $s_1,\dots,s_{k-1}$ generate the nilCoxeter algebra $\nck$, and $\xi_1,\dots,\xi_k$ generate the exterior algebra $\lk$. The associated $q$-graded algebra has generators $s_i, \xi_j$, $1 \leq i \leq k-1, 1 \leq j \leq k$, and satisfies the following defining relations:
\begin{gather} \label{rel Rkkh}
s_i^2=0, \quad s_is_{i+1}s_i=s_{i+1}s_is_{i+1}, \quad  s_is_{i'}=s_{i'}s_i ~\mbox{for}~|i-i'|>1; \\
\nonumber \xi_j^2=0, \quad \xi_j\xi_{j'}=-\xi_{j'}\xi_j ~\mbox{for}~j \neq j';   \\
\nonumber \xi_is_i=s_i\xi_{i+1},  \quad \xi_{i+1}s_i=s_i\xi_{i}, \quad \xi_js_i=s_i\xi_j ~\mbox{for}~j \neq i,i+1.
\end{gather}
The differential $d$ of $\rk$ preserves the $q$-grading and hence induces a differential on the associated $q$-graded algebra, which is given by
$$d(s_i)=\xi_i-\xi_{i+1}, \qquad d(\xi_j)=0.$$
The homological grading is given by $\deg(\xi_j)=1, \deg(s_i)=0$, and the $q$-grading is given by $\deg_q(\xi_j)=-2, \deg_q(s_i)=-2$.
We denote this $q$-graded dga by $\rkkh$.

\begin{figure}[ht]
\begin{overpic}
[scale=0.25]{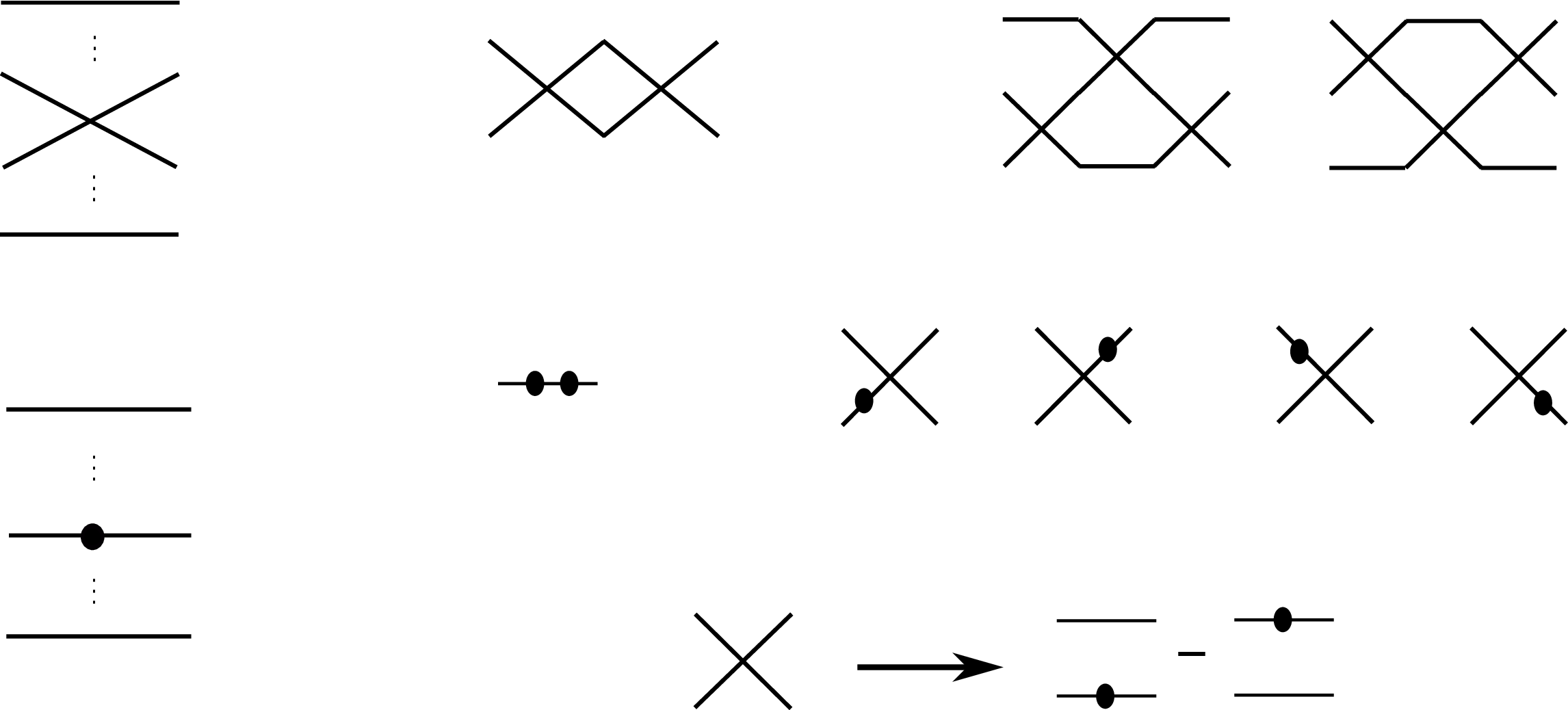}
\put(5,26){$s_i$}
\put(5,0){$\xi_j$}
\put(-3,4){${\scriptstyle 1}$}
\put(-3,10){${\scriptstyle j}$}
\put(-3,18){${\scriptstyle k}$}
\put(-3,29){${\scriptstyle 1}$}
\put(-3,35){${\scriptstyle i}$}
\put(-5,40){${\scriptstyle i+1}$}
\put(-3,44){${\scriptstyle k}$}
\put(57, 4){$d$}
\put(48,38){$=0$}
\put(40,20){$=0$}
\put(80,38){${=}$}
\put(89,20){${=}$}
\put(62,20){${=}$}
\end{overpic}
\caption{The generators $s_i, \xi_j$ of $\rkkh$ are on the left and the relations and the differential are on the right.}
\label{fig algn5}
\end{figure}
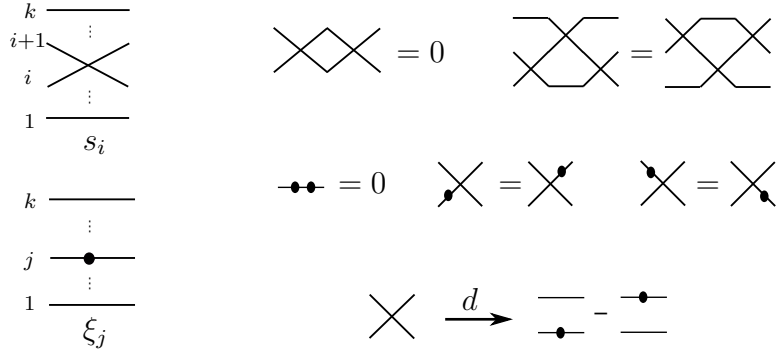

All the results for $\rk$ in the previous subsection still hold for $\rkkh$ and the proofs are the same. In particular, there is a family of elements $h_{k,i} \in \rkkh$ as in Definition \ref{def hk}. Each $h_{k,i}$ is homogeneous for both gradings: $\deg(h_{k,i})=1, \deg_q(h_{k,i})=-2i$. Among them, $h_{k,i}$ is closed for $1 \le i \le k$ from Lemma \ref{lem dhk}. The cohomology $H(\rkkh)$ has a linear basis 
$$\{[h_{k, i_1}\cdots h_{k, i_s}] ~|~ 1 \le i_1 < \cdots < i_s \le k, 0 \le s \le k\}$$ 
from Lemma \ref{lem H(Rk)}. 
We state the analog of Proposition \ref{prop H(Rk)} for later use:

\begin{prop} \label{prop H(Rkkh)}
Assuming $\op{char}(\F) \neq 2$, the cohomology algebra $H(\rkkh)$ is isomorphic to an exterior algebra generated by $[h_{k,i}]$, $1 \le i \le k$, with $\deg([h_{k,i}])=1$ and $\deg_q([h_{k,i}])=-2i$.
\end{prop}

%The analog of Conjecture \ref{conj Rk formal} holds for $\rkkh$. 
%
%\begin{prop}\label{prop rkkh formal}
%The elements $h_{k,i}$ of $\rkkh$ satisfy 
%$$h_{k,i}^2=0; \quad h_{k,i}h_{k,j}=-h_{k,j}h_{k,i}, i \neq j.$$
%Hence the map $H(\rkkh) \ra \rkkh$ given by $[h_{k,i}] \mapsto h_{k,i}$ is a quasi-isomorphism. 
%\end{prop}
%
%The proof is a little technical; see Appendix. 

We denote this $q$-graded dga $H(\rkkh)$ by $\lkkh$ and write $\lkkh=\lk([h_{k,1}], \dots, [h_{k,k}])$ to indicate the generators.

\subsubsection{Relationship to the nilHecke algebra}

We now discuss the relationship of $\rkkh$ to the nilHecke algebra $\nhk$ and the categorified quantum $\mathfrak{sl}_2$. 
Recall that $\nhk$ is a $q$-graded algebra with generators $\bdry_i, x_j$, $1 \leq i \leq k-1, 1 \leq j \leq k$, and defining relations: 
\begin{gather}\label{rel nhk}
\bdry_i^2=0, \qquad  \bdry_i\bdry_{i+1}\bdry_i=\bdry_{i+1}\bdry_i\bdry_{i+1}, \qquad \bdry_i\bdry_{i'}=\bdry_{i'}\bdry_i ~\mbox{for}~|i-i'|>1; \\
\nonumber x_jx_{j'}=x_{j'}x_j, ~\mbox{for}~j \neq j';  \\
\nonumber x_i\bdry_i-\bdry_ix_{i+1}=1,  \qquad  x_{i+1}\bdry_i-\bdry_ix_{i}=-1, \qquad x_j\bdry_i=\bdry_ix_j ~\mbox{for}~j \neq i,i+1.
\end{gather}
The $q$-grading is given by $\deg_q(x_j)=2, \deg_q(\bdry_i)=-2$. 

The nilHecke algebra plays a fundamental role in the categorification of quantum groups. We briefly summarize some results of Lauda~\cite{La}.
There are two subalgebras of $\nhk$: the nilCoxeter algebra $\nck$ generated by the $\bdry_i$, and the polynomial algebra $\polk=\F[x_1, \dots, x_k]$. 
The center of $\nhk$ is the algebra $\smk=\F[e_1, \dots, e_k]$ of symmetric polynomials of $x_1, \dots, x_k$, where $e_i$ is the $i$th elementary symmetric polynomial. 
The space $\polk$ is a free module over $\smk$ of rank $k!$, up to a grading shift.  Lauda showed that $\nhk$ is naturally isomorphic to $\End_{\smk}(\polk)$, which is further isomorphic to the matrix algebra $\op{Mat}(k!; \smk)$ up to a grading shift.

\subsubsection{Main result}

Let $\upl$ denote the additive monoidal category which categorifies the positive half of the quantum $\mathfrak{sl}_2$ from \cite{La}. It has objects $\E^{\otimes k}$ for $k \ge 0$ such that $\End(\E^{\otimes k})=\nhk$. 
Let $\dgupl$ be the dg category of complexes over $\upl$. It has a natural monoidal structure induced from $\upl$. 
Define the two-term complex:
\begin{equation} \label{def le}
\lle:=(\E \xra{x_1} \E),
\end{equation}   
where the two components are in cohomological degrees $0$ and $1$, and the differential is given by $x_1 \in \End(\E)$. 
Let $\End(\lek)$ denote the endomorphism algebra of $\lek$ in $\dgupl$ which is a $q$-graded dga. 

We first construct a family of natural inclusions of $q$-graded dgas 
$$\rho_k: \rkkh \hookrightarrow \End(\lek).$$ 
For $k=1$, $\End(\E) \cong \F[x_1]$, and $H(\End(\lle)) \cong \F[\zeta_1] / (\zeta_1^2)$. We view $\lle$ as the Koszul dual of $\E$.  
Define $\rho_1(\xi_1)$ as the generator $\zeta_1 \in H(\End(\lle))$, i.e.,  the map labeled by $\id$ between two rows representing $\lle$: 
$$\xymatrix{
\lle: & \E \ar[r]^{x_1} & \E  \\
\lle:  \ar[u]^{\rho_1(\xi_1)} &\E \ar[r]^{x_1} \ar[ur]^{\id} & \E.
}$$
For $k=2$, the images $\rho_2(\xi_1), \rho_2(\xi_2)$ can be computed by taking tensor products:
$$\xymatrix{
\lle^{\otimes 2}: & \E^2 \ar[rr]^{\scriptscriptstyle \begin{bmatrix} x_1, x_2 \end{bmatrix}} & &( \E^2 & \oplus & \E^2 ) \ar[rr]^{\scriptscriptstyle \begin{bmatrix} x_ 2 \\ -x_1 \end{bmatrix}} & & \E^2  \\
\lle^{\otimes 2}:  \ar@/_1pc/@{-->}[u]_{\rho_2(\xi_2)} \ar@/^1pc/[u]^{\rho_2(\xi_1)} & \E^2 \ar[urr]^{\id} \ar@{-->}[urrrr]^{\id} \ar[rr]_{\scriptscriptstyle \begin{bmatrix} x_1, x_2 \end{bmatrix}} & &( \E^2 \ar@{-->}[urrrr]^{-\id}  & \oplus & \E^2 ) \ar[urr]_{\id} \ar[rr]_{\scriptscriptstyle \begin{bmatrix} x_ 2 \\ -x_1 \end{bmatrix}} & & \E^2,
}$$ 
where each row is $\lle^{\otimes 2}$ --- a complex of four copies of $\E^2$ with a differential given by the $x_i$, and the solid and dashed lines between the two rows are $\rho_2(\xi_1)$  and $\rho_2(\xi_2)$, respectively. 
The image $\rho_2(s_1)$ consists of four components of $\bdry_1$ between the two rows:
$$\xymatrix{
\lle^{\otimes 2}: &  \E^2 \ar[rr]^{\scriptscriptstyle \begin{bmatrix} x_1, x_2 \end{bmatrix}} & & (\E^2 & \oplus & \E^2) \ar[rr]^{\scriptscriptstyle \begin{bmatrix} x_ 2 \\ -x_1 \end{bmatrix}} & & \E^2  \\
\lle^{\otimes 2}: \ar[u]^{\rho_2(s_1)} & \E^2 \ar[u]^{\bdry_1}  \ar[rr]_{\scriptscriptstyle \begin{bmatrix} x_1, x_2 \end{bmatrix}} & & (\E^2 \ar[urr]^{\bdry_1} & \oplus & \E^2) \ar[ull]^{\bdry_1} \ar[rr]_{\scriptscriptstyle \begin{bmatrix} x_ 2 \\ -x_1 \end{bmatrix}} & & \E^2 \ar[u]^{-\bdry_1},
}$$  
It is easy to verify that $\rho_2$ preserves the defining relations and the $q$-grading of $\rkkh$. Moreover, 
$$d(\rho_2(s_1))=d \circ \rho_2(s_1)-\rho_2(s_1)\circ d=\rho_2(\xi_1)-\rho_2(\xi_2).$$
Therefore we have a map of $q$-graded dgas $\rho_2: R_2^{\op{nil}} \ra \End(\lle^{\otimes 2})$. 
The construction can be extended to general $k$ by taking tensor products as follows: 
$$\rho_k(\xi_j)=\id^{\otimes j-1} \otimes \rho_1(\xi_1) \otimes \id^{\otimes k-j}, \quad \rho_k(s_i)=\id^{\otimes i-1} \otimes \rho_2(s_1) \otimes \id^{\otimes k-i-1},$$ 
where each $\id$ is $\id_{\lle}$.  We leave it to the reader to verify that $\rho_k: \rkkh \hookrightarrow \End(\lek)$ is a well-defined map of $q$-graded dgas. 

Now we give another description of the action $\rho_k$ of $\rkkh$ on $\lek$. Let us write the tensor product $\lek$ as $(\lk(\xi_1, \dots, \xi_k) \otimes \E^k, d)$, where $\lk(\xi_1, \dots, \xi_k)$ is the exterior algebra generated by the $\xi_j$ and the differential is $d=\textstyle{(\sum_{i=1}^{k} \xi_i \otimes x_i)}\cdot -$. Then the action is given by $\rho_k(\xi_j)=\xi_j \cdot -$, the multiplication by $\xi_j$ on the left; and $\rho_k(s_i)(\xi \otimes -)=s_i(\xi) \otimes (\bdry_i \cdot -)$, where the first factor is the symmetric group action on $\lk$, and the second is the nilCoxeter action.  

The algebra $\rkkh$ has a linear basis $\{\xi s_{w} ~|~ w \in S_k, \xi \in \lk ~\mbox{monomials}\}$,
where $s_w$ corresponds to $T_w$. 
Consider the restriction of $\rho_k(\xi s_{w})=\rho_k(\xi) \rho_k(s_{w})$ on the component $\E^k$ of the lowest cohomological degree. The factor $\rho_k(s_{w})$ is $\bdry_{w} \in \nck$, and $\rho_k(\xi)$ determines which component it maps into. Hence $\rho_k$ is an inclusion. 

\begin{proof}[Proof of Theorem \ref{thm rkkh dual}.]
We first compute $H(\End(\lek))$. Consider 
$$\lek=(\lk(\xi_1, \dots, \xi_k) \otimes \E^k, d)$$ 
as above. There are natural isomorphisms
$$\End(\E^k) \cong \nhk \cong \End_{\smk}(\polk),$$
where $\smk=\F[e_1, \dots, e_k]$. 
The composition induces a natural isomorphism of $q$-graded dgas
$$\End(\lek) \cong \End_{\cal{DG}(\smk)}(M_k),$$ 
where $\cal{DG}(\smk)$ is the dg category of complexes of graded free modules over $\smk$ and 
$$M_k=(\lk(\xi_1, \dots, \xi_k) \otimes \polk, d=\textstyle{(\sum_{i=1}^{k} \xi_i \otimes x_i)\cdot -}).$$
Since $M_k$ is a free resolution of the trivial module $\F$ over $\smk$, we have the following isomorphisms of $q$-graded dgas:
$$H(\End(\lek)) \cong H(\End_{\cal{DG}(\smk)}(M_k))  \cong H(\End_{\cal{D}(\smk)}(\F)) \cong \lk(e_1^*, \dots, e_k^*),$$
where $\cal{D}(\smk)$ is the derived category of modules over $\smk$ and $\lk(e_1^*, \dots, e_k^*)$ is the exterior algebra generated by the duals of the $e_i$ in $\smk$. The last isomorphism follows from the Koszul duality between $\smk$ and $\lk(e_1^*, \dots, e_k^*)$.  
The gradings are $\deg(e_i^*)=1, \deg_q(e_i^*)=-2i$. In particular, $H(\End(\lek))$ has a volume form $e_1^* \cdots e_k^*$, which has top cohomological grading $k$ and lowest $q$-grading $-k(k+1)$.

Recall that $H(\rkkh)=\lkkh=\lk([h_{k,1}], \dots, [h_{k,k}])$ from Proposition \ref{prop H(Rkkh)}. Both $q$-graded vector spaces $H(\rkkh)$ and $H(\End(\lek))$ has the same graded dimension. It remains to show that the algebra homomorphism ${\rho_k}_*: H(\rkkh) \ra H(\End(\lek))$ is injective. Since both algebras are exterior algebras, it suffices to show that ${\rho_k}_*$ preserves the volume form. The lowest $q$-grading of $\rkkh$ is $-k(k+1)$. The corresponding component is one-dimensional and is generated by $\xi_1 \cdots \xi_k s_{w_0}$, where $w_0$ is the longest word in $S_k$. Its class gives the volume form of $H(\rkkh)$ up to a scalar.  The lowest $q$-grading of $\End(\lek)$ is $-k(k+1)$, and the corresponding component is one-dimensional as well. It is easy to see that $\rho_k(\xi_1 \cdots \xi_k s_{w_0})$ is nonzero. 
%It is easy to show that the volume form $[h_{k,1}]\cdots [h_{k,k}]=\pm [\xi_1 \cdots \xi_k s_{w_0}]$, where $w_0$ is the longest word in $S_k$. It also has the top cohomological grading $k$ and the lowest $q$-grading $-k(k+1)$. The corresponding summand of $\rkkh$ is one-dimensional generated by the volume form. The same is true for $\End(\lek)$. 
We conclude that ${\rho_k}_*$ preserves the volume form up to scalars. Hence $\rho_k$ is a quasi-isomorphism. 

The dga $\End(\lek)$ is formal and quasi-isomorphic to $H(\End(\lek)) \cong \lk(e_1^*, \dots, e_k^*)$. The second claim follows.
\end{proof}

\begin{rmk} \label{rmk Koszul} 
The restriction $\op{char}(\F) \neq 2$ in Theorem \ref{thm rkkh dual} can be dropped.  We will explore it in future work. 
\end{rmk}

\section{The $A_n$-singularity} \label{sec An}

We discuss the connection to symplectic Khovanov homology in this section. 
The corresponding surface is a disk with $n$ singular points and with a stop of one point on the boundary, i.e., $S = D^2, |\tau| = 1$. Let $\pi:W_n=W_{S,n}\to S=D^2$ be the Lefschetz fibration over $D^2$ with critical values $x_1,\dots, x_n$ and regular fibers $A=S^1\times[-1,1]$. 
Let ${\bf a}=\{a_1,\dots, a_n\}$ be a pairwise disjoint collection of arcs in $D^2$ where $a_i$ is from $x_i$ to $\bdry D^2\setminus  \tau$ and let $L_i$ denote the corresponding Lagrangian thimble over $a_i$.

Let $\mfn=\{1,2,\dots,n\}$ and $\pnk$ be the set of $k$-element subsets of $\mfn$, where $0 \leq k \leq n$.
For $S \in \pnk$, let $L_S$ be the collection $\{L_i~|~i \in S\}$, viewed informally as the ``tensor product'' of the $L_i$.  Then we have
$$R^\tau(D^2,n,{\bf a};k)= \oplus_{S,T\in\pnk}\widehat{CF}^\tau(L_S,L_T).$$

We define a family of dgas $\rnk$ for $0 \le k \le n$ as a conjectural algebraic description of $R^\tau(D^2,n,{\bf a};k)$.
There is also an associated $q$-graded version $\rnkkh$. 
As $\rkkh$ is related to the nilHecke algebra $\nhk$ and the categorified quantum $\mathfrak{sl}_2$, $\rnkkh$ is expected to be related to the categorification of the tensor product representation $V^{\ot n}$ of quantum $\mathfrak{sl}_2$, where $V$ is the fundamental representation of $\mathfrak{sl}_2$.

\subsection{The basic version}

The dga under construction is a modification of the {\em strands algebra} $\cal{A}(n,k)$; see \cite[Section 3.1]{LOT}.
Recall that the strands algebra $\cal{A}(n,k)$ has a linear basis $(S, T, \phi)$, where $S, T \in \pnk$ and $\phi: S \ra T$ is a nondecreasing bijection. Such a generator can be described as a diagram of $k$ linear strands possibly with crossings on $[0,1] \times [1,n]$, which connect $\{0\} \times s$ to $\{1\} \times \phi(s)$ for all $s\in S$. %intersects the left and right boundaries along $\{0\} \times S$ and $\{1\} \times T$, respectively.
The multiplication is given by concatenating two strands diagrams and straightening; if the number of crossings is reduced during the straightening we set it equal to zero. 
%Note that a generator may have different diagrammatic presentations which are isotopic to each other.
The strand associated to $s\in S$ is {\em increasing} if $\phi(s)>s$.

We modify the strands algebra by locally replacing the LOT algebra by the dga $\rk$. Given a diagram corresponding to $(S,T,\phi)$, we replace each crossing by a braid-like crossing where the strand with the larger slope is above the other strand, and then allow dots on increasing strands. They satisfy relations similar to (R1)--(R3) for $\rk$; see Definition~\ref{def rnk} for more details.

%(1) two dots on disjoint strands anti-commute; (2) a diagram with more than one dot on a single strand is zero; and (3) a dot can freely slide through a crossing if it is on the strand above.\marginpar{\cg Maybe I'll just refer to the (Hi) below.}
% The differential of a crossing is defined in a similar way as in $\rk$: resolving the crossing and adding a dot.

We now introduce some notation to describe the placement of dots: Given a finite set $Y$, let $\mathcal{O}(Y)$ be the set of ordered subsets of $Y$ and let $\po(Y)=\mathcal{O}(Y)\cup\{\star\}$, where $\star$ is a distinguished element.  (Note that the empty set $\varnothing\in \mathcal{O}(Y)$ and $\star$ is different from $\varnothing$.) If $D$ and $D'$ have the same underlying subset of $Y$, let $\mu(D,D')$ denote the minimal number of transpositions needed to take $D$ to $D'$.
Any inclusion $\phi: Y \ra Z$ induces a natural inclusion $\po(Y) \ra \po(Z)$, which we still denote by $\phi$.
Our universe is $\mfn$. For any $Y \subset \mfn$, we identify $\po(Y)$ with a subset of $\po(\mfn)$.
For $D, D' \in \po(Y)$, define $D \# D'$ as the subset $D \sqcup D'$ with the order on $D$ first and the order on $D'$ next if $D,D' \neq \star$ and $D \cap D'=\es$; otherwise define $D \# D'$ as $\star$.
We call $D \# D'$ the {\em ordered union} of $D$ and $D'$.

For a nondecreasing bijection $\phi: S \ra T$,
let $I(\phi)=\{t \in T ~|~ \phi^{-1}(t)<t\}$ be the subset of $T$ consisting of right endpoints of increasing strands.
Let $\po(\phi):=\po(I(\phi))$. Let
\begin{gather} \label{eq inv}
\inv(\phi)=\{(i,j) ~|~ i,j \in T, i<j, \phi^{-1}(i)>\phi^{-1}(j)\}
\end{gather} 
be the number of inversions in $\phi$.

\begin{defn}[Definition of $\rnk$ as a dga] \label{def rnk} $\mbox{}$

As a vector space $\rnk$ is generated by $(S, T, \phi, D)$, where $S$, $T \in \pnk$, $\phi: S \ra T$ is a nondecreasing bijection, and $D \in \po(\phi)$, and quotiented out by the relations
\begin{gather*}
(S, T, \phi, D)=(-1)^{\mu(D,D')}(S, T, \phi, D'), \quad \mbox{if} ~D=D' ~\mbox{as subsets,} \\
(S, T, \phi, \star)=0.
\end{gather*}
Each generator $(S, T, \phi, D)$ can be drawn as a braid-like diagram of $k$ strands for $\phi$ with a collection of ordered dots on the increasing strands near the right boundary determined by $D \in \po(\phi)$.
%The multiplication is given by
%\begin{equation} \label{eq rnk}
%(S, T, \phi, D) \cdot (U, V, \psi, D')=\left\{
%\begin{array}{cc}
%(S,V,\psi\circ\phi, \psi(D)\#D'), & \mbox{if}~ U=T; \\
%0, & \mbox{otherwise.}
%\end{array}
%\right.
%\end{equation}

The multiplication is given by a concatenation of two diagrams (if the diagrams cannot be concatenated, the product is set to zero), subject to the following rules:
\be
\item[(H1)] disjoint diagrams supercommute;
\item[(H2)] a double crossing satisfies the quadratic relation as in the Hecke algebra; 
\item[(H3)] a diagram with more than one dot on a strand is set to be zero; 
\item[(H4)] a dot can freely slide through a crossing if it is on the strand above; and 
\item[(H5)] when a dot slides through a crossing from below, an extra term with $\hbar$ is needed; see (\ref{rel Rk}) and Figure \ref{fig alg1}.
\ee

The differential is defined on the generators by
$$\textstyle{d(S, T, \phi, D)=\sum_{(i,j) \in \inv(\phi)}\left((S, T, \phi_{(i,j)}, \{i\}\#D)-(S, T, \phi_{i,j}, \{j\}\#D)\right)},$$
where $\phi_{i,j}=\phi$ except at $\phi^{-1}(i), \phi^{-1}(j)$, and $\phi_{i,j}(\phi^{-1}(i))=j, \phi_{i,j}(\phi^{-1}(j))=i$.
The differential is a sum of resolution of crossings indexed by $\inv(\phi)$.
The resolution of a single crossing is defined similarly as in Definition \ref{def Rk}.

The degree of the generator $(S, T, \phi, D)$ is $|D|$, the number of elements of $D$.
In particular, the dga is non-negatively graded.

The unit $\mb=\sum_{S}\mb_S$, where $\mb_S=(S,S,\op{id},\es)$ is an idempotent which is closed.
\end{defn}

For example, a linear basis of $R(3;2)$ is drawn in Figure~\ref{fig alg5}.
%The nontrivial multiplication is a concatenation of the two diagrams whose decoration of dots is given by pushing $D$ via $\psi$ first and taking the ordered union with $D'$. If $\psi(D) \cap D' \neq \es$, then there is more than one dot on a strand so that the resulting multiplication is zero.

%Note that diagrammatic presentations of a generator might not be unique.

\begin{figure}[ht]
\begin{overpic}
[scale=0.3]{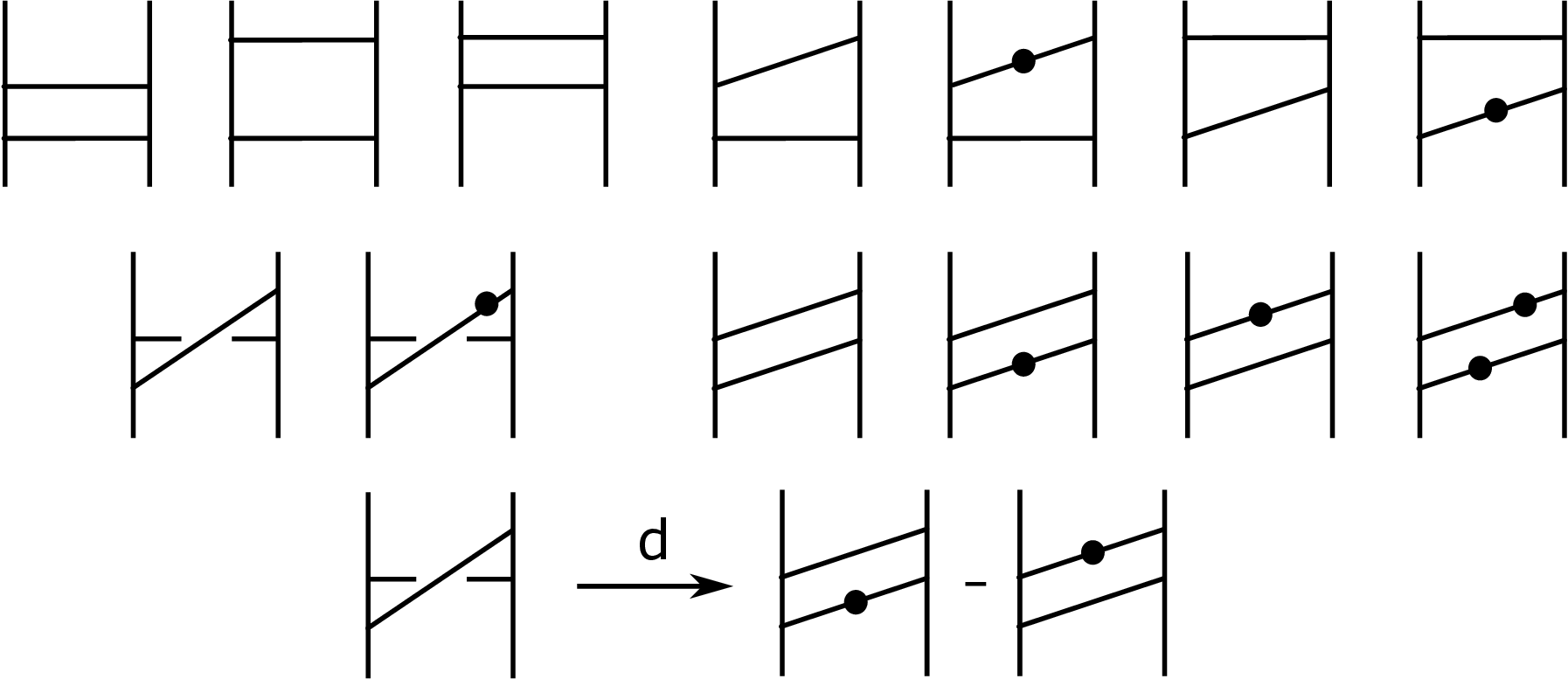}
\end{overpic}
\caption{A linear basis of $R(3;2)$, together with the differential.}
\label{fig alg5}
\end{figure}

\begin{lemma} \label{lem Rnk well}
The dga $\rnk$ is well-defined.
\end{lemma}

\begin{proof}
We prove that the differential $d$ satisfies $d^2=0$ and the Leibniz rule, using the diagrammatic presentation of $\rnk$.

Fixing $S, T \in \pnk$, the subspace $\mb_S \rnk \mb_T$ can be embedded into $\rk$ by normalizing both ends $S$ and $T$ to $\{1,2,\dots,k\}$ and adjusting the strands accordingly. The embedding commutes with the differentials. Hence $d^2=0$ since it holds for $\rk$.

Fixing $S, T, V \in \pnk$, the embedding $$\mb_S \rnk \mb_T \times \mb_T \rnk \mb_V \hookrightarrow \rk \times \rk$$ commutes with the multiplication maps on both sides. Hence $d$ satisfies the Leibniz rule since it does for $\rk$.
\end{proof}

\begin{rmk} $\mbox{}$
\be
\item The vector space $\rnk$ is finite-dimensional with a basis $(S, T, \phi, D)$ by choosing one order $D$ for each underlying subset of $I(\phi)$.

\item Unlike the algebra $\rk$, describing $\rnk$ in terms of generators and relations is not straightforward.
\ee
%For instance, $\rnk$ is not generated by elements whose diagrams contain at most one crossing if $k \ge 3$.
%The generator at the bottom left of Figure~\ref{fig alg5} contains two crossings, but it cannot be written as a product of two generators with fewer crossings.
\end{rmk}

%We compute the cohomology of $\mb_S \rnk \mb_T$ in the following.
Given $S,T\in \pnk$, we write $S=\{s_1, \dots, s_k\}$, $T=\{t_1,\dots,t_k\}$, such that $s_1<\cdots <s_k$ and $t_1<\cdots <t_k$.
We write $S \le T$ if $s_i \le t_i$ for all $i$. Let $m(S,T)$ denote the cardinality $|S \backslash T|=|T \backslash S|$.

\begin{prop} \label{prop H(Rnk)}
The space $\mb_S \rnk \mb_T$ is nonzero only if $S \le T$.
In this case, its cohomology is isomorphic to an exterior algebra $\lm_m$ as graded vector spaces, for $m=m(S,T)$.
\end{prop}

\begin{proof}
The first assertion is clear by definition that generators in $\mb_S \rnk \mb_T$ are given by nondecreasing bijections.

The proof of the second assertion is similar to that of Lemma \ref{lem H(Rk)} and is by induction on $k$. 
Consider $S, T \in \pnk$. 
%Let $S=\{s_1, \dots, s_k\}, T=\{t_1,\dots,t_k\}$, where $s_1<\cdots <s_k, t_1<\cdots <t_k$.
Let $j \in \{1,\dots,k\}$ such that $t_{j-1} < s_k \le t_j$.
Define $S'=S \backslash \{s_k\}$ and $T^p=T \backslash \{t_p\}$.
The complexes $\mb_{S'} \rnk \mb_{T^p}$ are canonically isomorphic to each other, for $j \le p \le k$.
Let $G^p$ denote the subspace of $\mb_S \rnk \mb_T$ spanned by diagrams containing a strand which connects the top ($k$th) endpoint on the left to the $p$th endpoint on the right, for $j \le p \le k$.
There is a finite filtration 
$$0 \subset F^k \subset \cdots \subset F^j = \mb_S \rnk \mb_T,$$
where $F^p=\oplus_{q \ge p}G^q$ is a subcomplex of $\mb_S \rnk \mb_T$.

There are two cases: $t_j > s_k$ and $t_j=s_k$.
In the $t_j > s_k$ case, each $G^p$ is isomorphic to $$\mb_{S'} \rnk \mb_{T^p} \oplus \mb_{S'} \rnk \mb_{T^p}[-1]$$
and the cohomology is isomorphic to the second page $$E_2=E_2^j \oplus E_2^k \cong \lm_{m'}[-1] \oplus \lm_{m'} \cong \lm_m$$ for $m'=m(S', T^p)=m(S,T)-1.$ The proof is similar to that of Lemma \ref{lem H(Rk)} and is left to the reader.

In the $t_j=s_k$ case, each $G^p$ is isomorphic to $$\mb_{S'} \rnk \mb_{T^p} \oplus \mb_{S'} \rnk \mb_{T^p}[-1],$$ 
except for $p=j$. The space $G^j$ is isomorphic to $\mb_{S'} \rnk \mb_{T^j}$ since dots are not allowed on the horizontal strand connecting $s_k$ to $t_j$.
By induction on $k$, the first page is:
$$E_1^p \cong \left\{ \begin{array}{cl}
\lm_m, & p=j, \\
\lm_{m}[-1] \oplus \lm_{m}, & j < p \le k,
\end{array}
\right.$$
where $m=m(S,T)=m(S', T^p)$.
The complex $(E_1^p, d_1^p)$ is exact except at $p=k$, and $E_2^k \cong \lm_m$.
The spectral sequence degenerates at the second page and the proposition follows.
\end{proof}

The differential on $\rnk$ is trivial if $k=1$. Hence $\{i\}R(n,1)\{j\}$ is two-dimensional if $i<j$, one-dimensional if $i=j$, and zero otherwise.
Let $(i)=\mb_{\{i\}}$ denote the idempotent, and $$(j|j+1), (j \bullet j+1) \in \{j\}R(n,1)\{j+1\}$$ 
denote the generators of degree zero and one, respectively.

\begin{prop} \label{prop Rn1}
The dga $R(n,1)$ is generated by $(i), (j|j+1), (j \bullet j+1)$ for $1 \le i \le n, 1\le j \le n-1$, with relations:
\begin{gather*}
(i)(i')=\delta_{i,i'}(i), \\
(i)(j|j+1)= \delta_{i,j}(j|j+1), \quad (j|j+1)(i)= \delta_{j+1,i}(j|j+1), \\
(i)(j\bullet j+1)= \delta_{i,j}(j\bullet j+1), \quad (j\bullet j+1)(i)= \delta_{j+1,i}(j\bullet j+1), \\
(j|j+1)(j+1\bullet j+2)=(j\bullet j+1)(j+1|j+2), \\
 (j\bullet j+1)(j+1\bullet j+2)=0.
\end{gather*}
The degree of $(j\bullet j+1)$ is one and is zero for the other generators.
The differential is trivial.
\end{prop}

\begin{prop} \label{prop rnk non formal}
The dga $\rnk$ is not formal for $2 \le k \le n-2$.
\end{prop}

\begin{proof}
There is a nontrivial Massey product in $H(R(4,2))$; see Figure~\ref{fig alg6}. Hence the dga $R(4,2)$ is not formal.
This Massey product is local and exists in $H(\rnk)$ for $2 \le k \le n-2$.
\end{proof}

\begin{figure}[ht]
\begin{overpic}
[scale=0.3]{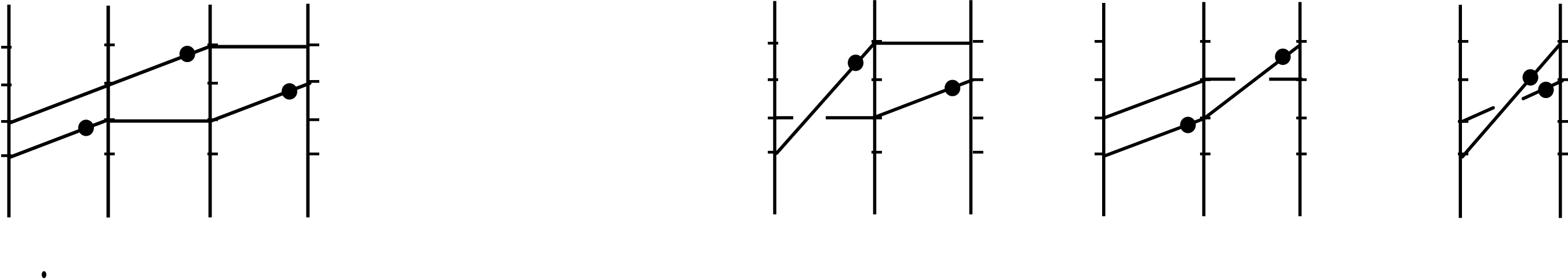}
\put(2,0){$a$}
\put(11,0){$b$}
\put(17,0){$c$}
\put(32,9){$\mu_3(a,b,c)=$}
\put(65,9){$-$}
\put(87,9){$=$}
\put(55,0){$gc$}
\put(75,0){$af$}
\put(94,0){$gc$}
\end{overpic}
\caption{A nontrivial Massey product on $H(R(4,2))$:  $\mu_3(a,b,c)=gc$, where $d(g)=ab, d(f)=-bc$.}
\label{fig alg6}
\end{figure}
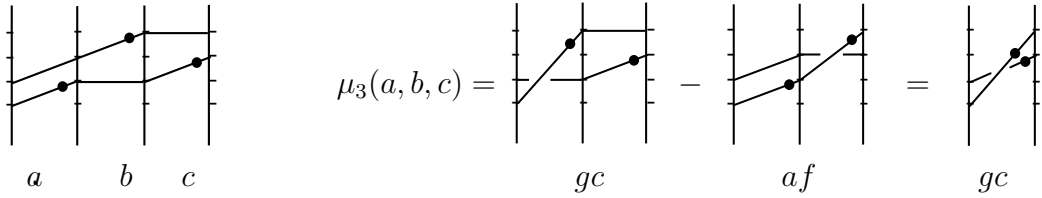

%We further conjecture that our combinatorial model is isomorphic to the $\ai$-algebra in geometry.
%\begin{conj}
%The higher multiplication $\mu_i$ on the $\ai$-algebra $\oplus_{S,T \in \pnk}\Hom(L_T, L_S)$ is trivial for $i \geq 3$.
%Moreover, the resulting dga is isomorphic to $\rnk$.
%\end{conj}

\subsection{A categorification of tensor product representation} \label{ssec An Kh}

This subsection is the graded version of the previous subsection, as $\rkkh$ is the associated $q$-graded dga of $\rk$. The dga $\rnk$ is a generalization of the strands algebra $\cal{A}(n, k)$ which is locally enriched by the dga $\rk$. 
There is a filtration on $\rnk$ induced by that on $\rk$. The associated $q$-graded dga is denoted by $\rnkkh$. It is a generalization of the strands algebra $\cal{A}(n, k)$ which is locally enriched by the dga $\rkkh$. 

%As $\rkkh$ is related to the nilHecke algebra $\nhk$ and categorified quantum $\mathfrak{sl}_2$, $\rnkkh$ is expected to be related to the categorification of the tensor product representation $V^{\ot n}$ of quantum $\mathfrak{sl}_2$.  
%In particular, the idempotents of $\rnkkh$ descend to the tensor basis of $V^{\ot n}$. We construct two bimodules $\ech, \fch$ over the $\rnkkh$ and show that they descend to the actions of $E, F$ on $V^{\ot n}$; see Theorems \ref{thm tensor f} and \ref{thm tensor e}. 

\subsubsection{Definition of $\rnkkh$}

The algebra $\rnkkh$ has a definition similar to that of $\rnk$ from Definition \ref{def rnk}, with the following changes: for each diagram a braid-like crossing is replaced by a crossing; a double crossing is set equal to zero; and a dot can freely slide through a crossing.  
The multiplication now can be explicitly given by: 
\begin{align*} 
& (S, T, \phi, D) \cdot (U, V, \psi, D') \\
=& \left\{
\begin{array}{cl}
(S,V,\psi\circ\phi, \psi(D)\#D'), & \mbox{if  $U=T$ and $|\inv(\psi\circ\phi)|=|\inv(\psi)|+|\inv(\phi)|$}; \\
0, & \mbox{otherwise.}
\end{array}
\right.
\end{align*}
Recall that $\inv(\phi)$ from  \eqref{eq inv} is the number of inversions of $\phi$. (The above formula automatically implies that a double crossing is zero in $\rnkkh$.)
%This simply means that a double crossing is zero in $\rnkkh$. 
The cohomological and $q$-gradings are given by
\begin{align} \label{eq An qgrading} 
\deg(S, T, \phi, D)&=|D|,\\
\deg_q(S, T, \phi, D)&=-2(|D|+|\inv(\phi)|)+(\|T\|-\|S\|),\nonumber
\end{align}
where $\|S\|=\textstyle{\sum_{1 \le i \le k}s_i}$ for $S=\{s_1, \dots, s_k\} \in \pnk$. Observe that the first part of the $q$-grading is induced from that of $\rkkh$.

The algebra $\rnkkh$ is a finite-dimensional quiver algebra with relations, where the quiver has no loops. There is a complete set of primitive idempotents 
$$\{\mb_S=(S,S,\op{id},\es) ~|~ S \in \pnk\},$$ 
where $\pnk$ is the set of $k$-element subsets of $\mfn=\{1,2,\dots,n\}$. 
Define 
$$(S, T)=\mb_S \cdot \rnkkh \cdot \mb_T,$$
as the subspace of $\rnkkh$. In particular, $(S, S)$ is one-dimensional, generated by $\mb_S$.  
For each $S \in \pnk$, there is a $q$-graded right projective dg module $P(S)=\mb_S \cdot \rnkkh$. The corresponding simple module $L(S)$ is one-dimensional. 

Let $M$ be a finitely-generated $q$-graded right dg $\rnkkh$ module so that it is finite-dimen\-sional. Let $M=\oplus_{i,j}M^i_j$, where $i$ denotes the cohomological grading and $j$ denotes the $q$-grading.  
There are two grading shifts: the cohomological shift $[a]$ and the $q$-grading shift $\{b\}$. Our convention is that $M[a]\{b\}^i_j=M^{i+a}_{j-b}$.  

Let $\cnk:=\mathcal{D}^b(\rnkkh)$ be the bounded derived category of finitely generated $q$-graded right dg $\rnkkh$ modules.  Recall a collection of objects {\em generates} the category $\cnk$ if the smallest full triangulated subcategory containing it is equivalent to $\cnk$. 
Then $\cnk$ is generated by the collection of all simple modules and their grading shifts. Since each simple module admits a finite projective resolution, $\cnk$ is also generated by $\{P(S)\{b\} ~|~ S \in \pnk, b\in \Z\}$.  

\begin{prop} \label{prop K0 rnk}
The Grothendieck group $K_0(\cnk)$ is isomorphic to the free $\zq$-module generated by the $[P(S)]$ for $S \in \pnk$ such that $[P(S)\{b\}]=q^b[P(S)]$.
\end{prop}

\subsubsection{Tensor product representation.}

Recall that the integral version of quantum $\mathfrak{sl}_2$ is an associative $\Z[q,q^{-1}]$-algebra with generators $E, F, K, K^{-1}$, and relations:
\begin{gather*}
 KK^{-1}=K^{-1}K=1, \\
 KE=q^2EK, \quad KF=q^{-2}FK, \\
 EF-FE=\frac{K-K^{-1}}{q-q^{-1}}.
\end{gather*} 
The comultiplication is given by
\begin{gather*}
\Delta(K^{\pm1})=K^{\pm1} \otimes K^{\pm1}, \\
 \Delta(E)=E\otimes 1+K\otimes E, \\
 \Delta(F)=F\otimes K^{-1}+1\otimes F.
\end{gather*}
Its fundamental representation $V$ has a basis $\{v_{-1}, v_1\}$ such that
\begin{align*}
K^{\pm1}v_{-1}=q^{\mp 1}v_{-1}, \quad & K^{\pm1}v_{1}=q^{^{\pm1}1}v_{1}, \\
Ev_{-1}=v_1, Ev_1=0, \quad  & Fv_{1}=v_{-1}, Fv_{-1}=0.
\end{align*}
The tensor product $V^{\otimes n}$ has a tensor basis $\{v(S) ~|~ S \in \pnk, 0 \le k \le n\}$, where $v(S)=v_{1(S)} \otimes \cdots \otimes v_{n(S)} \in V^{\otimes n}$, and $i(S)=1$ if $i \in S$; otherwise, $i(S)=-1$.

\s
\n{\em Proof of Theorem \ref{thm K0 cnk}.}
Proposition \ref{prop K0 rnk} implies that 
$$\oplus_{0 \le k \le n}K_0(\cnk) \ra V^{\otimes n}, \quad [P(S)] \mapsto v(S)$$ 
is an isomorphism of free $\Z[q, q^{-1}]$-modules. 
\qed
\s

We compute the action of $E,F$ on $V^{\otimes n}$ as follows.
For $S=\{s_1, \dots, s_k ~|~ s_1> \cdots > s_k\} \in \pnk$, let us write $f_i(S)=S \backslash \{s_i\} \in \pnkm$. 
For $S^c=\{s_1^c, \dots, s_{n-k}^c ~|~ s_1^c < \cdots < s_{n-k}^c\}$, let us write $e_i(S)=S \cup \{s_i^c\} \in \pnkp$.

\begin{prop} \label{prop ef}
The actions of $E,F$ on the tensor basis of $V^{\otimes n}$ are given by
\begin{align*}
F(v(S))=\textstyle\sum_{1 \le i \le k}q^{m_i(S)}v(f_i(S)), \\
E(v(S))=\textstyle\sum_{1 \le i \le n-k}q^{l_i(S)}v(e_i(S)),
\end{align*}
where $m_i(S)=n-s_i-2i+2, l_i(S)=s_i^c-2i+1$.
\end{prop}

\begin{proof}
The action of $E,F$ is determined by the comultiplication. The formula is straightforward except for the exponents of $q$. 
For the action of $F$, one component of the iterated comultiplication is $1\otimes \cdots \otimes 1 \otimes F \otimes K^{-1} \otimes \cdots \otimes K^{-1}$.  Hence $q^{m_i(S)}$ is the eigenvalue of  $K^{-1} \otimes \cdots \otimes K^{-1}$ on the last $n-s_i$ factors of the tensor product $v(S)$. Among them, there are $i-1$ copies of $v_1$ and $n-s_i-i+1$ copies of $v_{-1}$. 
We have $m_i(S)=(n-s_i-i+1) \cdot 1+ (i-1) \cdot (-1)=n-s_i-2i+2$. 
The computation of $l_i(S)$ is similar and is left to the reader.  
\end{proof}

\subsubsection{Definition of bimodules $\ech$ and $\fch$}

We construct a $(\rnkkh, \rnkp)$ dg bimodule $\ech_{k, k+1}$ for $0 \le k \le n-1$, and a $(\rnkkh, \rnkm)$ dg bimodule $\fch_{k, k-1}$ for $1 \le k \le n$. For simplicity we will refer to them as a $(k, k+1)$-bimodule $\ech$ and a $(k,k-1)$-bimodule $\fch$ when $k$ is understood.  
As a  cochain complex
$$\ech=\oplus_{\scriptscriptstyle S \in \pnk,  T \in \pnkp}(1\cdot S, 0 \cdot T)\subset \rnke,$$
where $1 \cdot S:=\{1\} \cup \{s+1 ~|~ s \in S\}, 0 \cdot T:=\{t+1 ~|~ t \in T\} \in \pnke$. %It is a subcomplex of $\rnke$. 
The bimodule structure is induced by the multiplication of $\rnke$. For each summand, the nontrivial multiplication is given by:
$$(S', S) \otimes (1\cdot S, 0 \cdot T) \cong (1 \cdot S', 1 \cdot S) \otimes (1\cdot S, 0 \cdot T) \ra (1\cdot S', 0 \cdot T),$$
$$(1\cdot S, 0 \cdot T) \otimes (T, T') \cong (1\cdot S, 0 \cdot T) \otimes (0 \cdot T, 0 \cdot T') \ra (1\cdot S, 0 \cdot T').$$
Dually, define 
$$\fch=\oplus_{\scriptscriptstyle S \in \pnk,  T \in \pnkm}(S\cdot 0, T \cdot 1)\subset \rnkf,$$
where $S \cdot 0:=S, T \cdot 1:=T \cup \{n+1\} \in \pnkf$. %It is a subcomplex of $\rnkf$. 
The bimodule structure is induced by the multiplication of $\rnkf$.  The bimodules $\ech$ and $\fch$ are schematically given by Figure~\ref{fig algn1}.
\begin{figure}[ht]
\begin{overpic}
[scale=0.4]{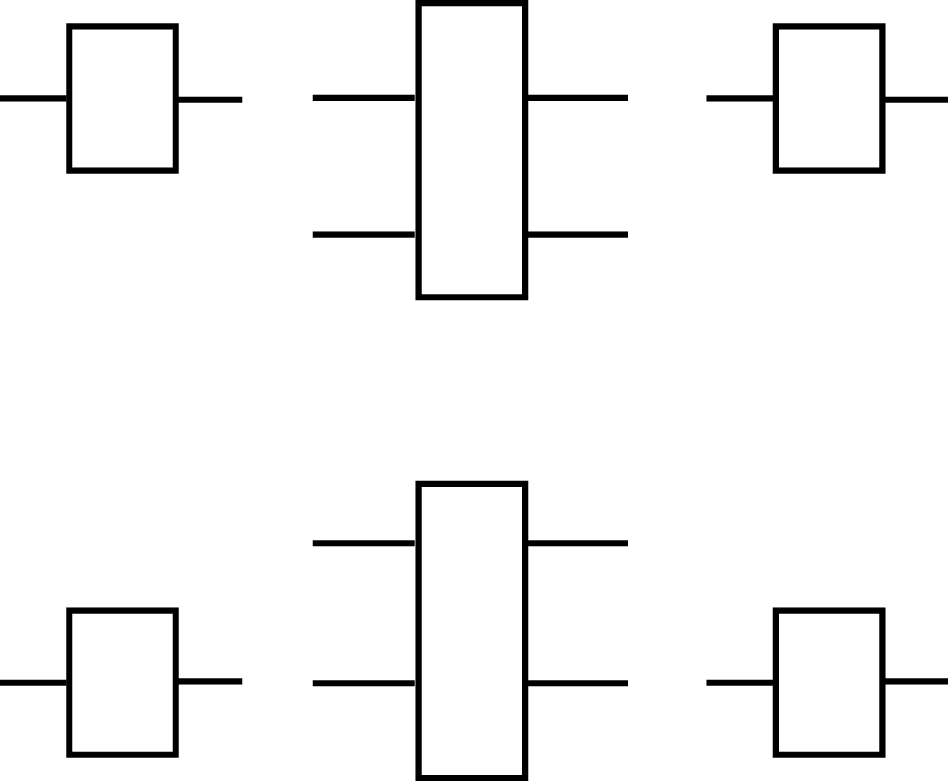}
\put(28,55){$\underline{1}$}
\put(68,55){$\underline{0}$}
\put(28,70){$S$}
\put(68,70){$T$}
\put(101,70){$T'$}
\put(-6,70){$S'$}
\put(-20,70){$\ech:$}
\put(28,23){$\underline{0}$}
\put(68,23){$\underline{1}$}
\put(28,8){$S$}
\put(68,8){$T$}
\put(101,8){$T'$}
\put(-6,8){$S'$}
\put(-20,8){$\fch:$}
\end{overpic}
\caption{The bimodules $\ech$ and $\fch$. There is a strand which begins/ends at $\underline{1}$, but $\underline{0}$ is just a placeholder and there is no strand which begins/ends at $\underline{0}$.}
\label{fig algn1}
\end{figure}
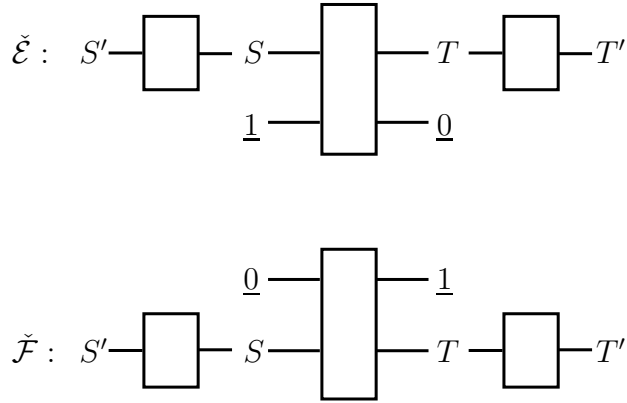

\begin{rmk}
The bimodules $\ech$ and $\fch$ can be viewed as the induction and restriction bimodules, respectively. Their tensor products $\ech\fch$ and $\fch\ech$ however are not easy to compute since the extra strand for $\ech$ is at the bottom, while the one for $\fch$ is at the top.  
\end{rmk}

\subsubsection{Computation of the functors.}

The derived tensor product with $\ech, \fch$ give functors $\cnk \ra \cnkp, \cnk \ra \cnkm$ that we still denote by $\ech, \fch$. 

We first compute $\fch(P(S))\cong P(S) \otimes_k \fch$ for $S \in \pnk$, where $\otimes_k$ is tensor product over $\rnkkh$. 
Recall that $f_i(S)=S \backslash \{s_i\} \in \pnkm$ for $S=\{s_1, \dots, s_k ~|~ s_1> \cdots > s_k\} \in \pnk$.
For each pair $(i,j)$ such that $1 \le i < j \le k$, there is a distinguished generator 
$$r_{i,j}=(f_i(S), f_j(S), \phi_{i,j}, \es) \in \rnkm,$$ 
where $\phi_{i,j}: f_i(S) \ra f_j(S)$ maps $s_j$ to $s_i$ and fixes the other elements.  Its diagrammatic presentation has a unique non-horizontal strand which intersects $j-i-1$ horizontal strands. Let $r_{i,j}^+$ be the generator of $\rnkm$ obtained by adding a dot on the unique non-horizontal strand of $r_{i,j}$.  The gradings are 
\begin{align*}
 \deg(r_{i,j})=0, \quad &\deg_{q}(r_{i,j})=-2(j-i-1)+s_i-s_j, \\
 \deg(r_{i,j}^+)=1, \quad &\deg_{q}(r_{i,j}^+)=-2(j-i)+s_i-s_j,
\end{align*}
in view of Equation \eqref{eq An qgrading}.

\begin{thm} \label{thm tensor f}
There is an isomorphism of right dg $\rnkm$ modules:
$$P(S) \otimes_k \fch \cong \left(\oplus_{1 \le i \le k} (P(f_i(S))\{m_i(S)+1\} \oplus P(f_i(S))[-1]\{m_i(S)-1\}), d\right),$$
where the differential is $d=\sum_{1 \le i < j \le k}(d_{i,j}+d_{i,j}[-1]+d_{i,j}')$ and each summand is given by left multiplication between right projective modules: 
\begin{align*}
d_{i,j}: & P(f_j(S))\{m_j(S)+1\} \xra{r_{i,j}^+ \cdot} P(f_i(S))\{m_i(S)+1\}, \\  
d_{i,j}[-1]: & P(f_j(S))[-1]\{m_j(S)-1\} \xra{-r_{i,j}^+ \cdot} P(f_i(S))[-1]\{m_i(S)-1\}, \\  
d_{i,j}': & P(f_j(S))\{m_j(S)+1\} \xra{-r_{i,j} \cdot} P(f_i(S))[-1]\{m_i(S)-1\}.  
\end{align*}
\end{thm}

\begin{proof}
The complex $P(S) \otimes_k \fch$ is isomorphic to $\oplus_{\scriptscriptstyle T \in \pnkm}(S\cdot 0, T \cdot 1)$. For $1 \le i \le k$, there are $k$ distinguished generators $r_i=(S \cdot 0, f_i(S) \cdot 1, \phi_i, \es) \in \fch$, where $\phi_i: S \cdot 0 \ra f_i(S) \cdot 1$ maps $s_i$ to $n+1$ and fixes the other elements. Its diagrammatic presentation also has a unique non-horizontal strand which intersects the other $i-1$ horizontal strands; see Figure~\ref{fig algn2}. Let $r_{i}^+$ be the generator of $(S \cdot 0, f_i(S) \cdot 1)$ by adding a dot on the unique non-horizontal strand in $r_i$. 
The gradings are 
\begin{align*}
 \deg(r_i)=0, \quad &\deg_{q}(r_i)=-2(i-1)+n+1-s_i=m_i(S)+1, \\
 \deg(r_i^+)=1, \quad &\deg_{q}(r_i^+)=-2i+n+1-s_i=m_i(S)-1.
\end{align*}
The multiplication $r_i \cdot \mb_{f_i(S)}=r_i$. 
Each $r_i$ generates a right $\rnkm$ module $r_i \cdot \rnkm$ which is isomorphic to the projective module $P(f_i(S))\{m_i(S)+1\}$. Similarly, $r_i^+$ generates a right $\rnkm$ module $r_i^+ \cdot \rnkm$ which is isomorphic to the projective module $P(f_i(S))[-1]\{m_i(S)-1\}$. Moreover, 
$$P(S) \otimes_k \fch \cong \oplus_{1 \le i \le k} \left(P(f_i(S))\{m_i(S)+1\} \oplus P(f_i(S))[-1]\{m_i(S)-1\}\right),$$ 
as right $\rnkm$ modules when ignoring the differential. 

It remains to compute $d(r_i), d(r_i^+)$. 
By definition, 
$$d(r_j)=\textstyle\sum_{i<j}(r_i\cdot r_{i,j}^+ -r_i^+ \cdot r_{i,j}), \quad d(r_j)=\textstyle\sum_{i<j} -r_i^+ \cdot r_{i,j}^+.$$
The summand $r_i\cdot r_{i,j}^+$ of $d(r_j)$ gives the summand 
$$d_{i,j}:  P(f_j(S))\{m_j(S)+1\} \xra{r_{i,j}^+ \cdot} P(f_i(S))\{m_i(S)+1\}.$$
The other summands of $d$ can be computed similarly and the details are left to the reader. 
\end{proof}

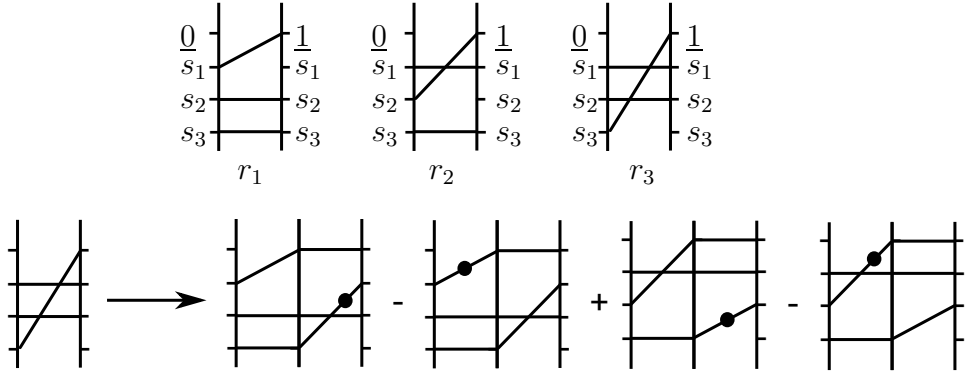
\begin{figure}[ht]
\begin{overpic}
[scale=0.4]{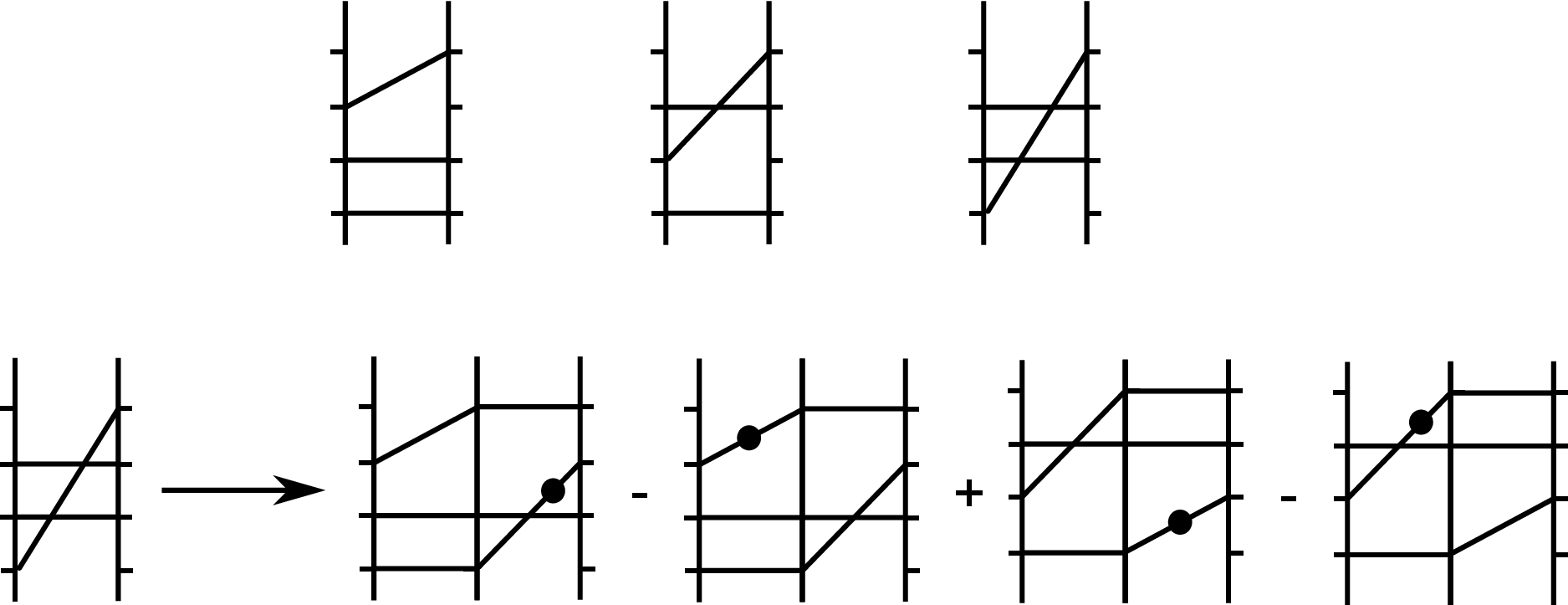}
\put(18,34){$\underline{0}$}
\put(18,31){$s_1$}
\put(18,27.5){$s_2$}
\put(18,24){$s_3$}
\put(30,31){$s_1$}
\put(30,27.5){$s_2$}
\put(30,24){$s_3$}
\put(30,34){$\underline{1}$}
\put(38,34){$\underline{0}$}
\put(38,31){$s_1$}
\put(38,27.5){$s_2$}
\put(38,24){$s_3$}
\put(51,31){$s_1$}
\put(51,27.5){$s_2$}
\put(51,24){$s_3$}
\put(51,34){$\underline{1}$}
\put(59,34){$\underline{0}$}
\put(59,31){$s_1$}
\put(59,27.5){$s_2$}
\put(59,24){$s_3$}
\put(71,31){$s_1$}
\put(71,27.5){$s_2$}
\put(71,24){$s_3$}
\put(71,34){$\underline{1}$}
\put(65,20){$r_3$}
\put(44,20){$r_2$}
\put(24,20){$r_1$}
\end{overpic}
\caption{The generators $r_1, r_2, r_3$ at the top and the differential $d(r_3)=r_1 \cdot r_{1,3}^+-r_1^+ \cdot r_{1,3} + r_2 \cdot r_{2,3}^+-r_2^+ \cdot r_{2,3}$ at the bottom.}
\label{fig algn2}
\end{figure}

\begin{rmk}
The module $\fch(P(S))$ has a submodule 
$$(\oplus_{1 \le i \le k} P(f_i(S))[-1]\{m_i(S)-1\}, \textstyle\sum_{1 \le i < j \le k}d_{i,j}[-1]),$$ 
and the quotient of $\fch(P(S))$ by this submodule is 
$$(\oplus_{1 \le i \le k} P(f_i(S))\{m_i(S)+1\}, \textstyle\sum_{1 \le i < j \le k}d_{i,j}).$$ 
Up to grading shifts, both modules can be viewed as $\cal{F}(P(S))$ for some functor $\cal{F}$ which descends to the action of $F \in \mathfrak{sl}_2$ on $V^{\ot n}$.  
\end{rmk}

\begin{example} \label{ex f}
For $S=\{1,2,4\} \in \cal{P}_3(\mathbf{4})$, we also write $S=(1101) \in \{0,1\}^4$, where the $s$th factor is $1$ for $s \in S$. 
We have $f_1(S)=(1100), f_2(S)=(1001), f_3(S)=(0101)$ in $\cal{P}_2(\mathbf{4})$. The right dg $R(4;2)^{\op{nil}}$ module $P(1101) \otimes \fch$ is isomorphic to:
$$\xymatrix{
P(1100)  & & P(1100)\{-2\} \\
P(1001)\{-2\} \ar[u]^{r_{1,2}^+} \ar[urr]^{r_{1,2}} & & P(1001)\{-4\} \ar[u]^{-r_{1,2}^+}  \\
P(0101)\{-4\} \ar[u]^{r_{2,3}^+} \ar[urr]_{-r_{2,3}} \ar[uurr]^{-r_{1,3}} \ar@/^3pc/[uu]^{r_{1,3}^+}& & P(0101)\{-6\}, \ar[u]^{-r_{2,3}^+} \ar@/_3pc/[uu]_{-r_{1,3}^+}
}$$ 
where the first and second columns are in cohomological degrees $0$ and $1$, respectively; and the differential is given by left multiplication with some generators. 
\end{example}

We now prepare to compute $\ech(P(S))\cong P(S) \otimes_k \ech$ for $S \in \pnk$.  
First consider the case $S=0 \cdot \hat{S}$ for $\hat{S} \in \cal{P}_{k}(\mathbf{n-1})$. 
There is an inclusion of dgas 
$$i_0: R(n-1, k+1)^{\op{nil}} \ra R(n,k+1)^{\op{nil}},$$ mapping $(\hat{T}, \hat{T}')$ to $(0 \cdot \hat{T}, 0 \cdot \hat{T}')$, for $\hat{T}, \hat{T}' \in \cal{P}_{k+1}(\mathbf{n-1})$. 
It lifts any diagram in $R(n-1, k+1)^{\op{nil}}$ one unit up. Note that this inclusion is not unital.
It induces a left dg  $R(n-1,k+1)^{\op{nil}}$ module structure on $R(n,k+1)^{\op{nil}}$ which is isomorphic to a direct sum of the free module $R(n-1,k+1)^{\op{nil}}$ and a zero module.  
Tensoring with $R(n,k+1)^{\op{nil}}$ gives a functor $$-\ot_{n-1,k+1} R(n,k+1)^{\op{nil}}: \cal{C}_{n-1,k+1} \ra \cal{C}_{n,k+1},$$ still denoted by $i_0$. 
The functor $i_0$ is an embedding. We have $i_0(P(\hat{T})) = P(0 \cdot \hat{T})$, and $i_0(M) \cong M$ as right $R(n-1,k+1)^{\op{nil}}$ modules. 

\begin{lemma} \label{lem e action0}
If $S=0 \cdot \hat{S}$ for $\hat{S} \in \cal{P}_{k}(\mathbf{n-1})$, then there exists a short exact sequence of right dg $R(n,k+1)^{\op{nil}}$ modules:
$$0 \ra i_0(\ech(P(\hat{S})))[-1]\{-1\}  \xra{g_1} (P(1 \cdot \hat{S})\{1\} \oplus P(1 \cdot \hat{S})[-1]\{-1\}) \xra{g_0} \ech(P(0\cdot \hat{S})) \ra 0.$$ 
\end{lemma}

\begin{proof}
The right $R(n,k+1)^{\op{nil}}$-module $\ech(P(0\cdot \hat{S})) \cong \oplus_{\scriptscriptstyle T \in \pnkp}(1\cdot 0\cdot \hat{S}, 0 \cdot T)$. For $T=1 \cdot \hat{S}$, there is a distinguished element $r_0=(1\cdot 0\cdot \hat{S}, 0 \cdot 1 \cdot \hat{S}, \phi_0, \es)$, where $\phi_0$ maps $1$ to $2$ and fixes the other elements. Its diagrammatic presentation has a unique non-horizontal strand connecting $1$ to $2$ and does not have any crossing; see Figure~\ref{fig algn3}. 
Let $r_{0}^+$ be the generator by adding a dot on the unique non-horizontal strand in $r_{0}$.  
The gradings are $\deg(r_{0})=0, \deg_{q}(r_{0})=1,\deg(r_{0}^+)=1, \deg_{q}(r_{0}^+)=-1.$
The right module $\ech(P(0\cdot \hat{S}))$ is generated by $r_0$ and $r_0^+$. Hence there is a surjection 
$$
\begin{array}{cccc}
g_0: & (P(1 \cdot \hat{S})\{1\} \oplus P(1 \cdot \hat{S})[-1]\{-1\}) & \ra & \ech(P(0\cdot \hat{S})), \\
& (\mb_{1 \cdot \hat{S}}, 0) & \mapsto & r_0, \\
& (0, \mb_{1 \cdot \hat{S}}) & \mapsto & r_0^+. \\
\end{array}
$$

As cochain complexes $i_0(\ech(P(\hat{S}))) \cong \ech(P(\hat{S}))$ and $\ech(P(\hat{S})) \cong \oplus_{\scriptscriptstyle \hat{T} \in \cal{P}_{k+1}(\mathbf{n-1})}(1 \cdot \hat{S}, 0 \cdot \hat{T})$. Define 
$$
\begin{array}{cccc}
g_1: & i_0(\ech(P(\hat{S})))[-1]\{-1\} & \ra & (P(1 \cdot \hat{S})\{1\} \oplus P(1 \cdot \hat{S})[-1]\{-1\}), \\
& r & \mapsto & (r^l, -r),
\end{array}
$$
for $r \in (1 \cdot \hat{S}, 0 \cdot \hat{T})$, where $r^l$ is obtained by adding a dot on the lowest strand of $r$ viewed from the left. 
The map $g_1$ is an inclusion of right dg $R(n,k+1)^{\op{nil}}$ modules.
We have
$$g_0(g_1(r))=r_0 \cdot i_0(r^l) + r_0^+ \cdot i_0(-r)=0,$$ 
since $r_0 \cdot i_0(r^l) =  r_0^+ \cdot i_0(r) \in (1\cdot 0\cdot \hat{S}, 0 \cdot 0 \cdot \hat{T})$. 
It is not hard to see that $\Ker(g_0)=\im(g_1)$. 
\end{proof}

\begin{figure}[ht]
\begin{overpic}
[scale=0.4]{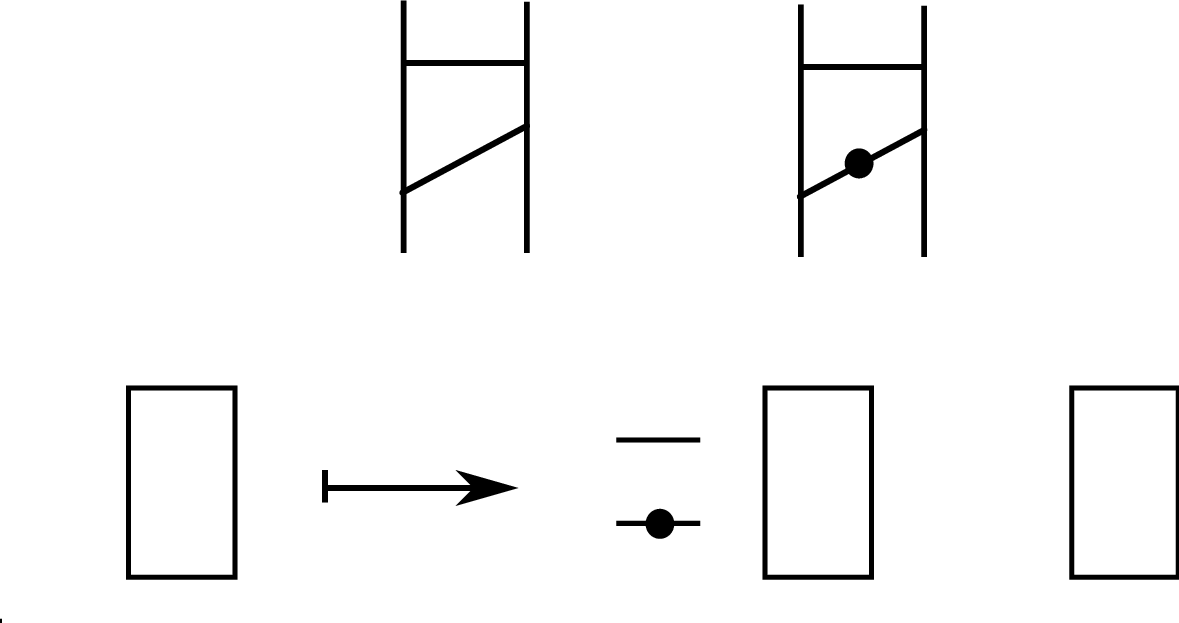}
\put(30,34){$\underline{1}$}
\put(30,40){$0$}
\put(30,46.5){$\hat{S}$}
\put(46,34){$\underline{0}$}
\put(46,40){$1$}
\put(46,46.5){$\hat{S}$}
\put(64,34){$\underline{1}$}
\put(64,40){$0$}
\put(63.5,46.5){$\hat{S}$}
\put(80,34){$\underline{0}$}
\put(80,40){$1$}
\put(80,46.5){$\hat{S}$}
\put(60.5,7){$1$}
\put(60.5,14){$\hat{S}$}
\put(75,7){$0$}
\put(75,14){$\hat{T}$}
\put(80,7){$,$}
\put(81,10){$-$}
\put(86,7){$1$}
\put(86,14){$\hat{S}$}
\put(101,7){$0$}
\put(101,14){$\hat{T}$}
\put(106,10){$)$}
\put(94,10){$r$}
\put(68,10){$r$}
\put(48,7){$1$}
\put(48,14){$\hat{S}$}
\put(45,10){$($}
\put(14,10){$r$}
\put(7,7){$1$}
\put(6.5,14){$\hat{S}$}
\put(21,7){$0$}
\put(21,14){$\hat{T}$}
\put(-5,10){$g_1:$}
\end{overpic}
\caption{The generators $r_0, r_0^+$ at the top and the map $g_1$ at the bottom, where the horizontal strand cannot carry a dot but the multiplication does exist and represents $r^l$.}
\label{fig algn3}
\end{figure}

Next consider the case $S=1 \cdot \hat{S}$ for $\hat{S} \in \cal{P}_{k-1}(\mathbf{n-1})$. 
There is a nonunital inclusion of dgas 
\begin{gather*}
i_1: R(n-1, k)^{\op{nil}} \ra R(n,k+1)^{\op{nil}},\\ 
(\hat{T}, \hat{T}') \mapsto (1 \cdot \hat{T}, 1 \cdot \hat{T}'),
\end{gather*}
where $\hat{T}, \hat{T}' \in \cal{P}_{k}(\mathbf{n-1})$. 
The map $i_1$ adds a horizontal strand at the bottom to any diagram in $R(n-1, k)^{\op{nil}}$. 
It induces a left dg $R(n-1,k)^{\op{nil}}$ module structure on $R(n,k+1)^{\op{nil}}$.

\begin{lemma} \label{leopard}
As a left dg $R(n-1,k)^{\op{nil}}$ module, $R(n,k+1)^{\op{nil}}$ is isomorphic to a complex of projective modules.
\end{lemma}

\begin{proof}
It suffices to show that the left projective $R(n,k+1)^{\op{nil}}$ module $\tilde{P}(T)=R(n,k+1)^{\op{nil}} \cdot \mb_{T}$ is a complex of left projective $R(n-1,k)^{\op{nil}}$ modules for $T \in \cal{P}_{k+1}(\mathbf{n})$. 
If $1 \in T$, then $\tilde{P}(T) \cong \tilde{P}(T\backslash \{1\})$ as left $R(n-1,k)^{\op{nil}}$ modules.
If $1 \notin T$ and $T=\{t_1, \dots, t_{k+1} ~|~ t_1 < \cdots t_{k+1}\}$, then let $f_i(T)=T \backslash \{t_i\}$. There are $k+1$ generators $u_i=(1\cdot f_i(T), T, \phi_i, \es)$ for $1 \le i \le k+1$, where $\phi_i$ maps $1$ to $t_i$ and fixes the other elements. 
Its diagrammatic presentation has a unique non-horizontal strand which intersects with other $i-1$ horizontal strands. Let $u_i^+$ be the generator obtained by adding a dot on the unique non-horizontal strand in $u_i$. 
Then these $2(k+1)$ elements $u_i, u_i^+$ generate the left projective $R(n-1,k)^{\op{nil}}$ modules 
$$\tilde{P}(f_i(T))\{2-2i+t_i-1\},\quad \tilde{P}(f_i(T))[-1]\{-2i+t_i-1\}.$$
Moreover, $\tilde{P}(T)$ is isomorphic to a complex of these projective $R(n-1,k)^{\op{nil}}$ modules.
\end{proof}

By Lemma~\ref{leopard}, there is a functor $$-\ot_{n-1,k} R(n,k+1)^{\op{nil}}: \cal{C}_{n-1,k} \ra \cal{C}_{n,k+1},$$ 
which we still denote by $i_1$. 
We have $i_1(P(\hat{T})) = P(1 \cdot \hat{T})$. 

\begin{lemma} \label{lem e action1}
If $S=1 \cdot \hat{S}$ for $\hat{S} \in \cal{P}_{k-1}(\mathbf{n-1})$, then there exists a short exact sequence of right dg $R(n,k+1)^{\op{nil}}$ modules:
$$0 \ra  i_1(\ech(P(\hat{S}))) \xra{q_0} \ech(P(1\cdot \hat{S}))\ra Q \ra 0,$$
where $Q$ is contractible.  
\end{lemma}

\begin{proof}
We have $\ech(P(\hat{S})) \cong \oplus_{\scriptscriptstyle \hat{T} \in \cal{P}_{k}(\mathbf{n-1})}(1 \cdot \hat{S}, 0 \cdot \hat{T})$ and $\ech(P(1 \cdot \hat{S})) \cong \oplus_{\scriptscriptstyle T \in \cal{P}_{k+1}(\mathbf{n})}(1 \cdot 1 \cdot \hat{S}, 0 \cdot T)$. 
For any $u=(1 \cdot \hat{S}, 0 \cdot \hat{T}, \phi, D) \in \ech(P(\hat{S}))$, there are three associated generators in $(1 \cdot 1 \cdot \hat{S}, 0 \cdot 1 \cdot \hat{T})$:
$$\begin{array}{ll}
u_0=(1 \cdot 1 \cdot \hat{S}, 0 \cdot 1 \cdot \hat{T}, \phi_0, D+1), & \phi_0: 1 \mapsto 2, s+1 \mapsto \phi(s)+1, \\
u_1=(1 \cdot 1 \cdot \hat{S}, 0 \cdot 1 \cdot \hat{T}, \phi_1, D+1), & \phi_1: 2 \mapsto 2, 1 \mapsto \phi(1), s+1 \mapsto \phi(s)+1, \\
u_2=(1 \cdot 1 \cdot \hat{S}, 0 \cdot 1 \cdot \hat{T}, \phi_2, \{2\}\#(D+1)), & \phi_2=\phi_0,
\end{array}$$
where $D+1=\{d+1 ~|~ d \in D\}$; see Figure~\ref{fig algn4}. 
The collection $$\{u_0 \cdot r ~|~ u \in \ech(P(\hat{S})), r \in (0\cdot 1 \cdot \hat{T}, 0 \cdot T)\}$$ spans a right $R(n,k+1)^{\op{nil}}$ submodules $Q_0$ of $\ech(P(1\cdot \hat{S}))$. 
Define a map
$$\begin{array}{cccc}
q_0: & \ech(P(\hat{S})) \ot_{n-1,k} R(n,k+1)^{\op{nil}} & \ra & \ech(P(1\cdot \hat{S})) \\
& u \ot \hat{r} & \mapsto & u_0 \cdot i_0(\hat{r}),
\end{array}$$
where $\hat{r} \in (1 \cdot \hat{T}, T)$ and $i_0(\hat{r}) \in (0\cdot 1 \cdot \hat{T}, 0 \cdot T)$.
It is an isomorphism of right $R(n,k+1)^{\op{nil}}$ modules onto $Q_0$.

Let $Q$ denote the quotient of $q_0$. 
Take a linear basis $\{u^{\alpha}\}$ of $\ech(P(\hat{S}))$.
The associated $\{u_1^{\alpha}, u_2^{\alpha}\}$ forms a linear basis of $Q$. Moreover, 
$$d(u_1)=\tilde{d}(u)_1 + u_2, d(u_2)=-\tilde{d}(u)_2,$$ 
where $\tilde{d}$ is the differential of $\ech(P(\hat{S}))$. 
There is an isomorphism of right $R(n,k+1)^{\op{nil}}$ modules: $$Q \cong  \ech(P(\hat{S})) \ot_{\F} (\F \xra{\id} \F)$$ up to grading shifts. Hence $Q$ is contractible and $q_0: i_1(\ech(P(\hat{S}))) \ra \ech(P(1\cdot \hat{S}))$ is a quasi-isomorphism.  
\end{proof}

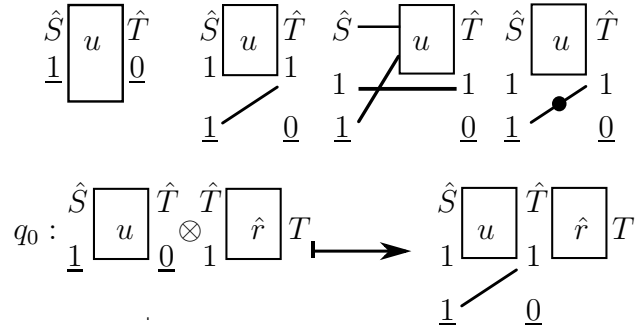
\begin{figure}[ht]
\begin{overpic}
[scale=0.4]{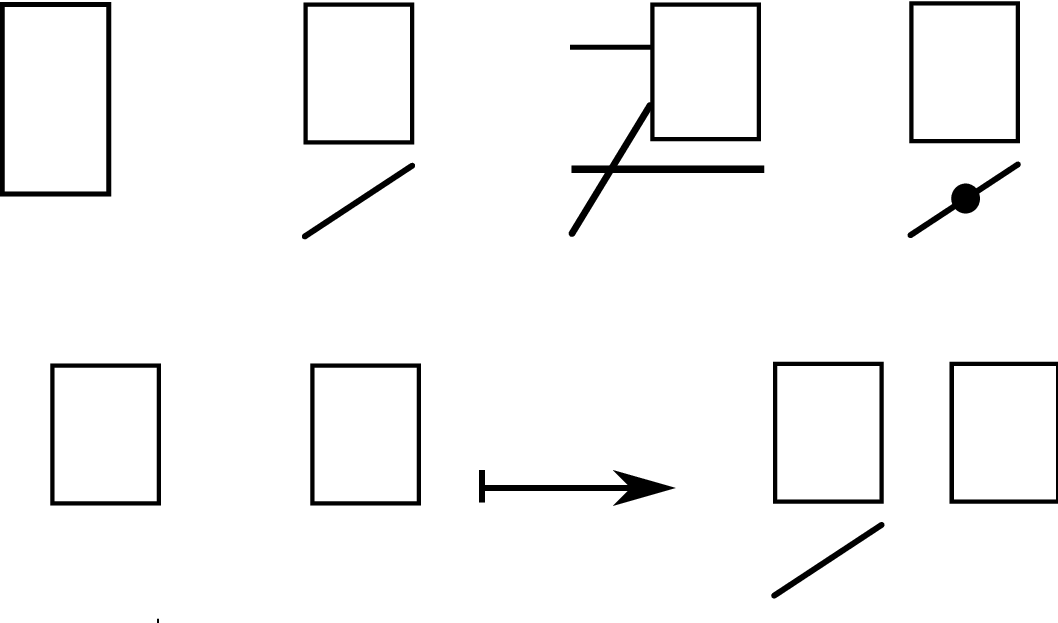}
\put(-4,45){$\underline{1}$}
\put(-4.5,52){$\hat{S}$}
\put(11.5,45){$\underline{0}$}
\put(11,52){$\hat{T}$}
\put(25,34){$\underline{1}$}
\put(25,45){$1$}
\put(24.4,52){$\hat{S}$}
\put(40,34){$\underline{0}$}
\put(40,45){$1$}
\put(40,52){$\hat{T}$}
\put(49.5,34){$\underline{1}$}
\put(49,42){$1$}
\put(49,52){$\hat{S}$}
\put(73,34){$\underline{0}$}
\put(73,42){$1$}
\put(73,52){$\hat{T}$}
\put(81,34){$\underline{1}$}
\put(81,42){$1$}
\put(81,52){$\hat{S}$}
\put(98.5,34){$\underline{0}$}
\put(98.5,42){$1$}
\put(98,52){$\hat{T}$}
\put(3,50){$u$}
\put(32,50){$u$}
\put(64,50){$u$}
\put(90,50){$u$}
\put(69,0){$\underline{1}$}
\put(69,10){$1$}
\put(68.5,20){$\hat{S}$}
\put(85,0){$\underline{0}$}
\put(85,10){$1$}
\put(85,20){$\hat{T}$}
\put(101,15){$T$}
\put(0,10){$\underline{1}$}
\put(0,20){$\hat{S}$}
\put(17,10){$\underline{0}$}
\put(16.5,20){$\hat{T}$}
\put(25,10){$1$}
\put(24.5,20){$\hat{T}$}
\put(41,15){$T$}
\put(20.5,15){$\otimes$}
\put(9,15){$u$}
\put(34,15){$\hat{r}$}
\put(76,15){$u$}
\put(94,15){$\hat{r}$}
\put(-10,15){$q_0:$}
\end{overpic}
\caption{The generators $u, u_0, u_1, u_2$ are on the top line and the map $q_0$ is on the bottom line.}
\label{fig algn4}
\end{figure}

By iteratively using Lemmas \ref{lem e action0} and \ref{lem e action1}, we obtain a complex of projective modules quasi-isomorphic to $\ech(P(S))$ for $S \in \pnk$.
Recall that $e_i(S)=S \cup \{s_i^c\} \in \pnkp$ for $S^c=\{s_1^c, \dots, s_{n-k}^c ~|~ s_1^c < \cdots < s_{n-k}^c\}$.  
There is a distinguished generator $$r_{i,i+1}=(e_i(S), e_{i+1}(S), \phi_{i,i+1}, \es) \in \rnkp,$$ where $\phi_{i,i+1}$ maps $s$ to $s+1$ for $s=s_i^c, \dots, s_{i+1}^c-1$ and fixes the other elements. Its diagrammatic presentation has $s_{i+1}^c-s_i^c$ non-horizontal strands and no crossing. Let $r_{i,i+1}^l$ (resp.\ $r_{i,i+1}^h$) be the generator obtained by adding a dot on the lowest (resp.\ highest) non-horizontal strand in $r_{i,i+1}$, and let $r_{i,i+1}^{lh}$ be a generator by adding a dot on both the lowest and highest non-horizontal strand in $r_{i,i+1}$ such that the dot on the lowest strand is to the left.   
Note that the lowest and highest strands coincide when $s_{i+1}^c-s_i^c=1$, and $r_{i,i+1}^l=r_{i,i+1}^h, r_{i,i+1}^{lh}=0$. 
The gradings are 
\begin{gather*}
\deg(r_{i,i+1})=0, \quad \deg_{q}(r_{i,i+1})=s_{i+1}^c-s_i^c, \\ \deg(r_{i,i+1}^t)=1, \quad \deg_{q}(r_{i,i+1}^t)=s_{i+1}^c-s_i^c-2, \quad \mbox{for $t=l,h$}, \\ \deg(r_{i,i+1}^{lh})=2, \quad \deg_{q}(r_{i,i+1}^{lh})=s_{i+1}^c-s_i^c-4.
\end{gather*}

\begin{thm} \label{thm tensor e}
The right dg $\rnkp$ module $\ech(P(S))$ is quasi-isomorphic to the following complex of projective modules:
$$(\oplus_{1 \le i \le n-k} \left(P(e_i(S))\{l_i(S)+1\} \oplus P(e_i(S))[-1]\{l_i(S)-1\}), d\right),$$
where the differential is $d=\sum_{1 \le i  \le n-k-1}(d_{i,i+1}+d_{i,i+1}^l+d_{i,i+1}^h+d_{i,i+1}^{lh})$ and each summand is given by left multiplication between right projective modules: 
\begin{align*}
d_{i,i+1}: & P(e_{i+1}(S))\{l_{i+1}(S)+1\} \xra{-r_{i,i+1} \cdot} P(e_i(S))[-1]\{l_i(S)-1\}, \\  
d_{i,i+1}^l: & P(e_{i+1}(S))\{l_{i+1}(S)+1\} \xra{r_{i,i+1}^l \cdot} P(e_i(S))\{l_i(S)+1\}, \\  
d_{i,i+1}^h: & P(e_{i+1}(S))[-1]\{l_{i+1}(S)-1\} \xra{-r_{i,i+1}^h \cdot} P(e_i(S))[-1]\{l_i(S)-1\}, \\
d_{i,i+1}^{lh}: & P(e_{i+1}(S))[-1]\{l_{i+1}(S)-1\} \xra{r_{i,i+1}^{lh} \cdot} P(e_i(S))\{l_i(S)+1\}.  
\end{align*}
\end{thm}

\begin{proof}
We prove the theorem for a particular example $S=\{3,4\} \in \cal{P}_{2}(\mathbf{5})$. The case in general is similar.
We have $S=(00110)$ as in Example \ref{ex f} so that $e_1(S)=(10110), e_2(S)=(01110), e_3(S)=(00111)$, and $l_1(S)=0, l_2(S)=-1, l_3(S)=0$. 
Lemma \ref{lem e action0} implies that $\ech(P(00110))$ is quasi-isomorphic to a complex:
\begin{equation} \label{eq e act1}
i_0(\ech(P(0110))[-1]\{-1\}  \xra{g_1} (P(10110)\{1\} \oplus P(10110)[-1]\{-1\}),
\end{equation}
where the two components are in cohomological degrees $-1$ and $0$. 
We omit the grading shifts for simplicity. 
Lemmas \ref{lem e action0} and \ref{lem e action1} implies that there are three short exact sequences:
\begin{gather*}
0 \ra i_0(\ech(P(110))) \xra{g_1'} P(1110)^{\oplus 2} \xra{g_0'} \ech(P(0110)) \ra 0, \\
0 \ra i_1(\ech(P(10))) \xra{q_0} \ech(P(110)) \ra Q \ra 0, \\
0 \ra i_1(\ech(P(0))) \xra{q_0'} \ech(P(10)) \ra Q' \ra 0,
\end{gather*}
where $Q$ and $Q'$ are contractible, and $\ech(P(0)) \cong P(1)^{\oplus 2}$.
Combined with the functors $i_0, i_1$, we have the following commutative diagram:
$$\xymatrix{
& & i_0(\ech(P(0110)) \ar[rr]^{g_1} & & P(10110)^{\oplus 2} \\
i_0^2(\ech(P(110)) \ar[rr]^{i_0(g_1')} & & P(01110)^{\oplus 2} \ar[rr]^{g_1 \circ i_0(g_0')} \ar[u]^{i_0(g_0')} & & P(10110)^{\oplus 2} \ar[u]^{\id} \\
i_0^2\circ i_1^2(P(1)) \ar[rr]^{i_0(g_1') \circ i_0^2(q_0 \circ i_1(q_0'))} \ar[u]^{ i_0^2(q_0 \circ i_1(q_0'))} & & P(01110)^{\oplus 2} \ar[rr]^{g_1 \circ i_0(g_0')} \ar[u]^{\id} & & P(10110)^{\oplus 2}. \ar[u]^{\id} 
}$$ 
Each row is quasi-isomorphic to $\ech(P(00110))$. 

The differential $g_1 \circ i_0(g_0')$ is: 
$$\begin{array}{ccc}
 P(01110)[-1] \oplus P(01110)[-2]\{-2\} & \xra{i_0(g_0')} &  i_0(\ech(P(0110)))[-1]\{-1\} \\
(\mb_{01110}, 0) & \mapsto & r_{1,2} \\
(0, \mb_{01110}) & \mapsto & r_{1,2}^+ \\
i_0(\ech(P(0110)))[-1]\{-1\} & \xra{g_1} & P(10110)\{1\} \oplus P(10110)[-1]\{-1\} \\
r_{1,2} & \mapsto & (r_{1,2}^l, -r_{1,2}) \\
r_{1,2}^+ & \mapsto & (0, -r_{1,2}^+),
\end{array}
$$
where $r_{1,2}$ is the distinguished element of $(10110, 01110)=(e_1(S), e_2(S))$ with a unique non-horizontal strand, and $r_{1,2}^l=r_{1,2}^h=r_{1,2}^+, r_{1,2}^{lh}=0$. 

The differential $i_0(g_1') \circ i_0^2(q_0 \circ i_1(q_0'))$ is: 
$$\begin{array}{ccc}
 P(00111)[-2]\{1\} \oplus P(00111)[-3]\{-1\} & \xra{i_0^2(q_0 \circ i_1(q_0'))} &  i_0^2(\ech(P(110)))[-2]\{-2\} \\
(\mb_{00111}, 0) & \mapsto & r_{2,3} \\
(0, \mb_{00110}) & \mapsto & r_{2,3}^h \\
i_0^2(\ech(P(110)))[-2]\{-2\}  & \xra{i_0(g_1')} & P(01110)[-1] \oplus P(01110)[-2]\{-2\} \\
r_{2,3} & \mapsto & (r_{2,3}^l, -r_{2,3}) \\
r_{2,3}^h & \mapsto & (r_{2,3}^{lh}, -r_{2,3}^h),
\end{array}
$$
where $r_{2,3}$ is the distinguished element of $(01110, 00111)=(e_2(S), e_3(S))$ with three non-horizontal strands and no crossing. 
\end{proof}

\s
\n{\em Proof of Theorem \ref{thm K0 ef}.}
This follows from Proposition \ref{prop ef} and Theorems \ref{thm tensor f} and \ref{thm tensor e}.
\qed
\s

\section{The surface case} \label{sec surface}

The goal of this section is to describe the dga associated to a general bordered surface $S$. We treat the ungraded version corresponding to $\rk$. The ground field is $\mathbb{F}=\mathbb{Z}_2$ in this section.   

\subsection{An annulus} \label{ssec annulus}

We first consider the annulus $S=A$ with $|\tau|=1$ and an arc $a$ connecting two points on the same component of $\bdry A$ and straddling the stop; see Figure~\ref{fig alg7}, where the stop and $\bdry a$ are on the top boundary component and the left and right sides are identified. Let $L$ be the Lagrangian annulus sitting above $a$ and let $L_k$ be the family of Lagrangians above $a_{(k)}$ consisting of $k$ parallel copies of $a$.  

We give a conjectural combinatorial description of the endomorphism algebras $\End(L_k)$.  Its diagrammatic presentation, denoted $R(A,0, a;k)$, is a modification of the strands algebra of an annulus.

\subsubsection{Cases $k=1,2$.}

For $k=1$, there are two intersection points on $A$: the right one is the identity $\mb$, and the left one is the generator $b$ of degree one. (See Figure~\ref{fig alg7}.)

%$\widetilde{R}_1 \cong \widetilde{\Lambda}_1 \cong H^*(T^2)$ with the trivial differential.

For $k=2$, there are eight intersection points on $A$.
They form eight pairs of intersection points in $\End(a_{(2)})$, which in turn can be viewed as elements of $\End(L_2)$:
\begin{align} \label{eq end L2}
\{\circled{4}\circled{7}, \circled{3}\circled{8}; ~~\circled{4}\circled{6}, \circled{3}\circled{5}, \circled{2}\circled{8}, \circled{1}\circled{7}; ~~\circled{2}\circled{5}, \circled{1}\circled{6}\}.
\end{align}

\begin{figure}[ht]
\begin{overpic}
[scale=0.5]{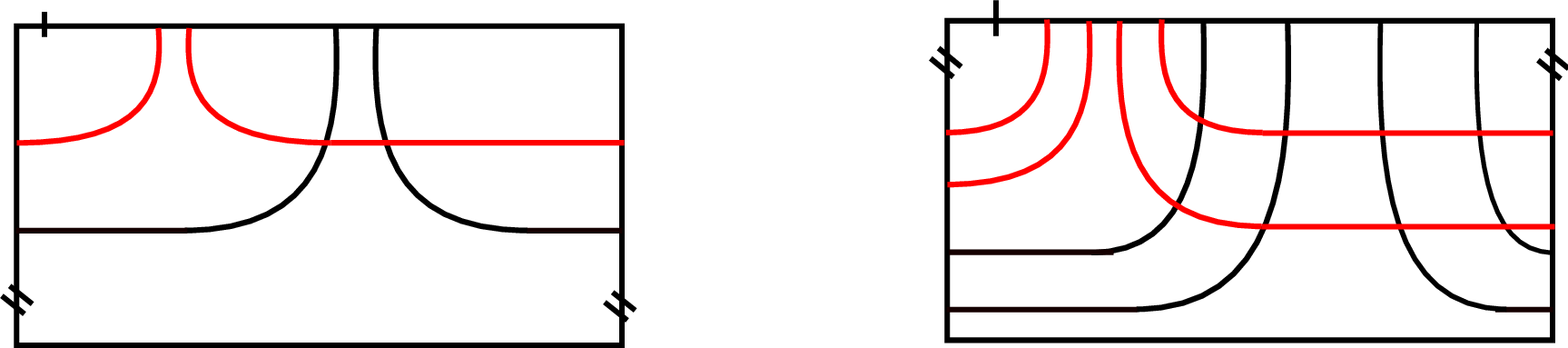}
\put(25,15){$\mb$}
\put(18,15){$b$}
\put(77,15){$\circled{1}$}
\put(82.5,15){$\circled{2}$}
\put(88.5,15){$\circled{3}$}
\put(95,15){$\circled{4}$}
\put(71,8){$\circled{5}$}
\put(81,5){$\circled{6}$}
\put(87,5){$\circled{7}$}
\put(96,9){$\circled{8}$}
\end{overpic}
\caption{Intersection points on $A$ for $\End(L_k)$, where the case $k=1$ is on the left and the case $k=2$ is on the right.}
\label{fig alg7}
\end{figure}
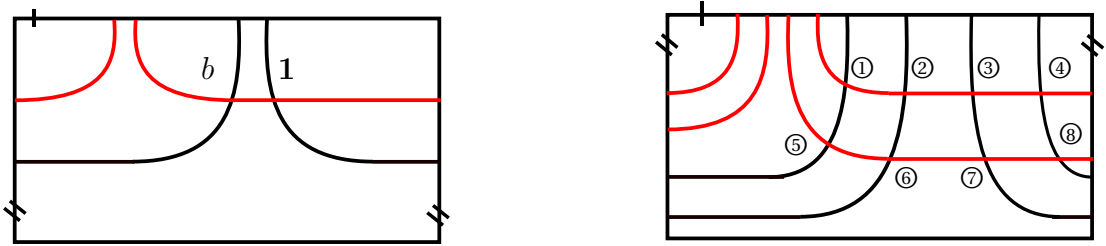

We express these generators in terms of a strands algebra of $A$, together with some extra data. Recall there is an {\em arc diagram} that one associates to a surface in bordered Heegaard Floer homology \cite[Definition 2.1]{Za}. The arc diagram $\cal{Z}$ for $A$ is a vertical interval with two points which are matched together by a half-circle as on the left-hand side of Figure~\ref{fig alg8}. The strands algebra $\mathcal{A}(\mathcal{Z})$ associated to $\cal{Z}$ is a variant of the strands algebra $\cal{A}(n,k)$ that appeared in Section \ref{sec An}, where the main difference is the diagram $\mb\in \mathcal{A}(\mathcal{Z})$ consisting of two dashed horizontal lines starting from a pair of matched points as in the middle of Figure~\ref{fig alg8}. Here $\mb$ is equal to the sum of two diagrams, each with a single horizontal strand, but the diagram consisting of a single horizontal strand is by itself not an element of $\mathcal{A}(\mathcal{Z})$. The strands algebra $\mathcal{A}(\mathcal{Z})$ is two-dimensional with a basis $\mb, b$, where $b$ is the unique increasing strand such that $b^2=0$.
%For $k=2$, no diagram exists since the left (or right) endpoints of any diagram cannot contain both matched points. So the strands algebra is zero.

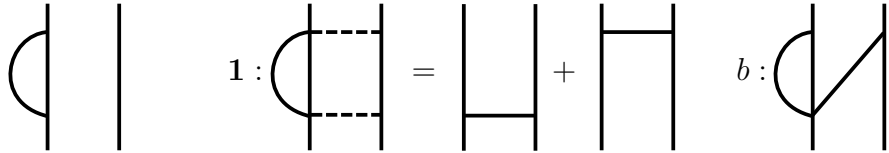
\begin{figure}[ht]
\begin{overpic}
[scale=0.5]{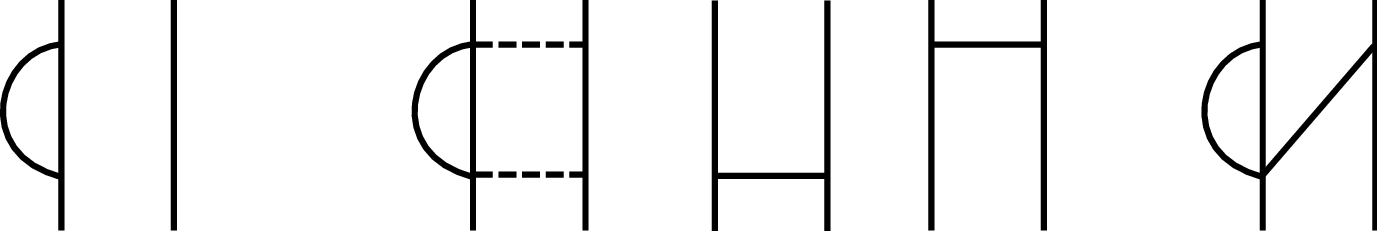}
\put(46,8){$=$}
\put(25,8){$\mb:$}
\put(62,8){$+$}
\put(83,8){$b:$}
\end{overpic}
\caption{The strands algebra $\mathcal{A}(\mathcal{Z})$: the matching half-circle is to the left, the generators $\mb$ and $b$ are in the middle and to the right.}
\label{fig alg8}
\end{figure}

The dga $R(A,0,a;k)$ is obtained from $\mathcal{A}(\mathcal{Z})$ by making the following two changes: (1) adding dots on the strands and (2) forming the disjoint union of $k$ copies of matched points.

For $k=1$, a dotted pair of horizontal strands gives an element $\xi$ which is a sum of two terms as in Figure~\ref{fig alg9}.
\begin{figure}[ht]
\begin{overpic}
[scale=0.4]{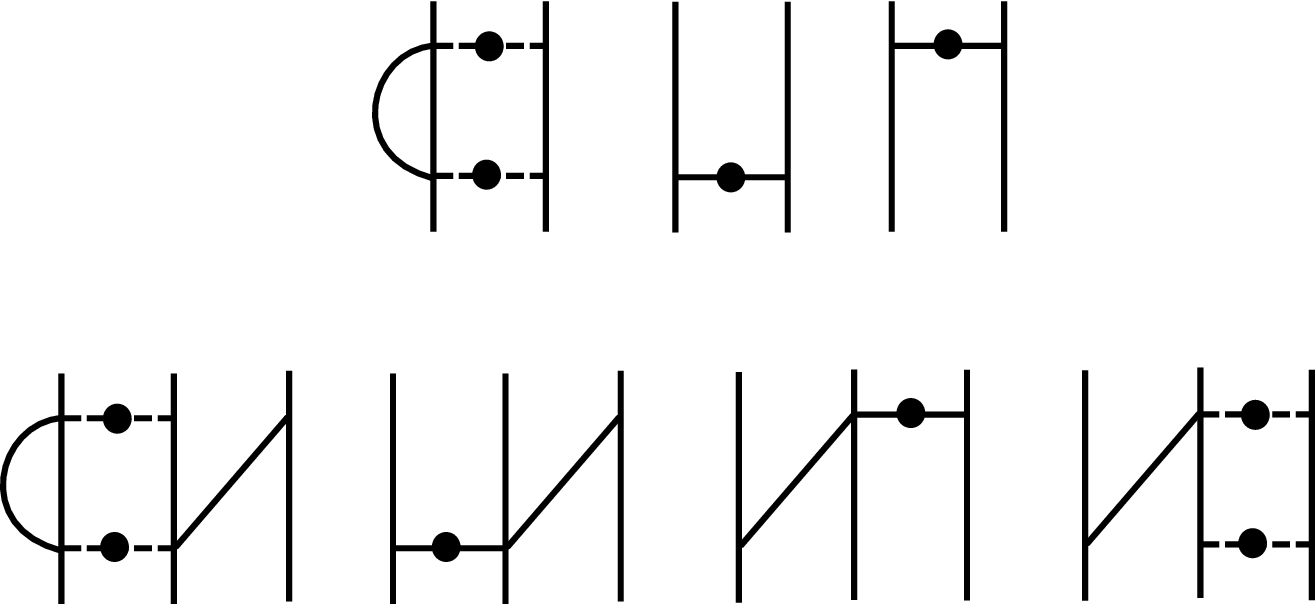}
\put(23,37){$\xi:$}
\put(44,37){$=$}
\put(62,37){$+$}
\put(25,8){$=$}
\put(50,8){$=$}
\put(76,8){$=$}
\end{overpic}
\caption{The dga $R(A,0,a;k=1)$: the generator $\xi$ on the top row and the relation $\xi b=b\xi$ on the bottom row.}
\label{fig alg9}
\end{figure}
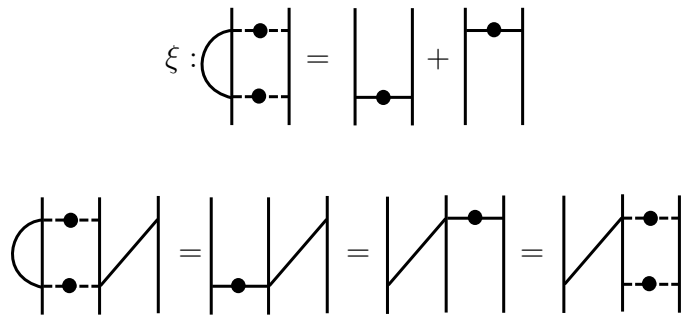
The generator $\xi$ corresponds to the degree one generator in the fiber direction in $\End(L)$; adjoining this generator $\xi$ we obtain:
\begin{align} \label{eq end L}
\End(L) \cong \F[b,\xi]/(b^2=\xi^2=0, b\xi=\xi b) \cong H^*(T^2).
\end{align}

For $k=2$, we form a disjoint union of two copies of matched points; see Figure~\ref{fig alg10}.%\marginpar{\cg Yin, in the picture the lower one looks negative; should we fix this?}
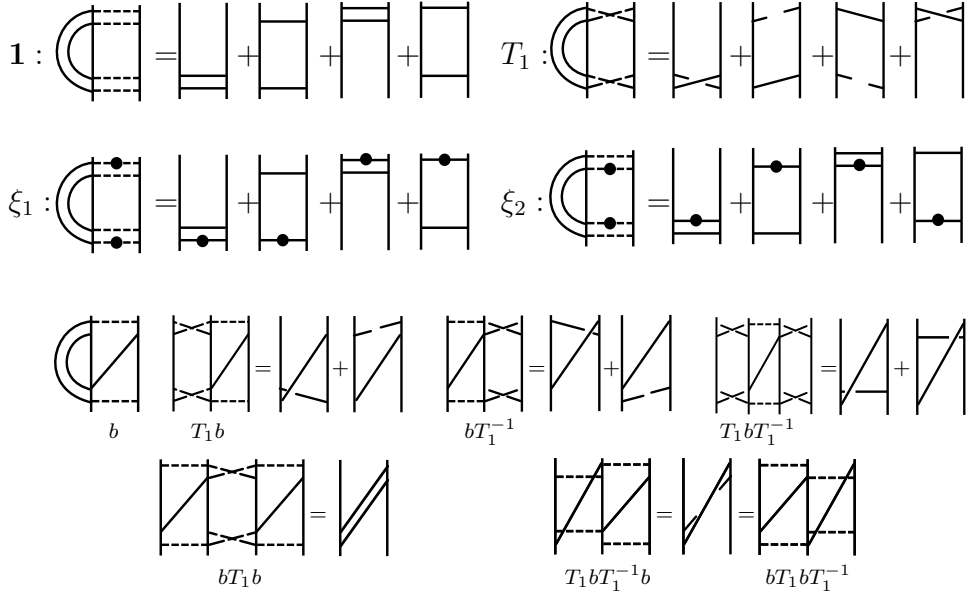
\begin{figure}[ht]
\begin{overpic}
[scale=0.35]{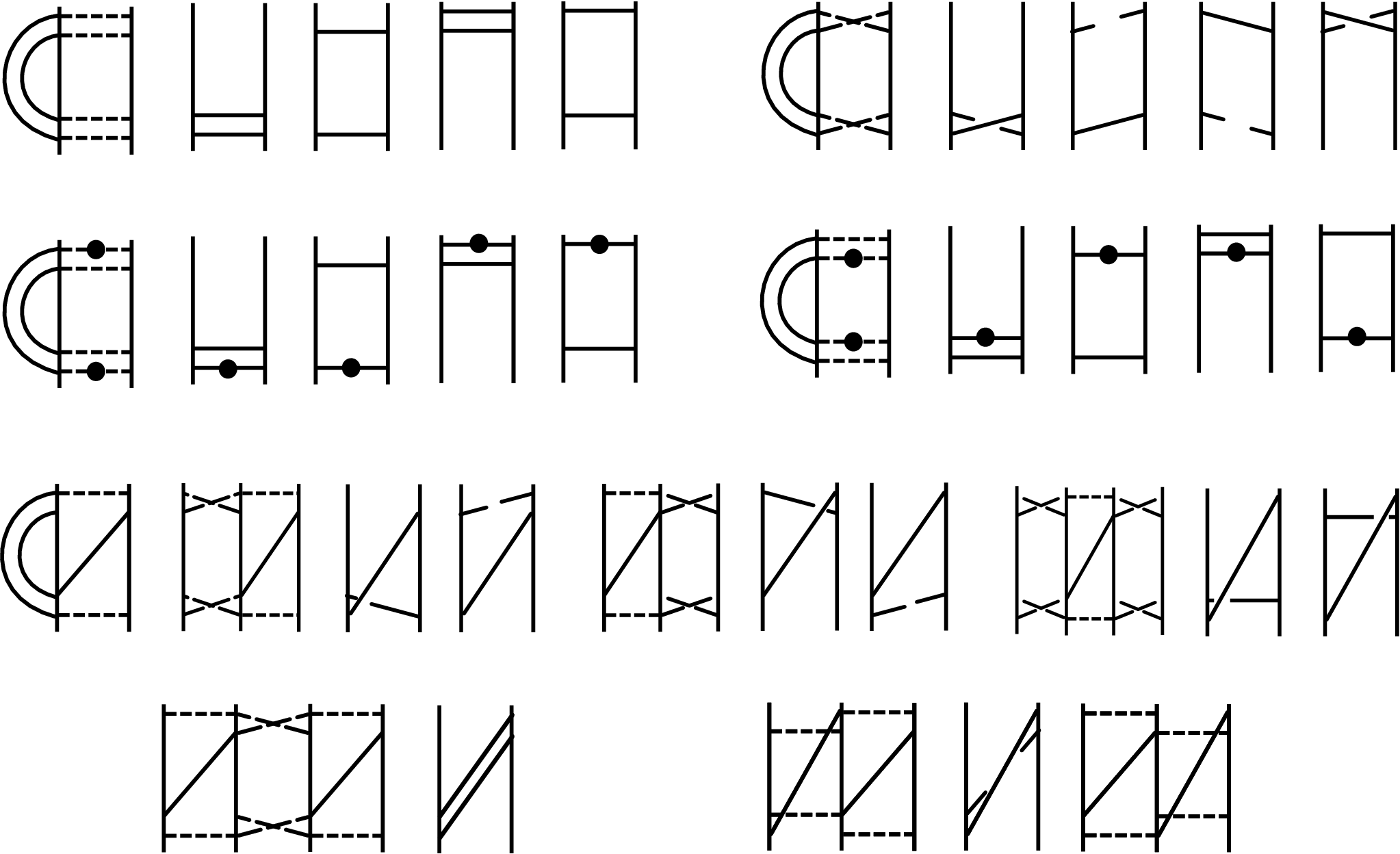}
\put(-5,54){$\mb:$}
\put(11,54){$=$}
\put(20,54){$+$}
\put(29,54){$+$}
\put(37.5,54){$+$}
\put(65,54){$=$}
\put(74,54){$+$}
\put(83,54){$+$}
\put(91.5,54){$+$}
\put(49,54){$T_1:$}
\put(-5,38){$\xi_1:$}
\put(11,38){$=$}
\put(20,38){$+$}
\put(29,38){$+$}
\put(37.5,38){$+$}
\put(49,38){$\xi_2:$}
\put(65,38){$=$}
\put(74,38){$+$}
\put(83,38){$+$}
\put(91.5,38){$+$}
\put(6,13){${\scriptstyle b}$}
\put(15,13){${\scriptstyle T_1b}$}
\put(22,20){${\scriptstyle =}$}
\put(30.5,20){${\scriptstyle +}$}
\put(52,20){${\scriptstyle =}$}
\put(60.2,20){${\scriptstyle +}$}
\put(45,13){${\scriptstyle bT_1^{-1}}$}
\put(84,20){${\scriptstyle =}$}
\put(92,20){${\scriptstyle +}$}
\put(73,13){${\scriptstyle T_1bT_1^{-1}}$}
\put(28,4){${\scriptstyle =}$}
\put(66,4){${\scriptstyle =}$}
\put(75,4){${\scriptstyle =}$}
\put(18,-3){${\scriptstyle bT_1b}$}
\put(56,-3){${\scriptstyle T_1bT_1^{-1}b}$}
\put(78,-3){${\scriptstyle bT_1bT_1^{-1}}$}
\end{overpic}
\caption{The dga $R(A,0,a;k=2)$.}
\label{fig alg10}
\end{figure}
The unit $\mb$ is a sum of four terms, and restricts to the identity for both pairs. A new generator $T_1$ is also a sum of four terms, each of which is a braid-like crossing (the convention is that the lower crossing of $T_1$ in Figure~\ref{fig alg10} is positive).  The generator $T_1$ is invertible: $T_1^{-1}=T_1-\hbar \mb$. The sum of diagrams for $T_1^{-1}$ is that of $T_1$ with the crossings reversed. The two generators $\mb$ and $T_1$ correspond to $\circled{4} \circled{7}$ and $\circled{3} \circled{8}$ in $\End(L_2)$, respectively.  Diagrammatically, some strands in $T_1$ are decreasing, but the slope can be made arbitrarily close to zero if the two matching half-circles are made sufficiently close. These are the only types of decreasing strands that are allowed in $R(A,0,a;k)$. 

The two generators $\xi_1$ and $\xi_2$ correspond to two generators of degree one in the fiber direction in $\End(L_2)$.
The generator $b$ corresponds to $\circled{4} \circled{6}$ in $\End(L_2)$ whose differential is zero.
This is compatible with the fact that the diagram $b$ contains no crossing.

The eight generators in (\ref{eq end L2}) then correspond to:
$$\{\mb, T_1; ~~b, T_1b, bT_1^{-1}, T_1bT_1^{-1}; ~~bT_1b, bT_1bT_1^{-1}\}.$$

\subsubsection{The dga $R(A,0,a;k)$ for $k\geq 1$.}

%Motivated by the case of $k=2$, we give an algebraic description of $\End(L_k)$ for $k$ in general. 

\begin{defn} \label{def tRk}
The dga $R(A,0,a;k)$, $k\geq 1$, has generators $b, T_i, \xi_j$ for $1 \le i \le k-1,  1 \le j \le k$, and relations:
\begin{gather*}
T_i^2=1+\hbar T_i, \quad T_iT_{i+1}T_i=T_{i+1}T_iT_{i+1}, \quad  T_iT_{i'}=T_{i'}T_i ~\mbox{for}~|i-i'|>1, \\
\nonumber \xi_j^2=0, \quad  \xi_j\xi_{j'}=-\xi_{j'}\xi_j, ~\mbox{for}~j \neq j',   \\
\nonumber \xi_iT_i=T_i\xi_{i+1},  \quad \xi_{i+1}T_i=T_i\xi_{i}-\hbar(\xi_{i}-\xi_{i+1}), \quad  \xi_jT_i=T_i\xi_j ~\mbox{for}~j \neq i,i+1, \\
b^2=0,\\
b\xi_j=\xi_jb, \quad 1 \le j \le k,\\
bT_i=T_ib, \quad 1 \le i \le k-2,\\
bT_{k-1}bT_{k-1}^{-1}=T_{k-1}bT_{k-1}^{-1}b.
\end{gather*}
The grading is $\deg(\xi_j)=\deg(b)=1$ and $\deg(T_i)=0$.
The differential on the generators is
$$d(\mb_k)=0, \qquad d(T_i)=\xi_i-\xi_{i+1}, \qquad d(\xi_j)=0, \qquad d(b)=0,$$ 
and is extended using the Leibniz rule: $d(ab)=d(a)b+a~d(b)$.
\end{defn}

To show the well-definition of $R(A,0,a;k)$, we need to verify that the differential preserves the relations. For example for the last relation:
\begin{align*}
d(bT_{k-1}bT_{k-1}^{-1})=b(\xi_{k-1}-\xi_k)bT_{k-1}^{-1}+bT_{k-1}b(\xi_{k-1}-\xi_k)=bT_{k-1}(\xi_{k-1}-\xi_k)b, \\
d(T_{k-1}bT_{k-1}^{-1}b)=(\xi_{k-1}-\xi_k)bT_{k-1}^{-1}b+T_{k-1}b(\xi_{k-1}-\xi_k)b=b(\xi_{k-1}-\xi_k)T_{k-1}^{-1}b.
\end{align*}
The two lines are the same since $T_{k-1}(\xi_{k-1}-\xi_k)=(\xi_{k-1}-\xi_k)T_{k-1}^{-1}$.
The proof for the other relations is straightforward.

%As a graded algebra, $\trk$ has the following simple description.
%Let $\tlk$ be the exterior algebra of $2k$ generators:
%$$\tlk=\F[\xi_1, \dots, \xi_k; b_1, \dots, b_k] / (a^2=0, ab=-ba ~\mbox{for}~ a,b=\xi_i, b_j),$$
%and $\deg(\xi_i)=\deg(b_j)=1$.
%Let $\trk'=\F[S_k] \rtimes \tlk$ denote the semi-product.
%There are two maps of graded algebras:
%$$\begin{array}{cccccccl}
%\trk & \ra & \trk' & && \trk' & \ra & \trk \\
%s_i & \mapsto & s_i, && &  s_i & \mapsto & s_i,\\
%\xi_j & \mapsto & \xi_j, & && \xi_j & \mapsto & \xi_j,\\
%b & \mapsto & b_k, & && b_j & \mapsto & s_{j}\cdots s_{k-1}b_ks_{k-1} \cdots s_{j}.
%\end{array}$$
%It is easy to verify that the two maps are well-defined and they are inverse to each other.
%
%\begin{lemma} \label{lemma tRk}
%The graded algebras $\trk$ and $\trk'$ are isomorphic. In particular, $\dim \trk=k!2^{2k}$.
%\end{lemma}
%
%
%
%\begin{rmk}
%Under the isomorphism $\trk \cong \trk'$, the differential of the generators $b_j$ is highly nontrivial.
%\end{rmk}

%\begin{prop} \label{prop H(tRk)} The cohomology algebra $H(\trk)$ is isomorphic to $\Lambda_{2k}$. \end{prop}

%\begin{conj} \label{conj tRk formal} The dga $\trk$ is not formal for $k \ge 2$. \end{conj}

\subsection{Parametrized surfaces} \label{ssec surface}

Let $(S,n,{\bf a})$ be a parametrized bordered surface with a basis of arcs ${\bf a}=\{a_1,\dots,a_n,a_{n+1},\dots,a_{n+s}\}$. The following slightly generalizes the definition of an arc diagram from \cite[Definition 2.1]{Za}, which was defined when there are no singular points.

\begin{defn} \label{def arc diagram}
An {\em arc diagram} $\cal{Z} = (Z,{\bf \ti{a}},M)$ is a triple consisting of:
\be
\item a disjoint union $Z = Z_1\sqcup \dots \sqcup Z_l$ of oriented line segments, 
\item a collection ${\bf \ti{a}} = \{\ti{a}_1, \dots, \ti{a}_n , \ti{a}_{n+1}, \dots, \ti{a}_{n+2s}\}$ of distinct points in $Z$, and 
\item a partial matching, i.e., a 2-to-1 function $M: \{\ti{a}_{n+1}, \dots, \ti{a}_{n+2s}\} \ra \{1, \dots , s\}$.
\ee
The points of ${\bf \ti{a}}$ are ordered lexicographically by the ordering of the components of $Z$, followed by the orientation of each $Z_i$. For simplicity we assume that the points $\{\ti{a}_1, \dots, \ti{a}_n , \ti{a}_{n+1}\}$ lie on the first component $Z_1$.
\end{defn}

Starting with the disjoint union $Z \times [0, 1]$ of rectangles, we attach $1$-handles along $M^{-1}(j) \times \{0\}$ for $j = 1,\dots, s$. This yields a surface $S$ with a basis of arcs ${\bf a}=\{a_1,\dots,a_n,a_{n+1},\dots,a_{n+s}\}$, where the arc $a_{n+j}$ corresponds to the core of the $1$-handle attached along $M^{-1}(j) \times \{0\}$ and the half-arc $a_{i}$ corresponds to the point $\ti{a}_i\in Z$ for $i=1,\dots,n$. Each parametrized surface corresponds to a unique arc diagram.

%For $\cal{Z} = (Z,{\bf \ti{a}},M)$, let $\op{St}(\cal{Z})=\{t \in \{1,\dots,n+2s\} ~|~ \ti{a}_{t}, \ti{a}_{t+1} ~\mbox{in the same line segment}\}.$ It indexes all the shortest increasing strand.
Let $\op{St}(a_{n+j}, a_{n+j'})$ be the set of primitive increasing strands from one point in $M^{-1}(j) \times \{0\}$ to another point in $M^{-1}(j') \times \{1\}$ in some rectangle $Z_i \times [0, 1]$, where a strand is {\em primitive} if it is not a concatenation of two increasing strands.
Note that $\#\op{St}(a_{n+j}, a_{n+j'})=0,1,2$ depending on their relative positions.

We define the grading on ${\bf \ti{a}} = \{\ti{a}_1, \dots, \ti{a}_n , \ti{a}_{n+1}, \dots, \ti{a}_{n+2s}\}$ as follows: $\deg(\ti{a}_i)=1$ if $\{\ti{a}_i, \ti{a}_{i'}\}=M^{-1}(j)$ and $i>i'$ and $\deg(\ti{a}_i)=0$ otherwise. In words, for each matched pair the point with the larger subscript has grading one and the other points have grading zero.

\vspace{.2cm}
We will construct a family of dgas $\rsnk$ for $k \ge 0$ in two steps.
It is a modification of the strands algebra associated to the arc diagram $\cal{Z} = (Z,{\bf \ti{a}},M)$.

\vspace{.2cm}

\n {\bf Step 1.} $k=0,1$.
The dga $\rsnk$ for $k=0$ is the ground field $\F$.

For $k=1$, the dga $\rsno$ has four types of generators:
\be
\item an idempotent $\mb(i)$ for $i=1,\dots,n+s$, represented by one horizontal strand from $\ti{a}_i$ to itself for $i=1,\dots,n$, or the sum of two diagrams, each with a single horizontal strand from $M^{-1}(i-n)$ to itself for $i=n+1,\dots,n+s$;
\item a generator $\xi(n+j)$ for $j=1,\dots,s$, represented by a difference of two diagrams, each with a single dotted horizontal strand from $M^{-1}(j)$ to itself;
\item a generator $b_p \in \op{St}(a_{n+j}, a_{n+j'})$, represented by a primitive increasing strand from $M^{-1}(j)$ to $M^{-1}(j')$;
\item a generator $r \in R(n+1;1)$, represented by an increasing strand (not necessarily primitive) possibly with a dot from one point in $\{\ti{a}_1, \dots, \ti{a}_n , \ti{a}_{n+1}\}$ to another.
\ee
The multiplication is given by the concatenation of the strands diagrams. The dots can freely move along a strand, and a strand with more than one dot is set to be zero.  The grading on the generators is given by
$$\deg(\mb_i)=\deg(\mb_j)=0, \deg(\xi_j)=\deg(\ti{b}_i)=1, \deg(b_p)=\deg(\ti{a}_{i+1})-\deg(\ti{a}_{i}),$$
if $b_p$ connects $\ti{a}_{i}$ to $\ti{a}_{i+1}$. Some generators of the first three types were given in Figures \ref{fig alg8} and \ref{fig alg9}.  The differential is trivial. 

\begin{example} \label{ex rsnk}
Consider a parametrized once-punctured torus $(S,{\bf a})$ with $n=2$ singular points.
The associated arc diagram is $\cal{Z} = (Z,{\bf \ti{a}},M)$, where $Z=Z_1$ contains ${\bf \ti{a}} = \{\ti{a}_1, \dots, \ti{a}_{6}\}$ and the
partial matching $M: \{\ti{a}_3, \dots, \ti{a}_{6}\} \ra \{1,2\}$ is given by $M^{-1}(1)=\{\ti{a}_3, \ti{a}_{5}\}, M^{-1}(2)=\{\ti{a}_4, \ti{a}_{6}\}$ which corresponds to the arcs $a_3, a_4$, respectively. 
See Figure~\ref{fig alg12}.

There are four types of generators of $\rsno$: (1) idempotents $\mb(i)$ for $1\le i \le 4$; (2) two generators $\xi(3), \xi(4)$ of the second type; (3) three generators $b_{p_1}, b_{p_2}, b_{p_3}$ of the third type such that $\op{St}(a_3, a_4)=\{b_{p_1}, b_{p_3}\}, \op{St}(a_4, a_3)=\{b_{p_2}\}$; and (4) a generator $r \in R(3;1)$ of the fourth type.
The gradings of $\xi(3)$, $\xi(4)$, and $b_{p_2}$ are $1$ and the gradings of the other generators are zero.
\end{example}

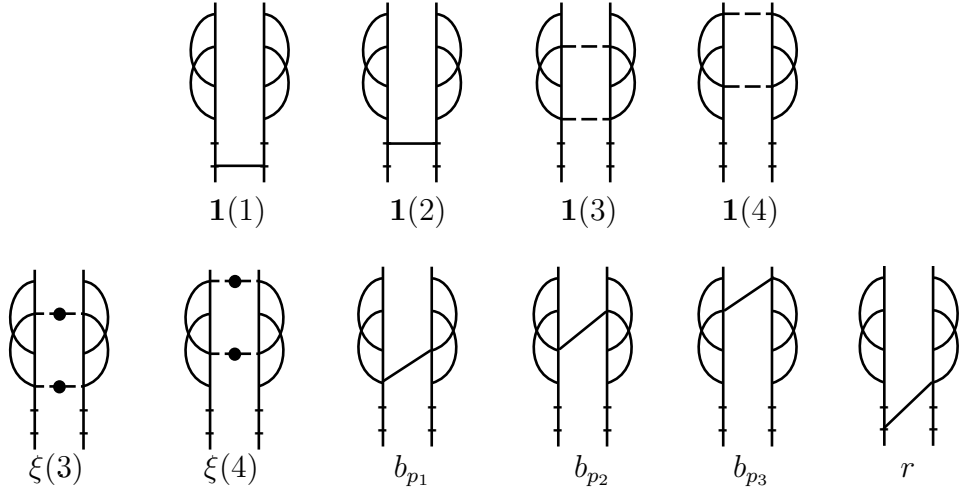
\begin{figure}[ht]
\begin{overpic}
[scale=0.35]{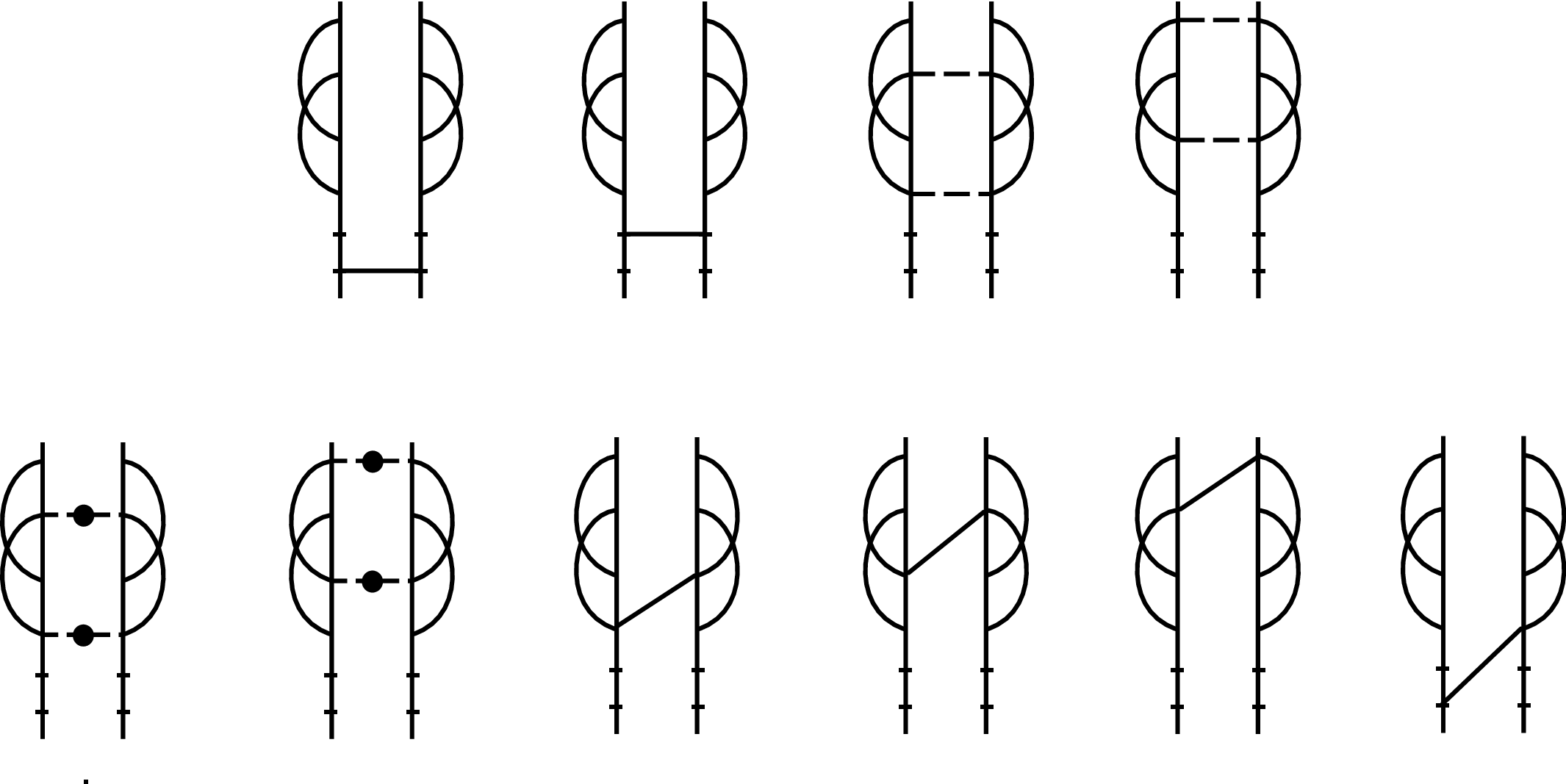}
\put(21,27){$\mb(1)$}
\put(40,27){$\mb(2)$}
\put(58,27){$\mb(3)$}
\put(75,27){$\mb(4)$}
\put(2.1,0){$\xi(3)$}
\put(20.5,0){$\xi(4)$}
\put(40.5,0){$b_{p_1}$}
\put(59.5,0){$b_{p_2}$}
\put(76.2,0){$b_{p_3}$}
\put(93.85,0){$r$}
\end{overpic}
\caption{Generators of $R(S,n=2, {\bf a}; k=1)$ when $S$ is a once-punctured torus.}
\label{fig alg12}
\end{figure}

\s\n {\bf Step 2.} $k \ge 2$.
We define
$$S^k({\bf a})=\{a_{\bf i}=a_1^{i_1}\cdots a_{n+s}^{i_{n+s}}~|~ {\bf i}=(i_1,\dots,i_{n+s}), |{\bf i}|=i_1+\dots+i_{n+s}=k\}.$$
We are using multiplicative notation for $a_{\bf i}$ and e.g., $a_{\bf i}/a_j:= a_1^{i_1}\cdots a_j^{i_j-1}  \cdots a_{n+s}^{i_{n+s}}$, where $a_j^{i_j-1}=0$ if $i_j-1<0$.
Each element of $S^k({\bf a})$ corresponds to a $k$-tuple of arcs.

The dga $\rsnk$ has five types of generators:
\be
\item an idempotent $\mb({\bf i})$ for $a_{\bf i} \in S^k({\bf a})$, represented by the sum of all horizontal strands from $a_{\bf i}$ to itself (by this we mean there is a single horizontal strand from each copy of $a_i$ to itself);
\item a generator $T({\bf i}, n+j)_t$ for $1 \le t \le i_{n+j}-1, 1 \le j \le s$, represented by the sum of braid-like diagrams from $a_{\bf i}$ to itself which exchange the $t$th and $(t+1)$st copies of $a_{n+j}$ and have a single horizontal strand from each remaining copy of $a_i$ to itself;
\item a generator $\xi({\bf i}, n+j)_t$ for $1 \le t \le i_{n+j}, 1 \le j \le s$, represented by a signed sum of horizontal strands from $a_{\bf i}$ to itself %(where there is a single horizontal strand from each copy of $a_i$ to itself) 
with a dot on the $t$th copy of $a_{n+j}$;
\item a generator $b({\bf i}, {\bf i'})_p$ for $b_p \in \op{St}(a_{n+j}, a_{n+j'})$, $a_{\bf i}/ a_{n+j}=a_{\bf i'}/a_{n+j'}\not=0$, represented by the sum of diagrams consisting of the shortest increasing strand of type $b_p$ from $a_{n+j}^{i_{n+j}}$ to $a_{n+j'}^{i_{n+j'}}$, together with horizontal strands from $a_{\bf i}/a_{n+j}$ to $a_{\bf i'}/a_{n+j}$. %where there is a single horizontal strand from each copy of $a_i$ to itself and the diagram does not have any intersection point;
\item a generator $r({\bf i}, {\bf i'})$ for $r \in R(n+1;k'), k' \le k$, represented by the sum of diagrams of $k$ strands from $a_{\bf i}$ to $a_{\bf i'}$, whose lower $k'$ strands connect the points in $\{\ti{a}_1, \dots, \ti{a}_n , \ti{a}_{n+1}\}$ and give the element $r \in R(n+1;k')$, and whose upper $k-k'$ strands are horizontal and do not intersect the lower $k'$ strands.
\ee
Some generators of the first four types are drawn in Figure~\ref{fig alg10} for $a_{\bf i}=a_{\bf i'}=a_{n+1}^{2}$.
%Endpoints of such generators corresponds to the arcs $\{a_{n+1}, \dots, a_{n+s}\}$ only.
%Generators of the fifth type for $R(3;2)$ are drawn in Figure~\ref{fig alg6}.

The multiplication is given by concatenating the diagrams, isotoping, and applying some relations. Any diagram can be decomposed into a product of the generators. The relations are as follows: %\marginpar{\cg Yin, I don't understand the top row of Fig 25; maybe add more description to item (4)?}
\be
\item disjoint diagrams supercommute; 
\item a dot can move freely along a strand and slide through a crossing from above; 
\item the generators of the type $T({\bf i},n+j)$ satisfy the Hecke algebra (= HOMFLY skein) relations; 
\item there are local relations involving generators of the second and fourth types as in Figure \ref{fig alg11}; and
\item a diagram containing a strand with more than one dot is set to be zero.
\ee

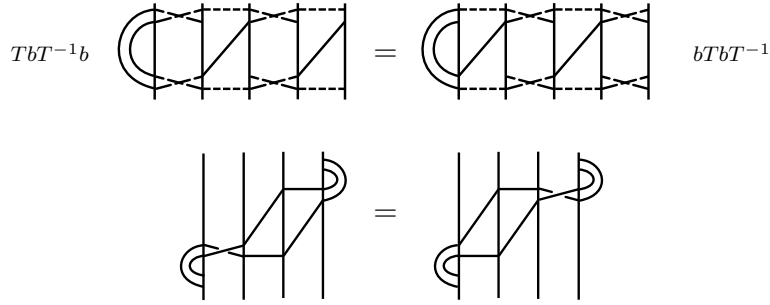
\begin{figure}[ht]
\begin{overpic}
[scale=0.35]{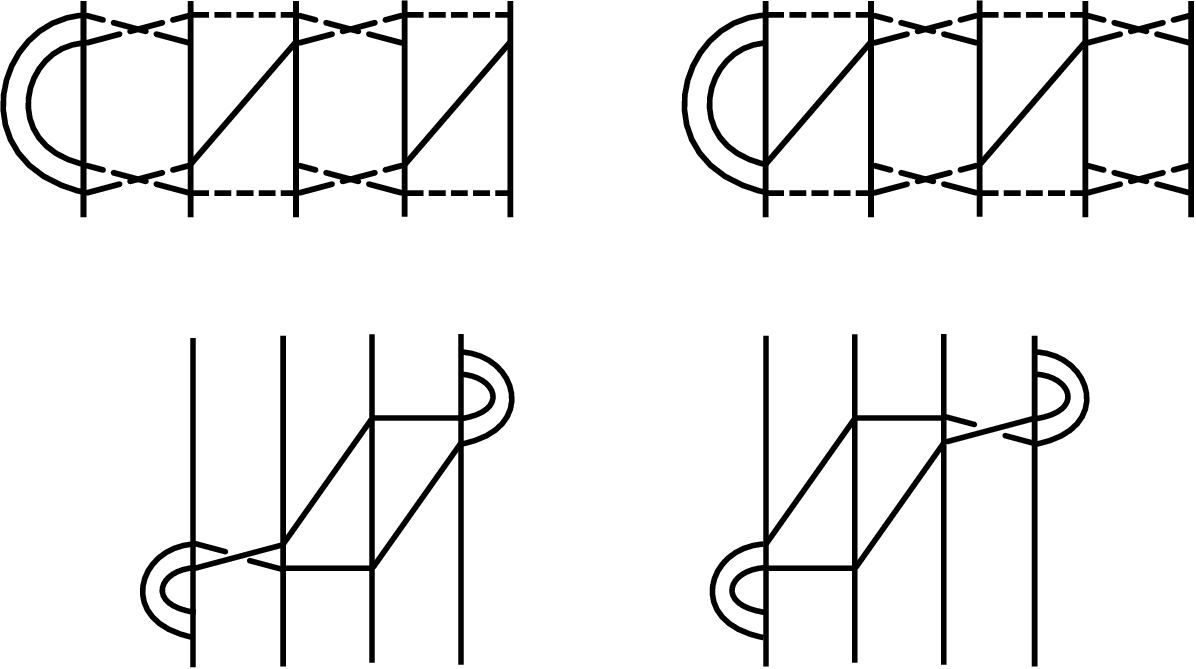}
\put(48,45){$=$}
\put(48,15){$=$}
\put(-20,45){${\scriptstyle TbT^{-1}b}$}
\put(108,45){${\scriptstyle bTbT^{-1}}$}
\end{overpic}
\caption{Local relations involving $b({\bf i}, {\bf i'})_p$, ${\bf i}={\bf i'}$ in the first case and ${\bf i} \neq {\bf i'}$ in the second case. Both sides of each relation are isotopic to a diagram with a single crossing after concatenation.}
\label{fig alg11}
\end{figure}

The grading is defined on the generators as follows: 
\begin{gather*}
\deg(\mb({\bf i}))=\deg(T({\bf i}, n+j)_t)=0, \quad \deg(\xi({\bf i}, n+j)_t)=1,\\
\deg(b({\bf i}, {\bf i'})_p)=\deg(b_p),\quad \deg(r)=\#\{\mbox{dots on the diagram $r$}\}.
\end{gather*}
The differential is defined on the generators as follows:
$$d(T({\bf i}, n+j)_t)=\xi({\bf i}, n+j)_t-\xi({\bf i}, n+j)_{t+1},$$ 
$d(r)$ defined as in $R(n+1;k)$, and the differentials of the remaining generators are zero. By a routine verification the differential preserves all the relations and hence the dga $\rsnk$ is well-defined.

Given two $k$-tuples of arcs $a_{\bf i}, a_{\bf i'} \in S^k({\bf a})$, the vector space $\mb(a_{\bf i}) \cdot \rsnk \cdot \mb(a_{\bf i'})$ has a linear basis of diagrams with $k$ strands from $a_{\bf i}$ to $a_{\bf i'}$ up to isotopy.
In particular, $\rsnk$ is finite-dimensional.

\s\n{\em Example \ref{ex rsnk}, continued.}
For $k=2$, $S^2({\bf a})=\{a_1a_2, a_1a_3, a_1a_4, a_2a_3, a_2a_4, a_3^2, a_3a_4, a_4^2\}$. Hence there are eight idempotents. Some generators of the fourth or fifth type are given in Figure~\ref{fig alg13}.
\begin{figure}[ht]
\begin{overpic}
[scale=0.4]{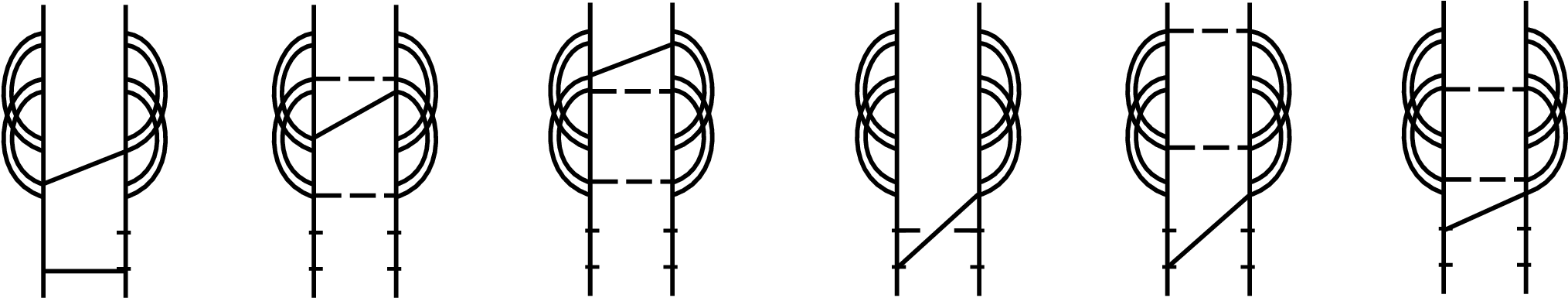}
\put(-1,-3.5){${\scriptstyle b(a_1a_3,a_1a_4)_{p_1}}$}
\put(17.2,-3.5){${\scriptstyle b(a_3a_4,a_3^2)_{p_2}}$}
\put(34.8,-3.5){${\scriptstyle b(a_3^2,a_3a_4)_{p_3}}$}
\put(54,-3.5){${\scriptstyle r(a_1a_2,a_2a_3)}$}
\put(71.6,-3.5){${\scriptstyle r(a_1a_4,a_3a_4)}$}
\put(90,-3.5){${\scriptstyle r(a_2a_3,a_3^2)}$}
\end{overpic}
\vskip.2in
\caption{Generators of $R(S,n=2, {\bf a};k=2)$, where $S$ is a once-punctured torus.}
\label{fig alg13}
\end{figure}
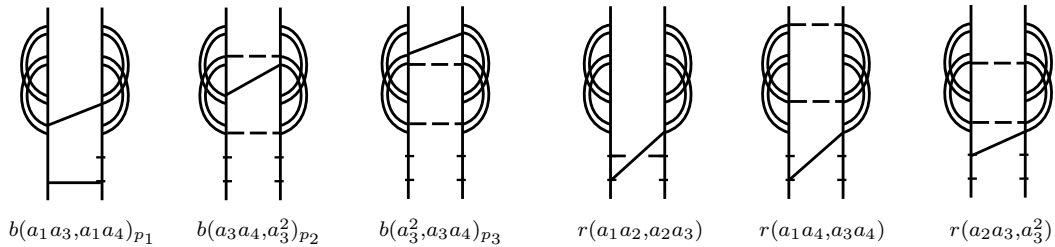

\begin{conj} \label{conj rnsk} $\mbox{}$
\be
\item The higher $\ai$-operations $\mu_m$ for $(S,n,{\bf a})$ are trivial for $m \geq 3$.
Moreover, the resulting dga is isomorphic to $\rsnk$.
\item The quasi-isomorphic class of $\rsnk$ is independent of choices of parametrization ${\bf a}$ of the surface $S$.
\ee
\end{conj}

\end{document}